\documentclass{amsart}
\usepackage{a4wide}

\usepackage[sort,noadjust,compress]{cite}

\usepackage{tikz} 
\usepackage{tkz-euclide}
\usetkzobj{all}
\usepackage{tikz-cd}
\usepackage{tkz-graph}
\usepackage{tkz-berge}

\DeclareFontEncoding{LS1}{}{}
\DeclareFontSubstitution{LS1}{stix}{m}{n}
\DeclareMathAlphabet\stixscr{LS1}{stixscr}{m}{n}
\usepackage{mathrsfs}
\usepackage{mathtools} 
\usepackage{amssymb} 

\usepackage{enumitem}
\setlist{leftmargin=*}

\usepackage{hyperref}

\tikzcdset{every diagram/.append style = {sep=1.5em}}

\newcommand{\im}{\operatorname{im}}
\newcommand{\Cl}{\operatorname{Cl}}
\newcommand{\Aut}{\operatorname{Aut}}

\newtheorem{theorem}{Theorem}[section] 
\newtheorem{proposition}[theorem]{Proposition}
\newtheorem{lemma}[theorem]{Lemma}
\newtheorem{corollary}[theorem]{Corollary} 
\theoremstyle{definition}
\newtheorem{definition}[theorem]{Definition}
\newtheorem{example}[theorem]{Example}
\newtheorem*{construction}{Construction}

\theoremstyle{remark} 
\newtheorem{remark}[theorem]{Remark}
\newtheorem{remarks}[theorem]{Remarks}

\newcommand{\injto}{\hookrightarrow}
\newcommand{\epito}{\twoheadrightarrow}
\newcommand{\bbT}{\mathbb{T}}
\newcommand{\cM}{\mathcal{M}}
\newcommand{\cT}{\mathcal{T}}
\newcommand{\pls}{\operatorname{PLS}}
\newcommand{\GQ}{\operatorname{GQ}}
\newcommand{\PQ}{\operatorname{PQ}}
\newcommand{\restr}{\mathord{\upharpoonright}}
\newcommand{\lseq}{\prescript{}{=}{}}
\newcommand{\kk}{k}

\def\cf{cf.\kern.3em}
\def\eg{e.g.\kern.3em}
\def\ie{i.e.\kern.3em}
\newcommand{\arXiv}[1]{\href{http://arxiv.org/abs/#1}{\texttt{arXiv\string:\allowbreak#1}}}

\title{On highly regular strongly regular graphs} 

\author[Ch.~Pech]{Christian Pech}
\address{ORCID ID: 0000-0002-5018-013X}
\email{cpech@freenet.de}
\urladdr{https://www.researchgate.net/profile/Christian\_Pech2}


\begin{document}

\begin{abstract}
	In this paper we unify several existing regularity conditions for graphs, including strong regularity, $\kk$-isoregularity, and the $t$-vertex condition. We develop an algebraic composition/decomposition theory of regularity conditions.

	Using our theoretical results we show that a family of non rank 3 graphs known to satisfy the $7$-vertex condition fulfills an even stronger  condition, $(3,7)$-regularity (the notion is defined in the text). 
  
	Derived from this family we obtain a new infinite family of non rank $3$ strongly regular graphs satisfying the $6$-vertex condition. This strengthens and generalizes previous results by Reichard.
\end{abstract}
 
\maketitle

\section{Introduction}
Strongly regular graphs (srgs) are simple regular graphs with the property that the number of common neighbors of a pair of distinct vertices depends only on whether the two vertices are connected by an edge or not. Originally introduced by R.~C.~Bose in \cite{Bos63}, they are one of the central notions of modern algebraic graph theory. Small examples include the pentagon, the Petersen graph, triangular graphs, the Clebsch graph,\dots (A.~E.~Brouwer maintains a list of known small examples at \cite{AEBList}). Srgs arise, \eg as  orbital graphs of permutation groups of rank three (such srgs are usually called \emph{rank $3$ graphs} or \emph{$2$-homogeneous graphs}). Thanks to the classification of finite simple groups, all rank $3$ graphs are known by now (\cf \cite{Ban72,KanLie82,LieSax86}). However, by no means all srgs arise in this way. In fact strongly regular graphs exist in such an abundance that nowadays a complete classification up to isomorphism seems hopeless (\cf \cite{Wal71,FDF02, Muz07}). In order to single out the more interesting specimen it is necessary to impose stronger regularity conditions. One possible such regularity condition is the so-called \emph{$t$-vertex condition} that was introduced by  D.~G.~Higman in \cite{Hig71} (\cf also \cite{HesHig71}). A graph is said to fulfill the $t$-vertex condition if the number of subgraphs with at most $t$ vertices of a given isomorphism type over a fixed pair of vertices depends only on whether or not the vertices are connected by an edge or whether they are equal. Thus the $t$-vertex condition is in fact a class of regularity conditions parameterized by $t$ which generalizes the regularity conditions of strongly regular graphs. In particular, the srgs are precisely the  graphs that fulfill the $3$-vertex condition. Clearly, all rank $3$ graphs satisfy the $t$-vertex condition for arbitrary $t$. Of special interest are non-rank 3 graphs that satisfy the $t$-vertex condition for some $t>3$. The smallest examples for $t=4$ have order $36$ (\cf \cite{KliMesReiRos05}). As non-rank 3 srgs with the $t$-vertex condition for $t>3$ appear to be very rare, there has been an ongoing research effort to discover new examples and to understand their nature (\cf \cite{Iva89,Iva94,Reich00,KliMesReiRos05,KasKhaOes12,Rei15}). 

Another class of regularity conditions strengthening strong regularity is $\kk$-isoregularity. A graph is said to be $\kk$-isoregular if for every set $S$ of at most $\kk$ vertices the number of common neighbors of the elements of $S$ depends only on the isomorphism type of the subgraph induced by $S$. The srgs are precisely the $2$-isoregular graphs. In the same way that the $t$-vertex condition is a combinatorial approximation of $2$-homogeneity, $\kk$-isoregularity is a combinatorial approximation of $\kk$-homogeneity.  The notion of $\kk$-isoregularity  has its origins in works by J.~M.~J.~Buczak, Ja.~Ju.~Gol'fand, and M.~Klin (\cite{Buc80,GolKli78}). 

For a comprehensive overview on the history and the literature related to the $t$-vertex condition and to $\kk$-isoregularity we refer to Section 9 of Reichard's \cite{Rei15}.

Every $5$-isoregular finite graph is homogeneous (\cf \cite{Cam80}), \ie every isomorphism between subgraphs extends to an automorphism. Similarly, it was conjectured by M.~Klin (\cf \cite{FarKliMuz94}) that there is a number $t_0$ such that an srg is $2$-homogeneous if and only if it satisfies the $t_0$-vertex condition. In order to prove or refute this conjecture it is necessary to have good methods for observing whether or not a given graph fulfills the $t$-vertex condition.  Already in \cite{HesHig71} Hestenes and Higman noticed that in order to verify the $4$-vertex condition it is enough to test it just for two types of subgraphs. More results on how to simplify the testing of the $t$-vertex condition were given by A.~V.~Ivanov and S.~Reichard \cite{Iva94,Reich00}. 

In this paper we develop a general theory of regularity conditions applicable, in principle, to many categories of combinatorial objects. This leads us to new criteria for the $t$-vertex condition and for $(\kk,t)$-regularity (a regularity condition that strengthens the concept of $\kk$-isoregularity in the same way as the $t$-vertex condition strengthens  the concept of $2$-isoregularity). 

Using our theory, we show that the point graphs of partial quadrangles (in the sense of \cite{Cam75}) fulfill the $5$-vertex condition (\cf Theorem~\ref{pqfivevert}). Moreover, we show that if the point graph of a partial quadrangle is $3$-isoregular, then it is $(3,7)$-regular (\cf Theorem~\ref{threseven}). In particular, the point graphs of generalized quadrangles of order $(q,q^2)$ are $(3,7)$-regular (this strengthens  a recent result by S.~Reichard \cite{Rei15} stating that the point graphs of $\GQ(q,q^2)$ satisfy the $7$-vertex condition). As a consequence we obtain that the point graphs of partial quadrangles of order $(s,t,\mu)=(q-1,q^2,q^2-q)$ satisfy the $6$-vertex condition (\cf Corollary~\ref{pqsixvert}). 

In this paper problems from algebraic graph theory are treated using methods from category theory. While the paper is written in a mostly self-contained manner, it may be helpful to have some standard literature from both fields at hand. A modern source for algebraic graph theory is \cite{GodRoy01}. For notions from category theory we refer to the classics \cite{HCA1,Mac98}. Finally, for a recent survey on  homogeneous structures we refer to \cite{Mac11}.

\section{Preliminaries}
Let us start by fixing some notations: A graph is a pair $(V,E)$ where $V$ is a finite set of vertices and $E\subseteq\binom{V}{2}$ is a set of undirected edges. If $\Gamma$ is a graph, then by $V(\Gamma)$ we denote the vertex set and by $E(\Gamma)$ we denote the edge set of $\Gamma$. If $M\subseteq V(\Gamma)$, then by $\Gamma(M)$ we denote the subgraph of $\Gamma$ induced by $M$. As usual, the \emph{order} of a graph is the number of its vertices and the \emph{valency} of a vertex is the number of its neighbors. A graph $\Gamma$ for which $E(\Gamma)=\binom{V(\Gamma)}{2}$ is called a \emph{complete graph}.  A  complete graph of order $n$  is  denoted by $K_n$. The \emph{complement} of a graph $\Gamma$ is $(V(\Gamma),\binom{V(\Gamma)}{2}\setminus E(\Gamma))$. It is denoted by $\overline{\Gamma}$.

The class of all graphs can be naturally equipped with a concept of homomorphisms: A \emph{graph homomorphism} (or short: \emph{homomorphism}) from a graph $\Gamma_1$ to a graph $\Gamma_2$ is a function $f\colon V(\Gamma_1)\to V(\Gamma_2)$ with the property that for each $\{v,w\}\in E(\Gamma_1)$ we have that $\{f(v),f(w)\}\in E(\Gamma_2)$. A one-to-one homomorphism $f\colon \Gamma_1\to \Gamma_2$ is called an \emph{embedding} if for all $\{v,w\}\in\binom{V(\Gamma_1)}{2}: \{v,w\}\in E(\Gamma_1)\iff \{f(v),f(w)\}\in E(\Gamma_2)$.  

Following the tradition of category theory (and conflicting with parts of the tradition of algebraic graph theory),  whenever $f\colon A\to B$, and $g\colon B\to C$, then the composition of $f$ and $g$ is a morphism from $A$ to $C$ that is denoted by $g\circ f$. That is, we use the intuition that morphisms are functions that are applied to elements of their domain from the left, so that $(f\circ g)(x) = f(g(x))$.

Next we  introduce the main construction principle of graphs relevant for this paper. It has a combinatorial and a category theoretic dimension. Let us start with the category theoretic one:
\begin{definition}\label{defpushout}
	Let $\Delta$, $\Theta_1$, $\Theta_2$ be graphs and let $f_1\colon \Delta\to\Theta_1$, $f_2\colon \Delta\to\Theta_2$ be homomorphisms. A \emph{compatible cocone} of $(f_1,f_2)$ is a pair $(g_1,g_2)$ where  $g_1\colon \Theta_1\to \Theta$, $g_2\colon \Theta_2\to\Theta$ for some graph $\Theta$, such that the following diagram commutes:
    \begin{equation}\label{pushout}
    \begin{tikzcd}
	    \Theta_1 \rar{g_1}& \Theta\\
    	 \Delta \uar{f_1}\rar{f_2}& \Theta_2\uar{g_2}
    \end{tikzcd}
    \end{equation}
	$(g_1,g_2)$ is called a \emph{limiting cocone} of $(f_1,f_2)$ if for any other compatible cocone $(h_1,h_2)$ of $(f_1,f_2)$ where $h_1\colon \Theta_1\to\Gamma$, $h_2\colon \Theta_2\to \Gamma$ there exists a unique homomorphism $k\colon \Theta\to\Gamma$ such that the following diagram commutes:
	\[
    \begin{tikzcd}
        ~ & & \Gamma\\
        \Theta_1\rar{g_1}\arrow[bend left]{urr}{h_1} & \Theta \urar[dashed]{k}\\
        \Delta \uar{f_1}\rar{f_2}& \Theta_2\uar{g_2}\arrow[bend right]{uur}{h_2}
    \end{tikzcd}
    \] 
   In that case the diagram~\eqref{pushout} is called a \emph{pushout square}. 
\end{definition}
For us, only the special case when $(f_2,f_2)$ is a pair of embeddings is of interest. In this case, for every limiting cocone $(g_1,g_2)$ of $(f_1,f_2)$ we have that $g_1$ and $g_2$ are embeddings, too. 
A concrete construction of limiting cocones of pairs of embeddings in the category of graphs goes as follows:  
\begin{construction}
    Let $f_1\colon \Delta\injto \Theta_1$, $f_2\colon \Delta\injto\Theta_2$ be embeddings. Let $\widetilde\Theta$ be the disjoint union of $\Theta_1$ and $\Theta_2$. Let $\theta\subseteq V(\widetilde\Theta)^2$ be the smallest equivalence relation that contains $\{ (f_1(v),f_2(v))\mid v\in V(\Delta)\}$.
    Let $\Theta:=\widetilde\Theta/\theta$ \textup{(}vertices of $\Theta$ are equivalence classes of $\theta$ and two classes are connected by an edge if  some representatives of the classes  are connected by an edge in $\widetilde\Theta$\textup{)}. Finally, let $g_1\colon \Theta_1\injto\Theta$ and $g_2\colon \Theta_2\injto\Theta$ be given by
    $g_1\colon  v\mapsto [v]_\theta$, $g_2\colon  w\mapsto[w]_\theta$.
    Then  $(g_1,g_2)$ is a limiting cocone for $(f_1,f_2)$. 
    
    Note that $\theta$ has equivalence classes of size $\le 2$. One can imagine that $\Theta$ is obtained by glueing  $\Theta_1$ and $\Theta_2$  together at a copy of $\Delta$, which is marked in $\Theta_1$ and $\Theta_2$ through $f_1$ and $f_2$, respectively. This construction is also known under the name \emph{graph amalgamation}, \emph{fibered sum} or \emph{amalgamated free sum} (\cf \cite{Nes79,Mac98}).
\end{construction}
\begin{example}
	Consider the following three graphs:
	\[
	\Delta\colon
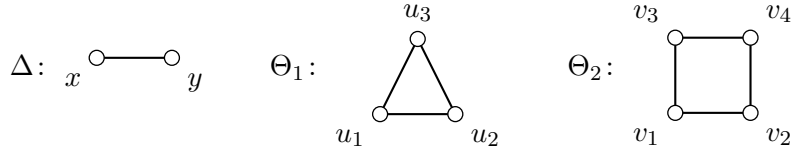
\begin{tikzpicture}[baseline={(current bounding box.center)}]
	\SetGraphUnit{1}
    \GraphInit[vstyle=Welsh]
    \renewcommand*{\VertexSmallMinSize}{2pt}
    \SetVertexMath
    \Vertex[Lpos=-135]{x}
    \EA[Lpos=-45](x){y}
    \Edge(x)(y)
	\end{tikzpicture}\qquad \Theta_1\colon
	\begin{tikzpicture}[baseline={(current bounding box.center)}]
	\SetGraphUnit{1}
    \GraphInit[vstyle=Welsh]
    \renewcommand*{\VertexSmallMinSize}{2pt}
    \SetVertexMath
    \Vertex[Lpos=-135]{u_1}
    \EA[Lpos=-45](u_1){u_2}
    \Vertex[x=0.5,y=1,Lpos=90]{u_3}
    \Edges(u_1,u_2,u_3,u_1)
	\end{tikzpicture}\qquad \Theta_2\colon
	\begin{tikzpicture}[baseline={(current bounding box.center)}]
	\SetGraphUnit{1}
    \GraphInit[vstyle=Welsh]
    \renewcommand*{\VertexSmallMinSize}{2pt}
    \SetVertexMath
    \Vertex[Lpos=-135]{v_1}
    \EA[Lpos=-45](v_1){v_2}
    \NO[Lpos=135](v_1){v_3}
    \NO[Lpos=45](v_2){v_4}
    \Edges(v_1,v_2,v_4,v_3,v_1)
	\end{tikzpicture}
	\]
	Define $f_1\colon\Delta\injto\Theta_1$ and $f_2\colon\Delta\injto\Theta_2$ according to
	\[
	f_1\colon x\mapsto u_1, y\mapsto u_2;\qquad f_2\colon x\mapsto v_3, y\mapsto v_4.
	\]
	According to the construction of amalgamated free sums we have
	$ 
	V(\Theta)=(V(\Theta_1)\dot\cup V(\Theta_2))/\theta,
	$
	where $\theta$ is the equivalence relation on $V(\Theta_1)\dot\cup V(\Theta_2)$ generated by 
	\[
	\{(f_1(x),f_2(x)),(f_1(y),f_2(y))\}=\{(u_1,v_3),(u_2,v_4)\}.
	\]
	 In other words, $V(\Theta)=\{\{u_1,v_3\},\{u_2,v_4\},\{u_3\},\{v_1\},\{v_2\}\}$, and $\Theta$ is given by
	 \[
	 \Theta\colon
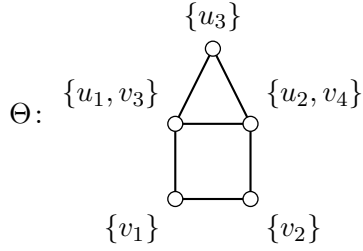
	\begin{tikzpicture}[baseline={(current bounding box.center)}]
	\SetGraphUnit{1}
    \GraphInit[vstyle=Welsh]
    \renewcommand*{\VertexSmallMinSize}{2pt}
    \SetVertexMath
    \Vertex[Lpos=-135,L={\{v_1\}}]{v_1}
    \EA[Lpos=-45,L={\{v_2\}}](v_1){v_2}
    \NO[Lpos=135,L={\{u_1,v_3\}}](v_1){v_3}
    \NO[Lpos=45,L={\{u_2,v_4\}}](v_2){v_4}
    \Vertex[x=0.5,y=2,Lpos=90,L={\{u_3\}}]{u_3}
    \Edges(v_1,v_2,v_4,v_3,v_1)
    \Edges(v_3,u_3,v_4)
	\end{tikzpicture}
	 \]
\end{example}

\section{Graph types and regularity conditions} \label{regularitycond}
The $t$-vertex condition arises from a local invariant of pairs of vertices of a graph. Let $\Gamma=(V,E)$ be a graph and let $(x,y)\in V^2$. We consider all induced subgraphs of $\Gamma$ that contain $x$ and $y$ and that have order $\le t$. Two such subgraphs are said to be of the same type if they are isomorphic by an isomorphism that fixes $x$ and $y$. The possible types of subgraphs correspond to all isomorphism classes of graphs of order $\le t$ with a pair of distinguished vertices. To the pair $(x,y)\in V^2$ we may associate a function $\varphi_{x,y}$ from the types to the natural numbers that maps every type to the number of induced subgraphs of $\Gamma$ that contain $x$ and $y$ and that are of this type. Graphs where the function $\varphi_{x,y}$ does not depend directly on the pair $(x,y)$ but only on whether $x=y$ or $\{x,y\}\in E$ or $\{x,y\}\in \binom{V}{2}\setminus E$, are said to fulfill the \emph{$t$-vertex condition}. In the following we give an equivalent definition of the $t$-vertex condition using the language of category theory.

\subsection{Basic definitions}
%In the previous section we defined the category of graphs.  In order to characterize the $t$-vertex condition and other invariants of graphs we  consider another category, derived from the category of graphs:   
\begin{definition} 
	A \emph{graph type} $\mathbb{T}$ is a triple $(\Delta,\iota,\Theta)$ where $\Delta$ and $\Theta$ are graphs and $\iota\colon \Delta\injto\Theta$ is an embedding. The \emph{order} of $\mathbb{T}$ is the pair $(n,m)$ where $n$ is the order of $\Delta$ and $m$ is the order of $\Theta$. The graphs $\Delta$ and $\Theta$ are called  \emph{base graph} and \emph{underlying graph} of $\mathbb{T}$, respectively.  
\end{definition}
\begin{example}\label{extype}
	Consider the following graphs:
	\[ 
	\Delta\colon
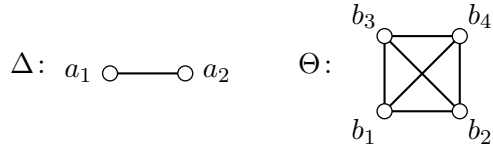
	\begin{tikzpicture}[baseline={([yshift=-0ex]current bounding box.center)}]
		\SetGraphUnit{1}
	    \GraphInit[vstyle=Welsh]
	    \renewcommand*{\VertexSmallMinSize}{2pt}
	    \SetVertexMath
	    \SetVertexLabel
	    \Vertex[Lpos=180]{a_1}
	    \EA(a_1){a_2}
	   % \Edge(a_1)(a_2)
	\end{tikzpicture}\qquad\Theta\colon
	\begin{tikzpicture}[baseline={([yshift=-0ex]current bounding box.center)}]
		\SetGraphUnit{1}
	    \GraphInit[vstyle=Welsh]
	    \renewcommand*{\VertexSmallMinSize}{2pt}
	    \SetVertexMath
	    \SetVertexLabel
	    \Vertex[Lpos=225,Ldist=-4]{b_1}
	    \EA[Lpos=-45,Ldist=-4](b_1){b_2}
	    \NO[Lpos=135,Ldist=-4](b_1){b_3}
	    \EA[Lpos=45,Ldist=-4](b_3){b_4}
	    \Edges(b_2,b_3,b_1,b_4,b_2)
	    \Edge(b_3)(b_4)
	\end{tikzpicture}
	\]	
	$\iota\colon\Delta\injto\Theta$ shall be given by $\iota\colon a_1\mapsto b_1, a_2\mapsto b_2$.  Then $\mathbb{T}=(\Delta,\iota,\Theta)$ is a graph type of order $(2,4)$.
\end{example}
For given graph types $\mathbb{T}_1=(\Delta_1,\iota_1,\Theta_1)$ and $\mathbb{T}_2=(\Delta_2,\iota_2,\Theta_2)$ a \emph{morphism} from $\mathbb{T}_1$ to $\mathbb{T}_2$ is pair $(f,g)$ of graph homomorphisms such that $f\colon \Delta_1\to \Delta_2$, $g\colon \Theta_1\to\Theta_2$ and such that the following diagram commutes.
\[
\begin{tikzcd}
    \Delta_2\rar[hook]{\iota_1} & \Theta_2\\
    \Delta_1\rar[hook]{\iota_1}\uar{f} & \Theta_1\uar[swap]{g}
\end{tikzcd}
\]
With this choice of morphisms graph types form a category. In particular, there is  a natural concept of isomorphism between graph types. 
\begin{remark}
	When we depict a graph type $\bbT=(\Delta,\iota,\Theta)$, we prefer a more compact representation, than in Example~\ref{extype}. We  draw a picture of $\Theta$. Then we mark $\iota(v)$ in black, for all $v\in V(\Delta)$. Clearly, this determines the graph type up to isomorphism.	For instance, the graph type from Example~\ref{extype} is depicted as follows:
	\[ 
	\mathbb{T}\colon
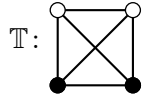
	\begin{tikzpicture}[baseline={([yshift=-0ex]current bounding box.center)}]
		\SetGraphUnit{1}
	    \GraphInit[vstyle=Welsh]
	    \renewcommand*{\VertexSmallMinSize}{2pt}
	    \SetVertexMath
	    \SetVertexNoLabel	
	    \Vertex[Lpos=225]{a_1}
	    \EA[Lpos=-45](b_1){a_2}
	    \SetVertexNoLabel	
	    \NO[Lpos=135](b_1){b_3}
	    \EA[Lpos=45](b_3){b_4}
	    \Edges(b_2,b_3,b_1,b_4,b_2)
	    \Edge(b_3)(b_4)
   	    \AddVertexColor{black}{a_1,a_2}
	 \end{tikzpicture}
	\]
	In case it is not implied otherwise by the context, we  always assume that the base graph $\Delta$ is an induced subgraph of $\Theta$ and that the embedding $\iota$ is the identical embedding.
\end{remark}
A first  observation about the category of graph types is:
\begin{lemma}\label{fingraphtypes}
	Given natural numbers $m$ and $n$ such that $m\le n$, there are just finitely many isomorphism classes of graph types of order $(m,n)$. 
\end{lemma}
\begin{proof}
	There are just finitely  many (say, $k$) unlabeled graphs of order $n$. Moreover, every graph of order $n$ accounts for at most $\binom{n}{m}$ graph types of order $(m,n)$, up to isomorphism.  Hence, there are at most $k\binom{n}{m}$ isomorphism classes of graph types of order $(m,n)$. 
\end{proof}

\begin{definition}
	Let $\mathbb{T}=(\Delta,\iota,\Theta)$ be a graph type, let $\Gamma$ be a graph, and let $\kappa\colon \Delta\injto\Gamma$ be an embedding. An embedding $\hat{\kappa}\colon \Theta\injto\Gamma$ is called an \emph{extension of $\kappa$ along $\iota$} if the following diagram commutes:
   \[
   \begin{tikzcd}
       \Theta \rar[hook]{\hat\kappa}& \Gamma \\
       \Delta\urar[hook]{\kappa}\uar[hook]{\iota}
   \end{tikzcd}
    \]
   The number of all extensions of $\kappa$ along $\iota$  is denoted by $\#(\Gamma, \mathbb{T},\kappa)$. If $\Delta$ embeds into $\Gamma$ and if for every pair of embeddings $\kappa,\kappa'\colon \Delta\injto\Gamma$ we have $\#(\Gamma, \mathbb{T},\kappa) = \#(\Gamma, \mathbb{T},\kappa')$, then this number is denoted by $\#(\Gamma, \mathbb{T})$. In case that $\Delta$ does not embed into $\Gamma$, we define $\#(\Gamma, \mathbb{T}):=0$. In both cases $\Gamma$ is called \emph{$\mathbb{T}$-regular}.
\end{definition}
\begin{example}
	Let us consider the complement graph $\Gamma_1$ of the Petersen graph:
	\[
		\Gamma_1\colon
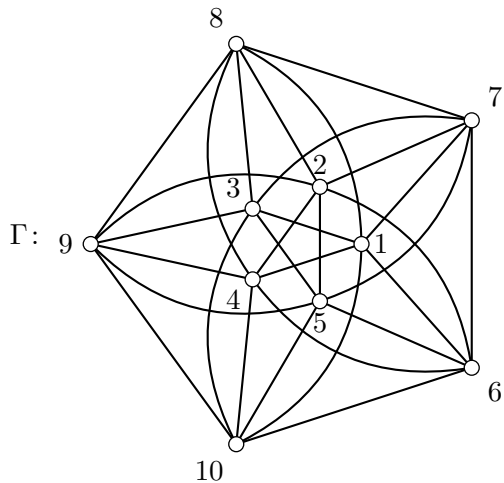
\begin{tikzpicture}[baseline={(current bounding box.center)}, x=0.8cm,y=0.8cm]
	    \GraphInit[vstyle=Welsh]
	    \renewcommand*{\VertexSmallMinSize}{2pt}
	    \SetVertexMath
	    \SetVertexLabel
		\Vertex[a=0,d=1,Lpos=0,Ldist=-2,L={1}]{a0}
		\Vertex[a=-36,d=3.5,Lpos=-36,L={6}]{b0}
		\Vertex[a=72,d=1,Lpos=90,Ldist=-2,L={2}]{a1}
		\Vertex[a=72-36,d=3.5,Lpos=72-36,L={7}]{b1}
		\Vertex[a=144,d=1,Lpos=144,Ldist=-2,L={3}]{a2}
		\Vertex[a=144-36,d=3.5,Lpos=144-36,L={8}]{b2}
		\Vertex[a=216,d=1,Lpos=216,Ldist=-2,L={4}]{a3}
		\Vertex[a=216-36,d=3.5,Lpos=216-36,L={9}]{b3}
		\Vertex[a=288,d=1,Lpos=270,Ldist=-2,L={5}]{a4}
		\Vertex[a=288-36,d=3.5,Lpos=288-36,L={10}]{b4}
	    \EdgeInGraphMod{a}{5}{2}
   	    \EdgeInGraphLoop{b}{5}
   	    \EdgeIdentity{a}{b}{5}
	    \EdgeMod{a}{b}{5}{1}
	    \EdgeMod[style={bend right}]{a}{b}{5}{2}
	    \EdgeMod[style={bend left}]{a}{b}{5}{4}
		\end{tikzpicture}
	\]
	Take the graph type from Example~\ref{extype}. Let us fix an embedding $\kappa\colon\Delta\injto\Gamma_1$, say, $\kappa\colon a_1\mapsto 1, a_2\mapsto 2$. The joint neighbors of $1$ and $2$ in $\Gamma_1$ are $4$, $6$, $7$, and $8$. These vertices induce a $4$-cycle. Thus, there are exactly eight extensions of $\kappa$ along $\iota$, namely
	\begin{align*}
	\hat\kappa_1 &\colon b_1\mapsto 1, b_2\mapsto 2, b_3\mapsto 4, b_4\mapsto 6, &
	\hat\kappa_2 &\colon b_1\mapsto 1, b_2\mapsto 2, b_3\mapsto 6, b_4\mapsto 4,\\
	\hat\kappa_3 &\colon b_1\mapsto 1, b_2\mapsto 2, b_3\mapsto 4, b_4\mapsto 8, &
	\hat\kappa_4 &\colon b_1\mapsto 1, b_2\mapsto 2, b_3\mapsto 8, b_4\mapsto 4, \\
	\hat\kappa_5 &\colon b_1\mapsto 1, b_2\mapsto 2, b_3\mapsto 6, b_4\mapsto 7, &
	\hat\kappa_6 &\colon b_1\mapsto 1, b_2\mapsto 2, b_3\mapsto 7, b_4\mapsto 6, \\
	\hat\kappa_7 &\colon b_1\mapsto 1, b_2\mapsto 2, b_3\mapsto 7, b_4\mapsto 8, &
	\hat\kappa_8 &\colon b_1\mapsto 1, b_2\mapsto 2, b_3\mapsto 8, b_4\mapsto 7.
	\end{align*}
	In particular, we observe that $\#(\Gamma_1, \mathbb{T},\kappa)=8$. Since the automorphism group of $\Gamma_1$ is a rank $3$ group, this number does not depend on the particular choice of $\kappa$. In other words, $\Gamma_1$ is $\mathbb{T}$-regular with $\#(\Gamma_1,\mathbb{T})=8$. 

	Let us now consider the Shrikhande graph $\Gamma_2$:
	\[
		\Gamma_2\colon\begin{tikzpicture}[baseline={(current bounding box.center)}, x=0.8cm,y=0.8cm]
	    \GraphInit[vstyle=Welsh]
	    \renewcommand*{\VertexSmallMinSize}{2pt}
	    \SetVertexMath
	    \SetVertexLabel
		\Vertex[a=0*45,d=1.449,Lpos=0,Ldist=-2,L={1}]{a0}
		\Vertex[a=0*45-0,d=3.5,Lpos=0,L={9}]{b0}
		\Vertex[a=1*45,d=1.449,Lpos=1*45,Ldist=-2,L={2}]{a1}
		\Vertex[a=1*45-0,d=3.5,Lpos=1*45-0,L={10}]{b1}
		\Vertex[a=2*45,d=1.449,Lpos=2*45,Ldist=-2,L={3}]{a2}
		\Vertex[a=2*45-0,d=3.5,Lpos=2*45-0,Ldist=-2,L={11}]{b2}
		\Vertex[a=3*45,d=1.449,Lpos=3*45,Ldist=-2,L={4}]{a3}
		\Vertex[a=3*45-0,d=3.5,Lpos=3*45-0,Ldist=-2,L={12}]{b3}
		\Vertex[a=4*45,d=1.449,Lpos=4*45,Ldist=-2,L={5}]{a4}
		\Vertex[a=4*45-0,d=3.5,Lpos=4*45-0,Ldist=-2,L={13}]{b4}
		\Vertex[a=5*45,d=1.449,Lpos=5*45,Ldist=-2,L={6}]{a5}
		\Vertex[a=5*45-0,d=3.5,Lpos=5*45-0,Ldist=-2,L={14}]{b5}
		\Vertex[a=6*45,d=1.449,Lpos=6*45,Ldist=-2,L={7}]{a6}
		\Vertex[a=6*45-0,d=3.5,Lpos=6*45-0,Ldist=-2,L={15}]{b6}
		\Vertex[a=7*45,d=1.449,Lpos=7*45,Ldist=-2,L={8}]{a7}
		\Vertex[a=7*45-0,d=3.5,Lpos=7*45-0,Ldist=-2,L={16}]{b7}
%		\Vertex[a=216,d=1,Lpos=216,Ldist=-2,L={4}]{a3}
%		\Vertex[a=216-36,d=3.5,Lpos=216-36,L={9}]{b3}
%		\Vertex[a=288,d=1,Lpos=270,Ldist=-2,L={5}]{a4}
%		\Vertex[a=288-36,d=3.5,Lpos=288-36,L={10}]{b4}
	    \EdgeInGraphMod{a}{8}{2}
	    \EdgeInGraphMod{a}{8}{3}
	    \EdgeMod{a}{b}{8}{1}
	    \EdgeMod{a}{b}{8}{7}
   	    \EdgeInGraphLoop{b}{8}
	    \EdgeInGraphMod{b}{8}{2}
%   	    \EdgeIdentity{a}{b}{5}
%	    \EdgeMod{a}{b}{5}{1}
%	    \EdgeMod[style={bend right}]{a}{b}{5}{2}
%	    \EdgeMod[style={bend left}]{a}{b}{5}{4}
		\end{tikzpicture}
	\]
	consider the embedding $\kappa\colon \Delta\injto \Gamma_2$ given by $\kappa\colon a_1\mapsto 1, a_2\mapsto 9$. The joint neighbors of $1$ and $9$ are $10$ and $16$, respectively. These two vertices are connected by an edge. Hence there are exactly two extensions of $\kappa$ along $\iota$, namely
	\begin{align*}
	\hat\kappa_1 &\colon b_1\mapsto 1, b_2\mapsto 9, b_3\mapsto 10, b_4\mapsto 16, &
	\hat\kappa_2 &\colon b_1\mapsto 1, b_2\mapsto 9, b_3\mapsto 16, b_4\mapsto 10.
	\end{align*}
	Thus, we have that $\#(\Gamma_2,\bbT, \kappa)=2$. However, if we consider  $\kappa'\colon\Delta\injto \Gamma_2$ given by $\kappa'\colon a_1\mapsto 1, a_2\mapsto 5$, then the two joint neighbors $3$ and $7$ of $1$ and $5$ are not connected by an edge in $\Gamma_2$. Consequently, there is no extension of $\kappa'$ along $\iota$ in $\Gamma_2$. In other words, $\#(\Gamma_2,\bbT,\kappa')=0$. It follows that the Shrikhande graph is not $\bbT$-regular.	
\end{example}

\begin{remarks}
	\begin{itemize}
		\item If $\bbT=(\Delta,\iota,\Theta)$ is a graph type of order $(0,n)$, and if $\Gamma$ is an arbitrary graph, then  $\Gamma$ is $\bbT$-regular. In this case  $\#(\Gamma,\bbT)$  is equal to the number of subgraphs of $\Gamma$ isomorphic to $\Theta$. 
		\item If $\bbT=(\Delta,\iota,\Theta)$ is a graph type of order $(n,n)$, and if $\Gamma$ is an arbitrary graph, then $\Gamma$ is $\bbT$-regular.  In this case $\#(\Gamma,\bbT)\in\{0,1\}$. It is $1$ if $\Gamma$ has a subgraph isomorphic to $\Theta$ and $0$ otherwise. 
		\item If $\bbT_1$ and $\bbT_2$ are isomorphic graph types, then every graph $\Gamma$ that is $\bbT_1$-regular, is also  $\bbT_2$-regular.  Moreover, in this case we have $\#(\Gamma, \bbT_1)= \#(\Gamma, \bbT_2)$. 
		\item   A concept equivalent to $\mathbb{T}$-regularity, but in the category of complete colored graphs, was introduced and studied by S.~Evdokimov and I.~Ponomarenko in \cite{EvdPon00} in relation with the $t$-vertex condition for association schemes.  
	\end{itemize}	
\end{remarks} 
A simple but important observation is:
\begin{lemma}\label{complreg}
	Let $\Gamma$ be $\bbT$-regular for $\bbT=(\Delta,\iota,\Theta)$. Then  $\overline{\Gamma}$ is $\overline{\bbT}$-regular, where $\overline{\bbT}:=(\overline{\Delta},\iota,\overline{\Theta})$. 
\end{lemma}
\begin{proof}
	Clear.
\end{proof}

\begin{definition}
	Let $m\le n$ be two natural numbers. We say that a graph $\Gamma$ is  
    \begin{itemize}
        \item $(\lseq m,\lseq n)$-regular if it is $\bbT$-regular for all graph types $\bbT$ of order $(m,n)$.
        \item $(\lseq m,n)$-regular if it is $(\lseq m,\lseq l)$-regular for all $m\le l\le n$,
        \item $(m,\lseq n)$-regular if it is $(\lseq k,\lseq n)$-regular for all  $k\le m$,
        \item $(m,n)$-regular if it is $(\lseq k,n)$-regular for all  $k\le m$,
    \end{itemize}
\end{definition}
The concept of $(m,n)$-regularity is a combinatorial approximation of the notion of $m$-homo\-geneity. Recall: 
\begin{definition}
    A graph $\Gamma$ is called \emph{$m$-homogeneous} if every isomorphism between induced subgraphs of order at most $m$ extends to an automorphism of $\Gamma$. It is called \emph{homogeneous} if every isomorphism between finite induced subgraphs extends to an automorphism. 
\end{definition}
It is not hard to see that for every graph $\Gamma$ of order $n$ we have that $m$-homogeneity is equivalent to $(m,n)$-regularity. 
\begin{lemma} 
	A graph $\Gamma$ satisfies the $t$-vertex condition if and only if it is  $(2,t)$-regular.
\end{lemma} 
\begin{proof}
	Clear.
\end{proof}

\subsection{Composition of graph types}
%Let us now see how we may construct new graph types from given ones:
\begin{definition}
	Let $\bbT_1=(\Delta_1,\iota_1,\Theta_1)$ and $\bbT_2=(\Delta_2,\iota_2,\Theta_2)$ be graph types, and let $e\colon \Delta_2\injto\Theta_1$.

	Let $\Lambda$ be a graph, $\lambda_1\colon \Theta_1\injto\Lambda$, $\lambda_2\colon \Theta_2\injto\Lambda$ such that the following is a pushout square (cf.~Definition~\ref{defpushout}):
	\begin{equation*}
	\begin{tikzcd}
 	\Theta_2 \rar[hook]{\lambda_2}& \Lambda \\
 	\Delta_2\uar[hook]{\iota_2}\rar[hook]{e} & \Theta_1\uar[hook]{\lambda_1}
	\end{tikzcd}
	\end{equation*}
	Then the graph type $(\Delta_1,\lambda_1\circ\iota_1,\Lambda)$ is called the \emph{free sum of $\bbT_1$ and $\bbT_2$  with respect to $e$}. It is  denoted by $\mathbb{T}_1 \oplus_e \mathbb{T}_2$.
\end{definition}
\begin{remark} 
	The following picture illustrates the construction of a free sum of types:
	
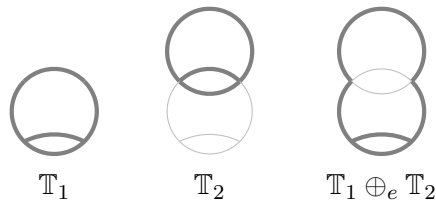
\begin{center}
	\begin{tikzpicture}[scale=0.4]
    \tkzDefPoints{-1/0/A,1/0/B,0/1/C1,0/-1/C2}
    \tkzDrawArc[ultra thick](C2,A)(B)
    \tkzDrawArc[ultra thick](C2,B)(A)
    \tkzDefPoints{-1/-2/X, 1/-2/Y, 0/-4/C3}
    \tkzDrawArc[ultra thick](C3,Y)(X)
    \node at (0,-3.5) {$\bbT_1$};
	\end{tikzpicture}\qquad
	\begin{tikzpicture}[scale=0.4]
    \tkzDefPoints{-1/0/A,1/0/B,0/1/C1,0/-1/C2}
    \tkzDrawArc[ultra thick](C1,B)(A)
    \tkzDrawArc[ultra thick](C1,A)(B)
    \tkzDrawArc[color=lightgray](C2,A)(B)
    \tkzDrawArc[ultra thick](C2,B)(A)
    \tkzDefPoints{-1/-2/X, 1/-2/Y, 0/-4/C3}
    \tkzDrawArc[color=lightgray](C3,Y)(X)
    \node at (0,-3.5) {$\bbT_2$};
	\end{tikzpicture}\qquad
	\begin{tikzpicture}[scale=0.4]
    \tkzDefPoints{-1/0/A,1/0/B,0/1/C1,0/-1/C2}
    \tkzDrawArc[ultra thick](C1,B)(A)
    \tkzDrawArc[color=lightgray](C1,A)(B)
    \tkzDrawArc[ultra thick](C2,A)(B)
    \tkzDrawArc[color=lightgray](C2,B)(A)
    \tkzDefPoints{-1/-2/X, 1/-2/Y, 0/-4/C3}
    \tkzDrawArc[ultra thick](C3,Y)(X)
    \node at (0,-3.5) {$\mathbb{T}_1\oplus_e\bbT_2$};
	\end{tikzpicture}
	\end{center}
	In the picture on the left we see $\mathbb{T}_1$. In the picture in the middle we see $\mathbb{T}_2$, and how $\Delta_2$ is embedded by $e$ into $\Theta_1$. In the picture on the right we see how $\Theta_1$ and $\Theta_2$ are glued together along $\Delta_2$, to obtain $\Lambda$. Now, $\Delta_1$ still naturally embeds into $\Lambda$ and we obtain the free sum of the types with respect to $e$. 
\end{remark}
\begin{example}\label{excomp}
	Let us consider the graph types $\bbT_1=(\Delta_1,\iota_1,\Theta_1)$ and $\bbT_2=(\Delta_2,\iota_2,\Theta_2)$ given by the following pictures:
	\[
	\bbT_1\colon
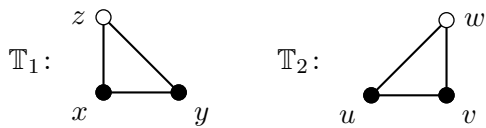
	\begin{tikzpicture}[baseline={(current bounding box.center)}]
	\SetGraphUnit{1}
    \GraphInit[vstyle=Welsh]
    \renewcommand*{\VertexSmallMinSize}{2pt}
    \SetVertexMath
    \Vertex[Lpos=-135]{x}
    \EA[Lpos=-45](x){y}
    \NO[Lpos=180](x){z}
    \Edges(x,z,y,x)
    \AddVertexColor{black}{x,y}
	\end{tikzpicture}\qquad \bbT_2\colon
	\begin{tikzpicture}[baseline={(current bounding box.center)}]
	\SetGraphUnit{1}
    \GraphInit[vstyle=Welsh]
    \renewcommand*{\VertexSmallMinSize}{2pt}
    \SetVertexMath
    \Vertex[Lpos=-135]{u}
    \EA[Lpos=-45](x){v}
    \NO(y){w}
    \Edges(u,w,v,u)
    \AddVertexColor{black}{u,v}
	\end{tikzpicture}
	\]	
	Let $e\colon\Delta_2\injto\Theta_1$ be given by
	\[
	e\colon u\mapsto z,\quad v\mapsto y
	\]
	To obtain the free sum of $\mathbb{T}_1$ and $\mathbb{T}_2$ with respect to $e$, we have to take the disjoint union of $\Theta_1$ and $\Theta_2$, and to identify $u$ with $z$ and $v$ with $y$. We end up with the graph $\Lambda$ in the following picture:
	\[
	\Lambda\colon
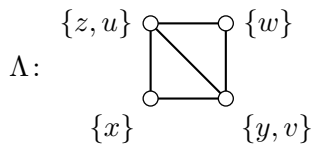
\begin{tikzpicture}[baseline={(current bounding box.center)}]
	\SetGraphUnit{1}
    \GraphInit[vstyle=Welsh]
    \renewcommand*{\VertexSmallMinSize}{2pt}
    \SetVertexMath
    \Vertex[Lpos=-135,L={\{x\}}]{x}
    \EA[Lpos=-45,L={\{y,v\}}](x){y}
    \NO[Lpos=180,L={\{z,u\}}](x){z}
    \NO[L={\{w\}}](y){w}
    \Edges(z,x,y,z,w,y)
	\end{tikzpicture}
	\]
	Finally, we have $\mathbb{T}_1\oplus_e\mathbb{T}_2=(\Delta_1,\iota,\Lambda)$, where $\iota\colon \Delta_1\injto\Lambda$ is given by $\iota\colon x\mapsto \{x\}, y\mapsto\{y,v\}$. If we forget about the labelling, we obtain:
	\[
	\mathbb{T}_1\oplus_e\mathbb{T}_2\colon
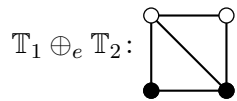
\begin{tikzpicture}[baseline={(current bounding box.center)}]
	\SetGraphUnit{1}
    \GraphInit[vstyle=Welsh]
    \renewcommand*{\VertexSmallMinSize}{2pt}
    \SetVertexMath
    \Vertex[Lpos=-135,NoLabel]{x}
    \EA[Lpos=-45,NoLabel](x){y}
    \NO[NoLabel](x){z}
    \NO[NoLabel](y){w}
    \Edges(z,x,y,z,w,y)
    \AddVertexColor{black}{x,y}
	\end{tikzpicture}
	\]
\end{example}

\subsection{Decomposition of graph types}
\begin{definition}
	Let $\bbT,\bbT_2$ be graph types. We say that $\bbT$ is \emph{$\bbT_2$-reducible} if $\bbT\cong\bbT_1\oplus_e\bbT_2$ for some $\bbT_1\ncong\bbT$ and for some $e$.   
\end{definition}
\begin{remark}\label{remdecomp}
	With the notions from above $\bbT=(\Delta,\iota,\Theta)$ is $\bbT_2$-reducible if and only if the set $V(\Theta)$ can be decomposed as a disjoint union of subsets $M_1$, $M_2$, and $M_3$, such that
    \begin{enumerate}
    \item $\im(\iota)\subseteq M_1\cup M_3$,
    \item $M_2\neq\emptyset$
    \item there are no edges in $\Theta$ between vertices from $M_1$ and vertices from $M_2$,
    \item $\bbT_2':=(\Theta(M_3),\iota',\Theta(M_2\cup M_3))\cong\bbT_2$ (here $\iota'$ denotes the identical embedding).
    \end{enumerate}
	
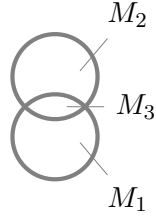
\[
	\begin{tikzpicture}[scale=0.4]
    \tkzDefPoints{-1/0/A,1/0/B,0/1/C1,0/-1/C2}
    \tkzDrawArc[ultra thick](C1,B)(A)
    \tkzDrawArc[ultra thick](C1,A)(B)
    \tkzDrawArc[ultra thick](C2,A)(B)
    \tkzDrawArc[ultra thick](C2,B)(A)
    \tkzDefPoints{-1/-2/X, 1/-2/Y, 0/-4/C3}
	\node [pin=60:$M_2$] at (0.5,0.8){};
	\node [pin=0:$M_3$] at (0,0){};
	\node [pin=-60:$M_1$] at (0.5,-0.8){};
	\end{tikzpicture}
	\]
    In this case, we have $\bbT_1=(\Delta,\iota,\Theta(M_1\cup M_3))$ and $\bbT\cong \bbT_1\oplus_e\bbT_2'$, where $e$ is the identical embedding of $M_3$ into $M_1\cup M_3$.  
\end{remark}
\begin{remark}
	In Example~\ref{excomp}, $\bbT_1\oplus_e\bbT_2$ is $\bbT_2$-reducible.  Beware that there are some degenerate forms of reducibility that we need to take care of: Every graph type $\bbT=(\Delta,\iota,\Theta)$ is $\bbT$-reducible, since $\bbT\cong\bbT_\Delta\oplus_{1_\Delta}\bbT$, where $\bbT_\Delta=(\Delta,1_\Delta,\Delta)$ (here $1_\Delta$ denotes the identity on $V(\Delta)$). Note that whenever $\bbT\cong\bbT'\oplus_e\bbT$ for some $\bbT'$ and some $e$, then $\bbT'\cong\bbT_\Delta$ and $e$ is an isomorphism of $\Delta$ to the base-graph of $\bbT'$.  
\end{remark}

\begin{definition}
    A graph type $\bbT$ is called \emph{$(m,n)$-irreducible} if whenever $\bbT\cong\bbT_1\oplus_e\bbT_2$ for a graph type $\bbT_1$ and a graph type $\bbT_2$, where $\bbT_2$ is of order $(k,l)$ with $k\le m$ and $l\le n$, then we already have $\bbT\cong\bbT_1$ or $\bbT\cong\bbT_2$. Otherwise, we call $\bbT$ \emph{$(m,n)$-reducible}
\end{definition}

\begin{remark}\label{remred}
    A graph type $\bbT$ is $(m,n)$-reducible if and only if it is $\bbT'$-reducible, for some graph type $\bbT'$ of order $(k,l)$, where $k\le m$ and $l\le n$, such that $\bbT'\ncong\bbT$.   
\end{remark}

\begin{example}
	Consider the following graph type of order $(1,4)$: 
	
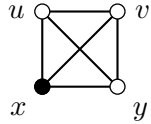
\[
	\begin{tikzpicture}
	\SetGraphUnit{1}
	    \GraphInit[vstyle=Welsh]
	    \renewcommand*{\VertexSmallMinSize}{2pt}
	    \SetVertexMath
	    \Vertex[Lpos=-135]{x}
	    \EA[Lpos=-45](x){y}
	    \NO[Lpos=180](x){u}
	    \NO(y){v}
	    \Edges(v,x,u,v,y,u)
	    \Edges(x,y)
	    \AddVertexColor{black}{x}
	\end{tikzpicture}
	\]
	It is $(2,4)$-reducible, since 
	
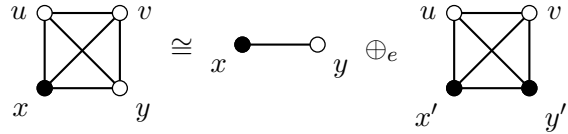
\[
	\begin{tikzpicture}[baseline=2.5ex]
	\SetGraphUnit{1}
	    \GraphInit[vstyle=Welsh]
	    \renewcommand*{\VertexSmallMinSize}{2pt}
	    \SetVertexMath
	    \Vertex[Lpos=-135]{x}
	    \EA[Lpos=-45](x){y}
	    \NO[Lpos=180](x){u}
	    \NO(y){v}
	    \Edges(v,x,u,v,y,u)
	    \Edges(x,y)
	    \AddVertexColor{black}{x}
	\end{tikzpicture}
	\cong
	\begin{tikzpicture}[baseline=-1ex]
	    \SetGraphUnit{1}
	    \GraphInit[vstyle=Welsh]
	    \renewcommand*{\VertexSmallMinSize}{2pt}
	    \SetVertexMath
	    \Vertex[Lpos=-135]{x}
	    \EA[Lpos=-45](x){y}
	    \Edge(x)(y)
	    \AddVertexColor{black}{x}
	\end{tikzpicture}\oplus_e
	\begin{tikzpicture}[baseline=2.5ex]
		\SetGraphUnit{1}
	    \GraphInit[vstyle=Welsh]
	    \renewcommand*{\VertexSmallMinSize}{2pt}
	    \SetVertexMath
	    \Vertex[Lpos=-135,L={x'}]{x}
	    \EA[Lpos=-45,L={y'}](x){y}
	    \NO[Lpos=180](x){u}
	    \NO[Lpos=0](y){v}
	    \Edges(v,x,u,v,y,u)
	    \Edges(x,y)
	    \AddVertexColor{black}{x,y}
	 \end{tikzpicture}
	\]
	Moreover, the graph type  
	
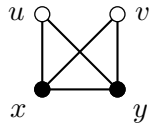
\[
	\begin{tikzpicture}
	\SetGraphUnit{1}
	    \GraphInit[vstyle=Welsh]
	    \renewcommand*{\VertexSmallMinSize}{2pt}
	    \SetVertexMath
	    \Vertex[Lpos=-135]{x}
	    \EA[Lpos=-45](x){y}
	    \NO[Lpos=180](x){u}
	    \NO[Lpos=0](y){v}
	    \Edges(x,u,y,x,v,y)
	    \AddVertexColor{black}{x,y}
	\end{tikzpicture}
	\]
	of order $(2,4)$ is $(2,3)$-reducible (and hence also $(2,4)$-reducible), because
	
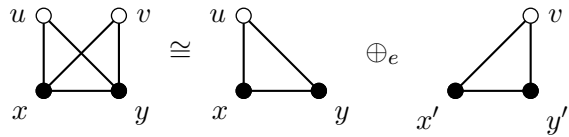
\[
	\begin{tikzpicture}[baseline=2.5ex]
	\SetGraphUnit{1}
	    \GraphInit[vstyle=Welsh]
	    \renewcommand*{\VertexSmallMinSize}{2pt}
	    \SetVertexMath
	    \Vertex[Lpos=-135]{x}
	    \EA[Lpos=-45](x){y}
	    \NO[Lpos=180](x){u}
	    \NO[Lpos=0](y){v}
	    \Edges(x,u,y,x,v,y)
	    \AddVertexColor{black}{x,y}
	\end{tikzpicture}
	\cong
	\begin{tikzpicture}[baseline=2.5ex]
	\SetGraphUnit{1}
	    \GraphInit[vstyle=Welsh]
	    \renewcommand*{\VertexSmallMinSize}{2pt}
	    \SetVertexMath
	    \Vertex[Lpos=-135]{x}
	    \EA[Lpos=-45](x){y}
	    \NO[Lpos=180](x){u}
	    \Edges(x,u,y,x)
	    \AddVertexColor{black}{x,y}
	\end{tikzpicture}\oplus_e
	\begin{tikzpicture}[baseline=2.5ex]
	\SetGraphUnit{1}
	    \GraphInit[vstyle=Welsh]
	    \renewcommand*{\VertexSmallMinSize}{2pt}
	    \SetVertexMath
	    \Vertex[Lpos=-135,L={x'}]{x}
	    \EA[Lpos=-45,L={y'}](x){y}
	    \NO[Lpos=0](y){v}
	    \Edges(x,v,y,x)
	    \AddVertexColor{black}{x,y}
	\end{tikzpicture}
	\]
	In both examples, $e\colon x'\mapsto x, y'\mapsto y$.
\end{example}

In the following it is our goal to link the concept  of $(m,n)$-reducibility to classical graph-theoretical terms.
\begin{definition}
    Let $\bbT=(\Delta,\iota,\Theta)$ be a graph type of order $(m,n)$. Suppose $\Theta=(T,E)$. Let $S\subseteq T$ be the image of $\iota$. Then we define $\Cl(\bbT)$ to be the graph with vertex set $T$ and with edge set $E\cup \binom{S}{2}$.
\end{definition}
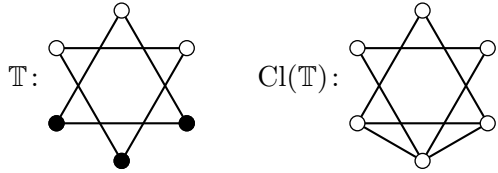
\begin{example}
	\[
	\bbT\colon\begin{tikzpicture}[baseline={(current bounding box.center)}]
		\SetGraphUnit{1}
	    \GraphInit[vstyle=Welsh]
	    \renewcommand*{\VertexSmallMinSize}{2pt}
	    \SetVertexMath
	    \SetVertexNoLabel
   		\Vertex[a=0*60-150,d=1]{x}
   		\Vertex[a=1*60-150,d=1]{y}
   		\Vertex[a=2*60-150,d=1]{z}
   		\Vertex[a=3*60-150,d=1]{u}
   		\Vertex[a=4*60-150,d=1]{v}
   		\Vertex[a=5*60-150,d=1]{w}
		\Edges(x,z,v,x)
		\Edges(y,u,w,y)
		\AddVertexColor{black}{x,y,z}
	   \end{tikzpicture}\qquad
	   \Cl(\bbT)\colon
	   \begin{tikzpicture}[baseline={(current bounding box.center)}]
		\SetGraphUnit{1}
	    \GraphInit[vstyle=Welsh]
	    \renewcommand*{\VertexSmallMinSize}{2pt}
	    \SetVertexMath
	    \SetVertexNoLabel
   		\Vertex[a=0*60-150,d=1]{x}
   		\Vertex[a=1*60-150,d=1]{y}
   		\Vertex[a=2*60-150,d=1]{z}
   		\Vertex[a=3*60-150,d=1]{u}
   		\Vertex[a=4*60-150,d=1]{v}
   		\Vertex[a=5*60-150,d=1]{w}
		\Edges(x,z,v,x)
		\Edges(y,u,w,y)
		\Edges(x,y,z)
	   \end{tikzpicture} 
	\]
\end{example}

Recall that a graph $\Gamma$ is called \emph{$k$-decomposable} if there exists
a $k$-element set of vertices whose deletion makes the graph
disconnected. Moreover, $\Gamma$ is called \emph{$(k+1)$-connected} if it is $l$-indecomposable, for all $l\in\{0,\dots,k\}$.

\begin{lemma}\label{connectcond}
    Let $m+2\le n$. A graph type  $\bbT=(\Delta,\iota,\Theta)$ of order  $(m,n)$ is $(m,n)$-irreducible if and only if $\Cl(\bbT)$ is $(m+1)$-connected. 
\end{lemma}
\begin{proof}
  ``$\Leftarrow$:''  Suppose that $\bbT=(\Delta,\iota,\Theta)$  is $(m,n)$-reducible. Then $\bbT$ is $\bbT'$-reducible for some graph type $\bbT'$ of order $(k,l)$, where $k\le m$ and $l\le n$, such that $\bbT'\ncong\bbT$ (\cf Remark~\ref{remred}). Let us fix some such graph type $\bbT'$. Like in Remark~\ref{remdecomp} we may decompose $V(\Theta)$ as a disjoint union of subsets $M_1$, $M_2$, $M_3$, such that  $\im(\iota)\subseteq M_1\cup M_2$, $M_2\neq\emptyset$, there are no edges in $\Theta$ between vertices from $M_1$ and $M_2$, and such that $\bbT'\cong\bbT'':=(\Theta(M_3),\iota'',\Theta(M_2\cup M_3))$, where $\iota''$ is the identical embedding. We claim that $M_1$ cannot be empty. Suppose on the contrary that $M_1=\emptyset$. Then $\im(\iota)\subseteq M_3$. Since $|\im(\iota)|=m\ge k$ and since $|M_3|=k$, it follows that $\im(\iota)=M_3$. Let $\tilde\iota$ be the image-restriction of $\iota$ to $\Theta(M_3)$. Then, evidently, the following diagram commutes:
  \[
  		\begin{tikzcd}
		\Delta \arrow[r,hook,"\iota"]\arrow[d,hook',"\tilde\iota"',"\cong"]& \Theta\arrow[d,"="]\\
		\Theta(M_3)\arrow[r,hook,"\iota''", "="'] & \Theta(M_2\cup M_3).
	\end{tikzcd}
  \]
  In other words, $(\tilde\iota,1_\Theta)$ is an isomorphism from $\bbT$ to $\bbT''$, a contradiction.

Now we observe that in $\Cl(\bbT)$ there are still no edges between vertices from $M_1$ and vertices from $M_2$, since only edges between vertices in $\im(\iota)\subseteq M_1$ are added in the course of the construction of $\Cl(\bbT)$. Thus, removing the $k$  vertices of $M_3$ from $\Cl(\bbT)$ makes the remainder disconnected. It follows that $\Cl(\bbT)$ is $k$-decomposable. Consequently, $\Cl(\bbT)$ is not $(m+1)$-connected.

``$\Rightarrow$:'' Suppose that $\hat\Theta:=\Cl(\bbT)$ is not $(m+1)$-connected. Then there exists some $k\le m$ such that  $\hat\Theta$ is $k$-decomposable. Thus, there exists pairwise disjoint subsets $M_1$, $M_2$, $M_3$ of $V(\hat\Theta)$, such that $M_1\cup M_2\cup M_3=V(\hat\Theta)$, $M_1,M_2\neq\emptyset$, $|M_3|=k$, and such that there are no edges in $\hat\Theta$ between vertices from $M_1$ and vertices from $M_2$. Thus, if $M\subseteq V(\Theta)$ denotes the image of $\iota$, then we have $M\subseteq M_1\cup M_3$ or $M\subseteq M_2\cup M_3$. Without loss of generality assume that $M\subseteq M_1\cup M_3$. Let $\Delta':=\Theta(M_3)$, and let $\Theta':=\Theta(M_2\cup M_3)$. Then, with $\bbT'=(\Delta',\iota',\Theta')$ (where $\iota'$ is the identical embedding), we obtain that $\bbT$ is $\bbT'$-reducible (\cf Remark~\ref{remdecomp}). By construction we have that $\bbT'$ is of order $(k,l)$, where $l=|M_2\cup M_3|<n$. Thus $\bbT'\ncong\bbT$. Consequently, $\bbT$ is $(m,n)$-reducible (\cf Remark~\ref{remred}). 
\end{proof}

\begin{example}\label{exHesHig}
    The only $3$-connected graph of order $4$ is the complete graph $K_4$. Thus, the only $(2,4)$-irreducible graph types of order $(2,4)$ are:
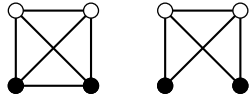
\[
\begin{tikzpicture}
\SetGraphUnit{1}
    \GraphInit[vstyle=Welsh]
    \renewcommand*{\VertexSmallMinSize}{2pt}
    \SetVertexMath
    \Vertex[Lpos=-135,NoLabel]{x}
    \EA[Lpos=-45,NoLabel](x){y}
    \NO[NoLabel](x){u}
    \NO[NoLabel](y){v}
    \Edges(v,x,u,v,y,u)
    \Edges(x,y)
    \AddVertexColor{black}{x,y}
\end{tikzpicture}\qquad
\begin{tikzpicture}
\SetGraphUnit{1}
    \GraphInit[vstyle=Welsh]
    \renewcommand*{\VertexSmallMinSize}{2pt}
    \SetVertexMath
    \Vertex[Lpos=-135,NoLabel]{x}
    \EA[Lpos=-45,NoLabel](x){y}
    \NO[NoLabel](x){u}
    \NO[NoLabel](y){v}
    \Edges(v,x,u,v,y,u)
    \AddVertexColor{black}{x,y}
\end{tikzpicture}
\]
\end{example}

\subsection{The dominance quasiorder of graph types}
\begin{definition}
    Let $\bbT_1=(\Delta_1,\iota_1,\Theta_1)$, $\bbT_2= (\Delta_2,\iota_2,\Theta_2)$ be graph types. Then we define $\bbT_1\preceq\bbT_2$ ($\bbT_2$ \emph{dominates} $\bbT_1$) if there exists a morphism $(f,g)\colon \bbT_2\to\bbT_1$ such that $f\colon \Delta_2\to\Delta_1$ is an isomorphism and such that $g\colon \Theta_2\to\Theta_1$ is surjective on vertices.  If, in addition,  $g$ is not an isomorphism, then we write $\bbT_1\prec\bbT_2$.
\end{definition}
\begin{lemma}
     The relation $\preceq$ defines a quasiorder on graph types. For finite graph types $\bbT_1$, $\bbT_2$ we have $\bbT_1\cong\bbT_2$ if and only if $\bbT_1\preceq\bbT_2$ and $\bbT_2\preceq\bbT_1$. 
\end{lemma}
\begin{proof}
    Clear. 
\end{proof}
\begin{example}
	In the picture below the order diagram of the domination quasiorder of all graph types of order $(2,k)$ for $2\le k\le 4$ with base graph $\Delta$ isomorphic to $K_2$ can be found (in this diagram, a graph type $\bbT_2$ dominates a graph type $\bbT_1$ iff $\bbT_2$ can be reached by an upwards-sloped path starting from $\bbT_1$).  
	\newcommand{\cpunit}{0.8}
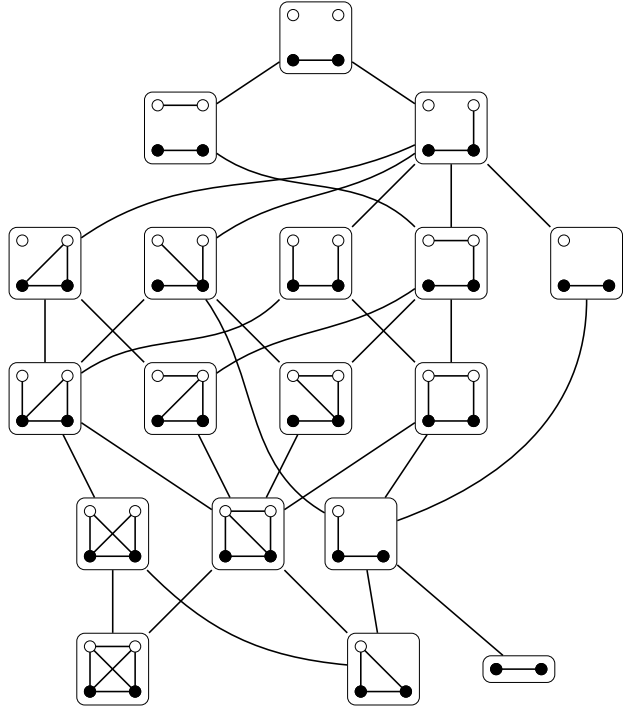
	\[\scalebox{0.75}{
	\begin{tikzpicture}[scale=0.8,baseline={(current bounding box.center)},outer/.style={rounded corners,draw}]
		\node[outer](a0) at (12,0){%
			\begin{tikzpicture}
			\SetGraphUnit{\cpunit}
	   	 	\GraphInit[vstyle=Welsh]
    		\renewcommand*{\VertexSmallMinSize}{2pt}
    		\SetVertexMath
    		\Vertex[NoLabel,Lpos=180]{x}
    		\EA[NoLabel,Lpos=0](x){y}
    		\Edge(x)(y)
		    \AddVertexColor{black}{x,y}
			\end{tikzpicture}};
		\node[outer](a1) at (9,0){%
			\begin{tikzpicture}
			\SetGraphUnit{\cpunit}
	   	 	\GraphInit[vstyle=Welsh]
    		\renewcommand*{\VertexSmallMinSize}{2pt}
    		\SetVertexMath
    		\Vertex[NoLabel,Lpos=180]{x}
    		\EA[Lpos=0,NoLabel](x){y}
    		\NO[NoLabel](x){u}
    		\Edges(x,u,y,x)
		    \AddVertexColor{black}{x,y}
			\end{tikzpicture}};
		\node[outer](a2) at (3,0){%
			\begin{tikzpicture}
			\SetGraphUnit{\cpunit}
    		\GraphInit[vstyle=Welsh]
    		\renewcommand*{\VertexSmallMinSize}{2pt}
    		\SetVertexMath
    		\Vertex[NoLabel,Lpos=180]{x}
    		\EA[Lpos=0,NoLabel](x){y}
    		\NO[NoLabel](x){u}
    		\NO[NoLabel](y){v}
    		\Edges(x,u,y,x,v,y)
    		\Edge(u)(v)
		    \AddVertexColor{black}{x,y}
			\end{tikzpicture}};
		\node[outer](a3) at (3,3){%
			\begin{tikzpicture}
			\SetGraphUnit{\cpunit}
    		\GraphInit[vstyle=Welsh]
    		\renewcommand*{\VertexSmallMinSize}{2pt}
    		\SetVertexMath
    		\Vertex[Lpos=180,NoLabel]{x}
    		\EA[Lpos=0,NoLabel](x){y}
    		\NO[NoLabel](x){u}
    		\NO[NoLabel](y){v}
    		\Edges(x,u,y,x,v,y)
		    \AddVertexColor{black}{x,y}
		\end{tikzpicture}};
		\node[outer](a4) at (6,3){%
			\begin{tikzpicture}
			\SetGraphUnit{\cpunit}
    		\GraphInit[vstyle=Welsh]
    		\renewcommand*{\VertexSmallMinSize}{2pt}
    		\SetVertexMath
    		\Vertex[NoLabel,Lpos=180]{x}
    		\EA[Lpos=0,NoLabel](x){y}
    		\NO[NoLabel](x){u}
    		\NO[NoLabel](y){v}
    		\Edges(y,u,x,y,v,u)
		    \AddVertexColor{black}{x,y}
		\end{tikzpicture}};
		\node[outer](a7) at (1.5,6){%
			\begin{tikzpicture}
			\SetGraphUnit{\cpunit}
    		\GraphInit[vstyle=Welsh]
    		\renewcommand*{\VertexSmallMinSize}{2pt}
    		\SetVertexMath
    		\Vertex[NoLabel,Lpos=180]{x}
    		\EA[Lpos=0,NoLabel](x){y}
    		\NO[NoLabel](x){u}
    		\NO[NoLabel](y){v}
    		\Edges(u,x,y,v,x)
		    \AddVertexColor{black}{x,y}
		\end{tikzpicture}};
		\node[outer](a8) at (4.5,6){%
			\begin{tikzpicture}
			\SetGraphUnit{\cpunit}
    		\GraphInit[vstyle=Welsh]
    		\renewcommand*{\VertexSmallMinSize}{2pt}
    		\SetVertexMath
    		\Vertex[NoLabel,Lpos=180]{x}
    		\EA[Lpos=0,NoLabel](x){y}
    		\NO[NoLabel](x){u}
    		\NO[NoLabel](y){v}
    		\Edges(u,v,x,y,v)
		    \AddVertexColor{black}{x,y}
		\end{tikzpicture}};
		\node[outer](a9) at (7.5,6){%
			\begin{tikzpicture}
			\SetGraphUnit{\cpunit}
    		\GraphInit[vstyle=Welsh]
    		\renewcommand*{\VertexSmallMinSize}{2pt}
    		\SetVertexMath
    		\Vertex[NoLabel,Lpos=180]{x}
    		\EA[Lpos=0,NoLabel](x){y}
    		\NO[NoLabel](x){u}
    		\NO[NoLabel](y){v}
    		\Edges(x,y,u,v,y)
		    \AddVertexColor{black}{x,y}
		\end{tikzpicture}};
		\node[outer](a10) at (10.5,6){%
			\begin{tikzpicture}
			\SetGraphUnit{\cpunit}
    		\GraphInit[vstyle=Welsh]
    		\renewcommand*{\VertexSmallMinSize}{2pt}
    		\SetVertexMath
    		\Vertex[NoLabel,Lpos=180]{x}
    		\EA[Lpos=0,NoLabel](x){y}
    		\NO[NoLabel](x){u}
    		\NO[NoLabel](y){v}
    		\Edges(x,y,v,u,x)
		    \AddVertexColor{black}{x,y}
		\end{tikzpicture}};
		\node[outer](a6) at (8.5,3){%
			\begin{tikzpicture}
			\SetGraphUnit{\cpunit}
    		\GraphInit[vstyle=Welsh]
    		\renewcommand*{\VertexSmallMinSize}{2pt}
    		\SetVertexMath
    		\Vertex[NoLabel,Lpos=180]{x}
    		\EA[Lpos=0,NoLabel](x){y}
    		\NO[NoLabel](x){u}
    		\Edges(u,x,y)
		    \AddVertexColor{black}{x,y}
		\end{tikzpicture}};
		\node[outer](a11) at (1.5,9){%
			\begin{tikzpicture}
			\SetGraphUnit{\cpunit}
    		\GraphInit[vstyle=Welsh]
    		\renewcommand*{\VertexSmallMinSize}{2pt}
    		\SetVertexMath
    		\Vertex[NoLabel,Lpos=180]{x}
    		\EA[Lpos=0,NoLabel](x){y}
    		\NO[NoLabel](x){u}
    		\NO[NoLabel](y){v}
    		\Edges(x,y,v,x)
		    \AddVertexColor{black}{x,y}
		\end{tikzpicture}};
		\node[outer](a12) at (4.5,9){%
			\begin{tikzpicture}
			\SetGraphUnit{\cpunit}
    		\GraphInit[vstyle=Welsh]
    		\renewcommand*{\VertexSmallMinSize}{2pt}
    		\SetVertexMath
    		\Vertex[NoLabel,Lpos=180]{x}
    		\EA[Lpos=0,NoLabel](x){y}
    		\NO[NoLabel](x){u}
    		\NO[NoLabel](y){v}
    		\Edges(u,y,v)
    		\Edge(x)(y)
		    \AddVertexColor{black}{x,y}
		\end{tikzpicture}};
		\node[outer](a13) at (7.5,9){%
			\begin{tikzpicture}
			\SetGraphUnit{\cpunit}
    		\GraphInit[vstyle=Welsh]
    		\renewcommand*{\VertexSmallMinSize}{2pt}
    		\SetVertexMath
    		\Vertex[NoLabel,Lpos=180]{x}
    		\EA[Lpos=0,NoLabel](x){y}
    		\NO[NoLabel](x){u}
    		\NO[NoLabel](y){v}
    		\Edges(u,x,y,v)
		    \AddVertexColor{black}{x,y}
		\end{tikzpicture}};
		\node[outer](a14) at (10.5,9){%
			\begin{tikzpicture}
			\SetGraphUnit{\cpunit}
    		\GraphInit[vstyle=Welsh]
    		\renewcommand*{\VertexSmallMinSize}{2pt}
    		\SetVertexMath
    		\Vertex[NoLabel,Lpos=180]{x}
    		\EA[Lpos=0,NoLabel](x){y}
    		\NO[NoLabel](x){u}
    		\NO[NoLabel](y){v}
    		\Edges(u,v,y,x)
		    \AddVertexColor{black}{x,y}
		\end{tikzpicture}};
		\node[outer](a15) at (13.5,9){%
			\begin{tikzpicture}
			\SetGraphUnit{\cpunit}
    		\GraphInit[vstyle=Welsh]
    		\renewcommand*{\VertexSmallMinSize}{2pt}
    		\SetVertexMath
    		\Vertex[NoLabel,Lpos=180]{x}
    		\EA[Lpos=0,NoLabel](x){y}
    		\NO[NoLabel](x){u}
    		\Edges(x,y)
		    \AddVertexColor{black}{x,y}
		\end{tikzpicture}};
		\node[outer](a17) at (4.5,12){%
			\begin{tikzpicture}
			\SetGraphUnit{\cpunit}
    		\GraphInit[vstyle=Welsh]
    		\renewcommand*{\VertexSmallMinSize}{2pt}
    		\SetVertexMath
    		\Vertex[NoLabel,Lpos=180]{x}
    		\EA[Lpos=0,NoLabel](x){y}
    		\NO[NoLabel](x){u}
    		\NO[NoLabel](y){v}
    		\Edges(x,y)
    		\Edge(u)(v)
		    \AddVertexColor{black}{x,y}
		\end{tikzpicture}};
		\node[outer](a16) at (10.5,12){%
			\begin{tikzpicture}
			\SetGraphUnit{\cpunit}
    		\GraphInit[vstyle=Welsh]
    		\renewcommand*{\VertexSmallMinSize}{2pt}
    		\SetVertexMath
    		\Vertex[NoLabel,Lpos=180]{x}
    		\EA[Lpos=0,NoLabel](x){y}
    		\NO[NoLabel](x){u}
    		\NO[NoLabel](y){v}
    		\Edges(x,y,v)
		    \AddVertexColor{black}{x,y}
		\end{tikzpicture}};
		\node[outer](a18) at (7.5,14){%
			\begin{tikzpicture}
			\SetGraphUnit{\cpunit}
    		\GraphInit[vstyle=Welsh]
    		\renewcommand*{\VertexSmallMinSize}{2pt}
    		\SetVertexMath
    		\Vertex[NoLabel,Lpos=180]{x}
    		\EA[Lpos=0,NoLabel](x){y}
    		\NO[NoLabel](x){u}
    		\NO[NoLabel](y){v}
    		\Edges(x,y)
		    \AddVertexColor{black}{x,y}
		\end{tikzpicture}};
		\Edge(a2)(a3)
		\Edge(a2)(a4) 
		\Edge[style={bend left=20}](a1)(a3)
		\Edge(a1)(a4)
		\Edge(a3)(a7)
		\Edge(a4)(a7)
		\Edge(a4)(a8)
		\Edge(a4)(a9)
		\Edge(a4)(a10)
		\Edge(a0)(a6)
		\Edge(a1)(a6)
		\draw [thick, out=150, in=-55] (a6) to (a12);
		\draw [thick, out=20, in=-90] (a6) to (a15);
		\Edge(a6)(a10)
		\Edge(a7)(a11)
		\Edge(a7)(a12)
		\draw [thick, out=35, in=-135](a7) to (a13);
		\Edge(a8)(a11)
		\draw [thick, out=35, in=-145](a8) to (a14);
		\Edge(a9)(a12)
		\Edge(a9)(a14)
		\Edge(a10)(a13)
		\Edge(a10)(a14)
		\draw [thick, out=35, in=-155](a11) to (a16);
		\draw [thick, out=35, in=-145](a12) to (a16);
		\Edge(a13)(a16)
		\Edge(a14)(a16)
		\draw [thick, out = 135, in=-35](a14) to (a17);
		\Edge(a15)(a16)
		\Edge(a16)(a18)
		\Edge(a17)(a18)
	\end{tikzpicture}}
	\]
\end{example}

\subsection{The type counting lemma}
Now all preparations are made so that we can come to the central auxiliary result of this paper from which all other results depend crucially. It is the place where algebraic graph theory meets category theory. Its proof critically depends on the universal property of amalgamated free sums.
\begin{lemma}[Type counting lemma]\label{type-counting}
    Given a graph $\Gamma$ and graph types $\bbT_1=(\Delta_1,\iota_1,\Theta_1)$ and  $\bbT_2=(\Delta_2,\iota_2,\Theta_2)$. Let $e\colon \Delta_2\injto\Theta_1$ be an embedding.
    Then $\Gamma$ is $\bbT_1\oplus_e\bbT_2$-regular if
    \begin{enumerate}
        \item $\Gamma$ is $\bbT_1$-regular,
        \item $\Gamma$ is $\bbT_2$-regular, and
        \item $\Gamma$ is $\bbT$-regular for every $\bbT\prec\bbT_1\oplus_e\bbT_2$. 
    \end{enumerate}
\end{lemma}
Before coming to the proof of the type counting lemma, we  need to prepare a few auxiliary tools:
\begin{definition}\label{quotient}
	Let $\Theta$ and $\Gamma$ be graphs, and let $h\colon\Theta\to\Gamma$ be a graph homomorphism. By $\Theta/h$ we  denote the graph whose vertex set is $V(\Theta)/{\ker h}$ and whose edge set is given by
	\begin{multline*}
	E(\Theta/h):= \{\{M_1,M_2\}\mid M_1,M_2\in V(\Theta/h),\\\{h(m_1),h(m_2)\}\in E(\Gamma), \text{for some } m_1\in M_1, \text{ and } m_2\in M_2\} 
	\end{multline*}
\end{definition}

\begin{lemma}\label{homdecomp}
	Let $h\colon \Theta\to\Gamma$ be a graph homomorphism. Then the natural mapping $\chi_h\colon V(\Theta)\to V(\Theta/h)$ defined by $\chi_h\colon v\mapsto [v]_{\ker h}$ is a surjective graph homomorphism to $\Theta/h$. Moreover, there is a unique graph embedding $\tilde{h}$ from $\Theta/h$ to $\Gamma$ such that $h=\tilde{h}\circ\chi_h$.
\end{lemma}
\begin{proof}
	Straightforward.
\end{proof}

Now we are ready to prove the type counting lemma. The reader is invited to study Example~\ref{gclexample} in parallel:
\begin{proof}[Proof of Lemma~\ref{type-counting}]
    Let us start by fixing some notations. Suppose $\bbT_1\oplus_e\bbT_2=(\Delta_1,\iota,\Theta)$. 
    Let $\lambda_1$, $\lambda_2$ be given such that the following is a pushout square:
    \begin{equation*}
    \begin{tikzcd}
        \Theta_2 \rar[hook]{\lambda_2}& \Theta \\
        \Delta_2\uar[hook]{\iota_2}\rar[hook]{e} & \Theta_1\uar[hook]{\lambda_1}
    \end{tikzcd}
    \end{equation*}
    and such that $\iota=\lambda_1\circ\iota_1$.
    
    For every compatible cocone $(\mu_1,\mu_2)$  of $(e,\iota_2)$,  let us denote by $h_{\mu_1,\mu_2}\colon \Theta\to\Upsilon$ the unique homomorphism  that makes the following diagram commutative:
    \begin{equation*}
    \begin{tikzcd}
        ~ & & \Upsilon\\
        \Theta_2\rar[hook]{\lambda_2}\arrow[bend left]{urr}{\mu_2} & \Theta \urar[dashed]{h_{\mu_1,\mu_2}}\\
        \Delta_2 \uar[hook]{\iota_2}\rar[hook]{e}& \Theta_1\uar[hook]{\lambda_1}\arrow[bend right]{uur}{\mu_1}
    \end{tikzcd}
    \end{equation*} 
    By Lemma~\ref{homdecomp} we have that  every $h_{\mu_1,\mu_2}$ decomposes uniquely into the  natural homomorphism $\chi_{\mu_1,\mu_2}\colon\Theta\to\Theta/h_{\mu_1,\mu_2}$ and an embedding $\tilde{h}_{\mu_1,\mu_2}\colon\Theta/h_{\mu_1,\mu_2}\injto\Upsilon$.
    
    Let us define $\bbT_{\mu_1,\mu_2}:= (\Delta_1,\chi_{\mu_1,\mu_2}\circ\iota,\Theta/{h_{\mu_1,\mu_2}})$. We claim that if $\mu_1$ and $\mu_2$ are embeddings, then $\bbT_{\mu_1,\mu_2}$ is a graph type, that is, $\chi_{\mu_1,\mu_2}\circ\iota$ is an embedding.  To see this, 
	observe that
	\begin{equation*}
	\tilde{h}_{\mu_1,\mu_2}\circ\chi_{\mu_1,\mu_2}\circ\iota = h_{\mu_1,\mu_2}\circ\iota = h_{\mu_1,\mu_2} \circ\lambda_1\circ\iota_1 = \mu_1\circ\iota_1.
	\end{equation*}
   Thus, since $\mu_1\circ\iota_1$  and $\tilde{h}_{\mu_1,\mu_2}$ are embeddings, it follows that so is $\chi_{\mu_1,\mu_2}\circ\iota$. Note that $\bbT_1\oplus_e\bbT_2$ dominates $\bbT_{\mu_1,\mu_2}$, since $(1_{\Delta_1}, \chi_{\mu_1,\mu_2})\colon \bbT_1\oplus_e\bbT_2\to\bbT_{\mu_1,\mu_2}$, and since $\chi_{\mu_1,\mu_2}$ is surjective. Let us collect the graph types obtained in this way in a set $\cT$: 
   	\[
   	 \cT:=\{\bbT_{\mu_1,\mu_2}\mid (\mu_1,\mu_2) \text{ is a compatible cocone of }(e,\iota_2), \text{ and } \mu_1,\mu_2\text{ are embeddings}\}.
   	\]
   	Note that in the definition of $\cT$ the compatible cocones $(\mu_1,\mu_2)$ of $(e,\iota_2)$ are \emph{not} restricted to a fixed codomain $\Upsilon$. In particular they form a proper class and it is not clear a priori that $\cT$ is actually a set. Next we will prove the following claims:
   	\begin{enumerate}[label=(\Alph*),ref=(\Alph*)]
   		\item\label{claimfinite} $\cT$ is a finite set.
   		\item\label{claimiso} Exactly one element of $\cT$, namely $\bbT_{\lambda_1,\lambda_2}$, is isomorphic to $\bbT_1\oplus_e\bbT_2$. In particular, all other elements of $\cT$ are strictly dominated by  $\bbT_1\oplus_e\bbT_2$. 
   	\end{enumerate}
   	About \ref{claimfinite}:  Recall that for every compatible cocone $(\mu_1,\mu_2)$ of $(e,\iota_2)$ 
   	we have $\bbT_{\mu_1,\mu_2}=(\Delta_1,\chi_{\mu_1,\mu_2}\circ\iota,\Theta/{h_{\mu_1,\mu_2}})$. Let us analyze $\Theta/{h_{\mu_1,\mu_2}}$. 
   	According to Definition~\ref{quotient} its vertex set is $V(\Theta)/\ker h_{\mu_1,\mu_2}$. 
   	Thus, the number of possible quotients $\Theta/{h_{\mu_1,\mu_2}}$ is bounded from above by $B_n\cdot 2^{\binom{n}{2}}$, where $n=|V(\Theta)|$ and where $B_n$ denotes the $n$-th Bell number. Since   $\chi_{\mu_1,\mu_2}\circ\iota\colon\Delta_1\to\Theta/{h_{\mu_1,\mu_2}}$ is an embedding, it is in particular a function. Thus, the cardinality of $\cT$ can be estimated from above by $B_n\cdot 2^{\binom{n}{2}}\cdot n^m$, where $m=|V(\Delta_1)|$. 
   	   	
   	About \ref{claimiso}: First we note that $\bbT_{\lambda_1,\lambda_2}\in \cT$, since $\lambda_1$ and $\lambda_2$ are embeddings and since $(\lambda_1,\lambda_2)$ is a limiting cocone for $(e,\iota_2)$. Clearly, $h_{\lambda_1,\lambda_2}=1_{\Theta}$. So $\ker h_{\lambda_1,\lambda_2}$ is the equality relation and $\Theta/h_{\lambda_1,\lambda_2}$ is obtained from $\Theta$ by renaming each vertex $v$ to the singleton class $\{v\}=[v]_{\ker h_{\lambda_1,\lambda_2}}$. In particular, $\chi_{\lambda_1,\lambda_2}\colon\Theta\to\Theta/h_{\lambda_1,\lambda_2}$ is an isomorphism. Thus $(1_{\Delta_1},\chi_{\lambda_1,\lambda_2})\colon \bbT_1\oplus_e\bbT_2 \to \bbT_{\lambda_1,\lambda_2}$ is an isomorphism, too. It remains to show that $\bbT_{\lambda_1,\lambda_2}$ is the only element of $\cT$ that is isomorphic to $\bbT_1\oplus_e\bbT_2$: Suppose that $\bbT_{\mu_1,\mu_2}= (\Delta_1,\chi_{\mu_1,\mu_2}\circ\iota,\Theta/{h_{\mu_1,\mu_2}})$ is an element of $\cT$ isomorphic to $\bbT_1\oplus_e\bbT_2$. Then in particular, $\Theta/h_{\mu_1,\mu_2}$ is isomorphic to $\Theta$. Since $|V(\Theta/h_{\mu_1,\mu_2})|=|V(\Theta)|$, we have that $\ker h_{\mu_1,\mu_2}$ is the equality relation. Thus  
   	$V(\Theta/h_{\lambda_1,\lambda_2})= V(\Theta/h_{\mu_1,\mu_2})$, and $\chi_{\lambda_1,\lambda_2}$ and $\chi_{\mu_1,\mu_2}$ coincide as functions. Moreover, since $|E(\Theta)|=|E(\Theta/h_{\mu_1,\mu_2})|$, we obtain, that $\chi_{\mu_1,\mu_2}$ is an isomorphism. Consequently, $\bbT_{\mu_1,\mu_2}=\bbT_{\lambda_1,\lambda_2}$, which proves  Claim~\ref{claimiso}.   	
   	
   	At this point it is essential to notice that $\cT$  only depends on $\bbT_1$, $\bbT_2$, and $e$, but not on $\Gamma$. 
    
    Let us fix an embedding $\kappa\colon \Delta_1\to\Gamma$.  
 	Let $\mu_1\colon\Theta_1\injto\Gamma$, $\mu_2\colon\Theta_2\injto\Gamma$. Then $(\mu_1,\mu_2)$ is called a \emph{$\kappa$-compatible pair}  if   
    \begin{enumerate}[label=\emph{\alph*})]
    \item $\mu_1$ extends $\kappa$ along $\iota_1$ (\ie $\kappa=\mu_1\circ\iota_1$),
    \item $\mu_2$ extends $\mu_1\circ e$ along $\iota_2$ (\ie $\mu_1\circ e = \mu_2\circ\iota_2$).
    \end{enumerate}
    Clearly, every $\kappa$-compatible pair is a compatible cocone for $(e,\iota_2)$. Thus, to every $\kappa$-compatible pair we can associate the graph type $\bbT_{\mu_1,\mu_2}$ from $\cT$. 
    Let us define
    \begin{align*}
    P_\kappa &:=\{(\mu_1,\mu_2)\mid (\mu_1,\mu_2) \text{ is a $\kappa$-compatible pair}\},\\
    P_{\kappa,\bbT}&:=\{(\mu_1,\mu_2)\mid (\mu_1,\mu_2)\in P_\kappa, \bbT_{\mu_1,\mu_2}=\bbT\}.
    \end{align*}
    In the following we will show that for every $\bbT=(\Delta_1,\tilde\iota,\widetilde\Theta)\in \cT$ there is a one to one correspondence between the elements of $P_{\kappa,\bbT}$ and the  extensions of $\kappa$ along $\tilde\iota$: Let $\bbT=(\Delta_1,\tilde\iota,\widetilde\Theta)\in \cT$. Then there exists a compatible cocone $(\nu_1,\nu_2)$ of $(e,\iota_2)$, such that $\nu_1$ and $\nu_2$ are embeddings, and such that $\bbT=\bbT_{\nu_1,\nu_2}$. In particular, $\tilde\iota=\chi_{\nu_1,\nu_2}\circ\iota$ and $\widetilde\Theta=\Theta/h_{\nu_1,\nu_2}$. 
    
	Let $\hat{\kappa}\colon \widetilde\Theta\injto\Gamma$ be an extension of  $\kappa$ along $\tilde\iota$ (\ie $\kappa=\hat{\kappa}\circ\tilde\iota$). Define $\mu_1^{[\hat\kappa]}\colon \Theta_1\injto\Gamma$ by $\mu_1^{[\hat\kappa]}:=\hat{\kappa}\circ\chi_{\nu_1,\nu_2}\circ\lambda_1$ and $\mu_2^{[\hat\kappa]}\colon \Theta_2\injto\Gamma$ by $\mu_2^{[\hat\kappa]}:=\hat{\kappa}\circ\chi_{\nu_1,\nu_2}\circ\lambda_2$. 
	    \begin{equation*}
    \begin{tikzcd}[column sep=3em]
    	\Theta_2\arrow[bend left,hook,dashed]{rrr}{\mu_2^{[\hat\kappa]}}\rar[hook]{\lambda_2} & \Theta\rar[->>]{\chi_{\nu_1,\nu_2}} & \widetilde\Theta \rar[hook]{\hat\kappa}& \Gamma\\
    	\Delta_2\rar[hook]{e}\uar[hook]{\iota_2} & \Theta_1\arrow[hook,dashed,swap]{urr}[near end]{\mu_1^{[\hat\kappa]}}\uar[hook]{\lambda_1} & & \Delta_1.\uar[hook,swap]{\kappa}\arrow[hook']{ll}{\iota_1}
    \end{tikzcd}
    \end{equation*} 
    We claim that $(\mu_1^{[\hat\kappa]},\mu_2^{[\hat\kappa]})$ is a $\kappa$-compatible pair. First we
    note that $\mu_1^{[\hat\kappa]}$ is an embedding, since $\tilde{h}_{\nu_1,\nu_2}\circ(\chi_{\nu_1,\nu_2}\circ\lambda_1)=\nu_1$ is an embedding and $\mu_2^{[\hat\kappa]}$ is an embedding, since  $\tilde{h}_{\nu_1,\nu_2}\circ(\chi_{\nu_1,\nu_2}\circ\lambda_2)=\nu_2$ is an embedding. Next we  compute that 
    \[
    \mu_1^{[\hat\kappa]}\circ\iota_1 = \hat\kappa\circ\chi_{\nu_1,\nu_2}\circ\lambda_1\circ\iota_1 = \hat\kappa\circ\chi_{\nu_1,\nu_2}\circ\iota = \hat\kappa\circ\tilde\iota = \kappa,
    \]
    thus $\mu_1^{[\hat\kappa]}$ extends $\kappa$ along $\iota_1$, and
    \[
      \mu_1^{[\hat\kappa]}\circ e = \hat\kappa\circ\chi_{\nu_1,\nu_2}\circ\lambda_1\circ e = \hat\kappa\circ\chi_{\nu_1,\nu_2}\circ\lambda_2\circ\iota_2 = \mu_2^{[\hat\kappa]}\circ\iota_2,
    \]
     thus $\mu_2^{[\hat\kappa]}$ extends $\mu_1^{[\hat\kappa]}\circ e$ along $\iota_2$ and the claim is proved.
     
     The next step is to show that the assignment $\hat\kappa\mapsto(\mu_1^{[\hat\kappa]},\mu_2^{[\hat\kappa]})$ is a bijection. 
     
     ``injectivity'': Let $\hat\kappa_1$ and $\hat\kappa_2$ be extensions of $\kappa$ along $\tilde\iota$ and suppose that $(\mu_1^{[\hat\kappa_1]},\mu_2^{[\hat\kappa_1]})=(\mu_1^{[\hat\kappa_2]},\mu_2^{[\hat\kappa_2]})$. Note that $\hat\kappa_1\circ\chi_{\nu_1,\nu_2}$ is the unique mediating morphism from the limiting cocone $(\lambda_1,\lambda_2)$ to $(\mu_1^{[\hat\kappa_1]},\mu_2^{[\hat\kappa_1]})$, and that  $\hat\kappa_2\circ\chi_{\nu_1,\nu_2}$ is the unique mediating morphism from  $(\lambda_1,\lambda_2)$ to $(\mu_1^{[\hat\kappa_2]},\mu_2^{[\hat\kappa_2]})$.  Since $(\mu_1^{[\hat\kappa_1]},\mu_2^{[\hat\kappa_1]})=(\mu_1^{[\hat\kappa_2]},\mu_2^{[\hat\kappa_2]})$, we have $\hat\kappa_1\circ\chi_{\nu_1,\nu_2}=\hat\kappa_2\circ\chi_{\nu_1,\nu_2}$. Since $\chi_{\nu_1,\nu_2}$ is surjective, we conclude $\hat\kappa_1=\hat\kappa_2$. 
     
     ``surjectivity'': Let $(\mu_1,\mu_2)$ be any $\kappa$-compatible pair such that $\bbT_{\mu_1,\mu_2} = \bbT = \bbT_{\nu_1,\nu_2}$. In particular, $\Theta/h_{\mu_1,\mu_2}=\widetilde\Theta=\Theta/h_{\nu_1,\nu_2} $, and thus also $\chi_{\mu_1,\mu_2}=\chi_{\nu_1,\nu_2}$. We claim that $\tilde{h}_{\mu_1,\mu_2}$ is an extension of $\kappa$ along $\tilde\iota$. Indeed, we may compute
     \[
      \tilde{h}_{\mu_1,\mu_2}\circ\tilde\iota =  \tilde{h}_{\mu_1,\mu_2}\circ\chi_{\nu_1,\nu_2}\circ\iota =\tilde{h}_{\mu_1,\mu_2}\circ\chi_{\mu_1,\mu_2}\circ\iota = h_{\mu_1,\mu_2}\circ\iota =  h_{\mu_1,\mu_2}\circ\lambda_1\circ\iota_1 = \mu_1\circ\iota_1 = \kappa.
     \]
     It remains to show that $\tilde{h}_{\mu_1,\mu_2}$ is really a preimage of $(\mu_1,\mu_2)$ under our correspondence. For this we compute
     \begin{align*}
     	\tilde{h}_{\mu_1,\mu_2} \circ\chi_{\nu_1,\nu_2}\circ\lambda_1 &=\tilde{h}_{\mu_1,\mu_2} \circ\chi_{\mu_1,\mu_2}\circ\lambda_1= h_{\mu_1,\mu_2}\circ\lambda_1 = \mu_1,\text{ and}\\
     	\tilde{h}_{\mu_1,\mu_2} \circ\chi_{\nu_1,\nu_2}\circ\lambda_2 &=\tilde{h}_{\mu_1,\mu_2} \circ\chi_{\mu_1,\mu_2}\circ\lambda_2= h_{\mu_1,\mu_2}\circ\lambda_2 = \mu_2.
     \end{align*}
     
	Now we are able to count the number of $\kappa$-compatible pairs in two different ways. From one hand, by definition, we have
	\[
		|P_\kappa| = \#(\Gamma,\bbT_1)\cdot\#(\Gamma,\bbT_2).
	\]
    From the other hand, by what we showed above, we have 
    \[
    	|P_\kappa| = \sum_{\bbT\in T} |P_{\kappa,\bbT}| = \sum_{\bbT\in T}\#(\Gamma,\bbT,\kappa).
    \]
    Let $\{\bbT^{(1)},\dots,\bbT^{(m)}\}$ be a transversal of the isomorphism classes of the graph types in $\cT$. Without loss of generality, let $\bbT^{(m)}\cong\bbT_1\oplus_e\bbT_2$. For each $i\in\{1,\dots,m\}$, by $c_i$ we denote the number of graph types in $\cT$ isomorphic to $\bbT^{(i)}$. Note that $c_m=1$ and that $\#(\Gamma,\bbT^{(m)},\kappa)=\#(\Gamma,\bbT_1\oplus_e\bbT_2,\kappa)$. With this in mind, we may rewrite 
    \[
    \#(\Gamma,\bbT_1\oplus_e\bbT_2,\kappa) = \#(\Gamma,\bbT_1)\cdot\#(\Gamma,\bbT_2)-\sum_{i=1}^{m-1} c_i\cdot\#(\Gamma,\bbT^{(i)}).
    \]
    It follows that $\#(\Gamma,\bbT_1\oplus_e\bbT_2,\kappa)$ actually does not depend on $\kappa$. In other words, $\Gamma$ is $\bbT_1\oplus_e\bbT_2$-regular.
\end{proof}
The type counting lemma is the technical backbone of all further results in this paper. Alas,  while the language of category theory used in the proof assures correctness, is not ideal to illustrate the combinatorial intuitions behind the proof. To amend this situation, we  elaborate on an extended example:
\begin{example}\label{gclexample}
	Suppose, we are given a $(2,4)$-regular graph $\Gamma$. In other words, $\Gamma$ is strongly regular and satisfies the $4$-vertex condition. Let us illustrate the idea behind the type counting lemma by analyzing the graph type $\bbT=(\Delta,\iota,\Theta)$ given by the following picture (here $\Delta=\Theta(\{x,y\})$, and $\iota\colon \Delta\injto\Theta$ is the identical embedding):
	\[
	\bbT\colon
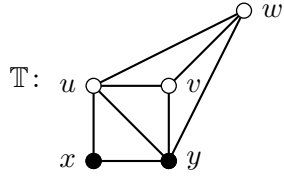
\begin{tikzpicture}[baseline={(current bounding box.center)}]
		\SetGraphUnit{1}
	    \GraphInit[vstyle=Welsh]
	    \renewcommand*{\VertexSmallMinSize}{2pt}
	    \SetVertexMath
	    \SetVertexLabel
	    \Vertex[Lpos=180]{x}
	    \EA(x){y}
	    \NO[Lpos=180](x){u}
	    \Edges(x,u,y,x)
	    \NO(y){v}
	    \NOEA(v){w}
	    \Edges(u,v,y)
	    \Edges(u,w,y)
	    \Edge(v)(w)
	    \AddVertexColor{black}{x,y}
	\end{tikzpicture}
	\]
	Our first observation is that $\bbT$ is $(2,4)$-reducible. In particular we have $\bbT\cong \bbT_1\oplus_e\bbT_2$, where $\bbT_1=(\Delta,\iota_1,\Theta_1)$ and $\bbT_2=(\Delta_2,\iota_2,\Theta_2)$ are given by:
	\[
	\bbT_1\colon
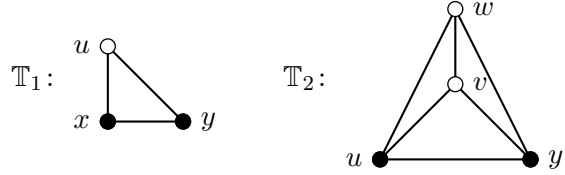
\begin{tikzpicture}[baseline={(current bounding box.center)}]
		\SetGraphUnit{1}
	    \GraphInit[vstyle=Welsh]
	    \renewcommand*{\VertexSmallMinSize}{2pt}
	    \SetVertexMath
	    \SetVertexLabel
	    \Vertex[Lpos=180]{x}
	    \EA(x){y}
	    \NO[Lpos=180](x){u}
	    \Edges(x,u,y,x)
	    \AddVertexColor{black}{x,y}
	\end{tikzpicture}\qquad \bbT_2\colon\begin{tikzpicture}[baseline={(current bounding box.center)}]
		\SetGraphUnit{1}
	    \GraphInit[vstyle=Welsh]
	    \renewcommand*{\VertexSmallMinSize}{2pt}
	    \SetVertexMath
	    \SetVertexLabel
	    \Vertex[Lpos=180]{u}
	    \NOEA(u){v}
	    \NO(v){w}
	    \SOEA[Lpos=0](v){y}
	    \Edges(u,w,y,u,v,w)
	    \Edge(v)(y)
	    \AddVertexColor{black}{u,y}
	\end{tikzpicture}
	\]
	and where $e\colon\Delta_2\injto\Theta_1$ is the identical embedding. 
	Since $\Gamma$ is $(2,4)$-regular, it is $\bbT_1$- and $\bbT_2$-regular.

	Let $(\mu_1,\mu_2)$ be an arbitrary compatible cocone of $(e,\iota_2)$, where $\mu_i\colon\Theta_i\injto\Upsilon$ ($i\in\{1,2\}$), say 
	\begin{align*}
		\mu_1 &\colon x\mapsto a, y\mapsto b, u\mapsto c,\\
		\mu_2 &\colon y\mapsto b, u\mapsto c, v\mapsto d, w\mapsto o, \text{ where } a,b,c,d,o\in V(\Upsilon).
	\end{align*}
	Then the unique mediating morphism $h_{\mu_1,\mu_2}$ is given by
	\[
	h_{\mu_1,\mu_2}\colon x\mapsto a,y\mapsto b,u\mapsto c,v\mapsto d,w\mapsto o.
	\]
	In the following we list all possibilities what the subgraph of $\Upsilon$ induced by $\{a,b,c,d,o\}$ might look like (depending on $\Upsilon$ and on $(\mu_1,\mu_2)$). This list is obtained by constructing all graphs vertex labeled by $\{a,b,c,d,o\}$ in such a way that every vertex has at least one label and such that every label is used exactly once, subject to the condition that the above given functions $\mu_1$ and $\mu_2$ define graph-embeddings. In our case this means that the subgraph induced by the vertices labeled by elements of $\{a,b,c\}$ induce $K_3$ and those labeled by elements from $\{b,c,d,o\}$ induce $K_4$:
	\begin{align*}
		\text{(1)}\colon & 
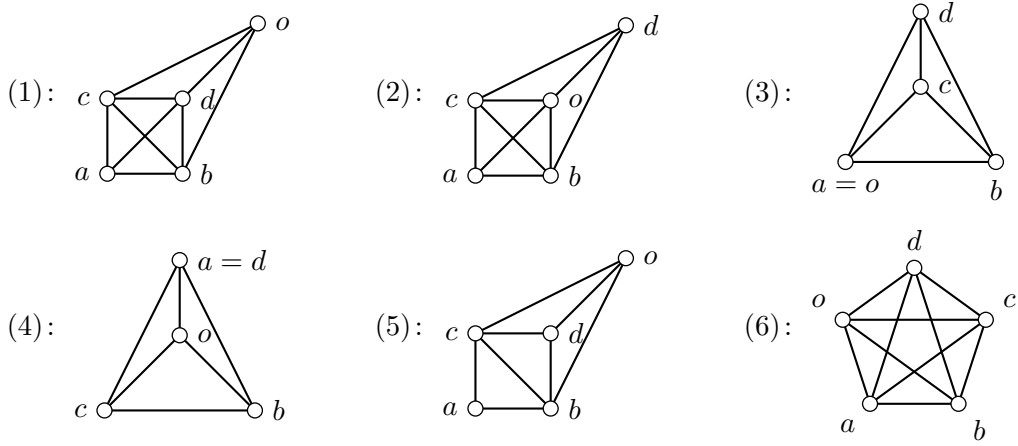
\begin{tikzpicture}[baseline={(current bounding box.center)}]
		\SetGraphUnit{1}
	    \GraphInit[vstyle=Welsh]
	    \renewcommand*{\VertexSmallMinSize}{2pt}
	    \SetVertexMath
	    \SetVertexLabel
	    \Vertex[Lpos=180,L=a]{x_1}
	    \EA[L=b](x_1){y_1}
	    \NO[Lpos=180,L=c](x_1){x_2}
	    \Edges(x_1,x_2,y_1,x_1)
	    \NO[L=d](y_1){u}
	    \NOEA[L=o](u){v}
	    \Edges(x_2,u,y_1)
	    \Edges(x_2,v,y_1)
	    \Edge(u)(v)
	    \Edge(x_1)(u)
	\end{tikzpicture} & 
	\text{(2)}\colon & \begin{tikzpicture}[baseline={(current bounding box.center)}]
		\SetGraphUnit{1}
	    \GraphInit[vstyle=Welsh]
	    \renewcommand*{\VertexSmallMinSize}{2pt}
	    \SetVertexMath
	    \SetVertexLabel
	    \Vertex[Lpos=180,L=a]{x_1}
	    \EA[L=b](x_1){y_1}
	    \NO[Lpos=180,L=c](x_1){x_2}
	    \Edges(x_1,x_2,y_1,x_1)
	    \NO[L=o](y_1){v}
	    \NOEA[L=d](u){u}
	    \Edges(x_2,u,y_1)
	    \Edges(x_2,v,y_1)
	    \Edge(u)(v)
	    \Edge(x_1)(v)
	\end{tikzpicture}&
	\text{(3)}\colon & \begin{tikzpicture}[baseline={(current bounding box.center)}]
		\SetGraphUnit{1}
	    \GraphInit[vstyle=Welsh]
	    \renewcommand*{\VertexSmallMinSize}{2pt}
	    \SetVertexMath
	    \SetVertexLabel
	    \Vertex[Lpos=-90,L={a=o}]{x_2}
	    \NOEA[L=c](x_2){u}
	    \NO[L=d](u){v}
	    \SOEA[Lpos=-90,L=b](u){y_2}
	    \Edges(x_2,v,y_2,x_2,u,v)
	    \Edge(u)(y_2)
	\end{tikzpicture} \\
	\text{(4)}\colon & \begin{tikzpicture}[baseline={(current bounding box.center)}]
		\SetGraphUnit{1}
	    \GraphInit[vstyle=Welsh]
	    \renewcommand*{\VertexSmallMinSize}{2pt}
	    \SetVertexMath
	    \SetVertexLabel
	    \Vertex[Lpos=180,L=c]{x_2}
	    \NOEA[L=o](x_2){v}
	    \NO[L={a=d}](v){u}
	    \SOEA[Lpos=0,L=b](v){y_2}
	    \Edges(x_2,v,y_2,x_2,u,v)
	    \Edge(u)(y_2)
	\end{tikzpicture}&
	\text{(5)}\colon & \begin{tikzpicture}[baseline={(current bounding box.center)}]
		\SetGraphUnit{1}
	    \GraphInit[vstyle=Welsh]
	    \renewcommand*{\VertexSmallMinSize}{2pt}
	    \SetVertexMath
	    \SetVertexLabel
	    \Vertex[Lpos=180,L=a]{x_1}
	    \EA[L=b](x_1){y_1}
	    \NO[Lpos=180,L=c](x_1){x_2}
	    \Edges(x_1,x_2,y_1,x_1)
	    \NO[L=d](y_1){u}
	    \NOEA[L=o](u){v}
	    \Edges(x_2,u,y_1)
	    \Edges(x_2,v,y_1)
	    \Edge(u)(v)
	\end{tikzpicture} & \text{(6)}\colon & \begin{tikzpicture}[baseline={(current bounding box.center)}]
		\SetGraphUnit{1}
	    \GraphInit[vstyle=Welsh]
	    \renewcommand*{\VertexSmallMinSize}{2pt}
	    \SetVertexMath
	    \SetVertexLabel
   		\Vertex[a=0*72-126,d=1,Lpos=0*72-126,Ldist=0,L=a]{x_1}
   		\Vertex[a=1*72-126,d=1,Lpos=1*72-126,Ldist=0,L=b]{y_1}
   		\Vertex[a=2*72-126,d=1,Lpos=2*72-126,Ldist=0,L=c]{x_2}
   		\Vertex[a=3*72-126,d=1,Lpos=3*72-126,Ldist=0,L=d]{u}
   		\Vertex[a=4*72-126,d=1,Lpos=4*72-126,Ldist=0,L=o]{v}
   		\Edges(x_1,y_1,x_2,u,v,x_1,x_2,v,y_1,u,x_1)
	   \end{tikzpicture}
	\end{align*}
	Now we are ready to construct the set $T$ mentioned in the proof of the type counting Lemma.
	In cases (1) and (2) we obtain
	\[
	\bbT^{(1)}:=\bbT_{\mu_1,\mu_2}\colon
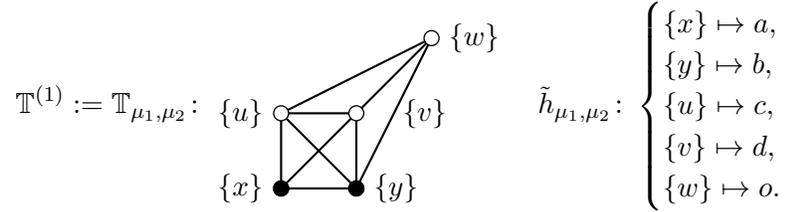
	\begin{tikzpicture}[baseline={(current bounding box.center)}]
		\SetGraphUnit{1}
	    \GraphInit[vstyle=Welsh]
	    \renewcommand*{\VertexSmallMinSize}{2pt}
	    \SetVertexMath
	    \SetVertexLabel
	    \Vertex[Lpos=180,L={\{x\}}]{x_1}
	    \EA[L={\{y\}}](x_1){y_1}
	    \NO[Lpos=180,L={\{u\}}](x_1){x_2}
	    \Edges(x_1,x_2,y_1,x_1)
	    \NO[Ldist=1em,L={\{v\}}](y_1){u}
	    \NOEA[L={\{w\}}](u){v}
	    \Edges(x_2,u,y_1)
	    \Edges(x_2,v,y_1)
	    \Edge(u)(v)
	    \Edge(x_1)(u)
	    \AddVertexColor{black}{x_1,y_1}
	\end{tikzpicture}\quad \tilde{h}_{\mu_1,\mu_2}\colon\begin{cases}
		\{x\}\mapsto a,\\
		\{y\}\mapsto b,\\
		\{u\}\mapsto c,\\
		\{v\}\mapsto d,\\
		\{w\}\mapsto o.
	\end{cases}
	\]
	In case (3) we obtain
	\[
		\bbT^{(2)}:=\bbT_{\mu_1,\mu_2}\colon
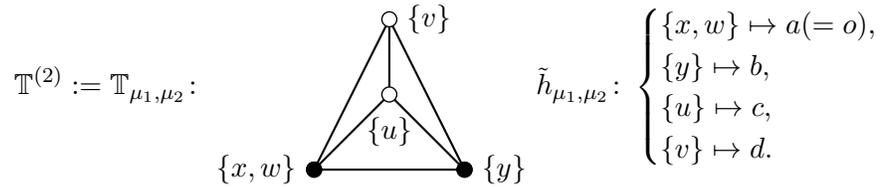
		\begin{tikzpicture}[baseline={(current bounding box.center)}]
		\SetGraphUnit{1}
	    \GraphInit[vstyle=Welsh]
	    \renewcommand*{\VertexSmallMinSize}{2pt}
	    \SetVertexMath
	    \SetVertexLabel
	    \Vertex[Lpos=180,L={\{x,w\}}]{x_2}
	    \NOEA[Lpos=-90,Ldist=2,L={\{u\}}](x_2){u}
	    \NO[L={\{v\}}](u){v}
	    \SOEA[Lpos=0,L={\{y\}}](u){y_2}
	    \Edges(x_2,v,y_2,x_2,u,v)
	    \Edge(u)(y_2)
	    \AddVertexColor{black}{x_2,y_2}
		\end{tikzpicture}\tilde{h}_{\mu_1,\mu_2}\colon\begin{cases}
		\{x,w\}\mapsto a(=o),\\
		\{y\}\mapsto b,\\
		\{u\}\mapsto c,\\
		\{v\}\mapsto d.
	\end{cases}
	\]
	In case (4) we obtain
	\[
	\bbT^{(3)}:=\bbT_{\mu_1,\mu_2}\cong\bbT^{(2)}\colon
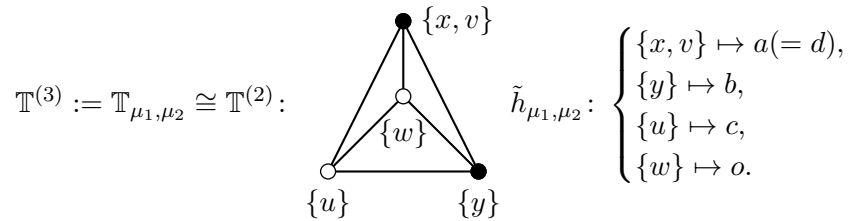
	\begin{tikzpicture}[baseline={(current bounding box.center)}]
		\SetGraphUnit{1}
	    \GraphInit[vstyle=Welsh]
	    \renewcommand*{\VertexSmallMinSize}{2pt}
	    \SetVertexMath
	    \SetVertexLabel
	    \Vertex[Lpos=-90,L={\{u\}}]{x_2}
	    \NOEA[Lpos=-90,Ldist=2,L=\{w\}](x_2){u}
	    \NO[L={\{x,v\}}](u){v}
	    \SOEA[Lpos=-90,L={\{y\}}](u){y_2}
	    \Edges(x_2,v,y_2,x_2,u,v)
	    \Edge(u)(y_2)
	    \AddVertexColor{black}{v,y_2}
	\end{tikzpicture} \tilde{h}_{\mu_1,\mu_2}\colon\begin{cases}
		\{x,v\}\mapsto a(=d),\\
		\{y\}\mapsto b,\\
		\{u\}\mapsto c,\\
		\{w\}\mapsto o.
	\end{cases}
	\]
	In case (5) we obtain
	\[
	\bbT^{(4)}:=\bbT_{\mu_1,\mu_2}\cong\bbT\colon
	\begin{tikzpicture}[baseline={(current bounding box.center)}]
		\SetGraphUnit{1}
	    \GraphInit[vstyle=Welsh]
	    \renewcommand*{\VertexSmallMinSize}{2pt}
	    \SetVertexMath
	    \SetVertexLabel
	    \Vertex[Lpos=180,L={\{x\}}]{x_1}
	    \EA[L={\{y\}}](x_1){y_1}
	    \NO[Lpos=180,L={\{u\}}](x_1){x_2}
	    \Edges(x_1,x_2,y_1,x_1)
	    \NO[Ldist=1em,L={\{v\}}](y_1){u}
	    \NOEA[L={\{w\}}](u){v}
	    \Edges(x_2,u,y_1)
	    \Edges(x_2,v,y_1)
	    \Edge(u)(v)
   	    \AddVertexColor{black}{x_1,y_1}
	\end{tikzpicture}\quad \tilde{h}_{\mu_1,\mu_2}\colon\begin{cases}
		\{x\}\mapsto a,\\
		\{y\}\mapsto b,\\
		\{u\}\mapsto c,\\
		\{v\}\mapsto d,\\
		\{w\}\mapsto o.
	\end{cases}
	\] 
	In case (6) we obtain 
	\[
	\bbT^{(5)}:=\bbT_{\mu_1,\mu_2}\colon
	\begin{tikzpicture}[baseline={(current bounding box.center)}]
		\SetGraphUnit{1}
	    \GraphInit[vstyle=Welsh]
	    \renewcommand*{\VertexSmallMinSize}{2pt}
	    \SetVertexMath
	    \SetVertexLabel
   		\Vertex[a=0*72-126,d=1,Lpos=0*72-126,Ldist=0,L={\{x\}}]{x_1}
   		\Vertex[a=1*72-126,d=1,Lpos=1*72-126,Ldist=0,L={\{y\}}]{y_1}
   		\Vertex[a=2*72-126,d=1,Lpos=2*72-126,Ldist=0,L={\{u\}}]{x_2}
   		\Vertex[a=3*72-126,d=1,Lpos=3*72-126,Ldist=0,L={\{v\}}]{u}
   		\Vertex[a=4*72-126,d=1,Lpos=4*72-126,Ldist=0,L={\{w\}}]{v}
   		\Edges(x_1,y_1,x_2,u,v,x_1,x_2,v,y_1,u,x_1)
	    \AddVertexColor{black}{x_1,y_1}
	\end{tikzpicture}\quad \tilde{h}_{\mu_1,\mu_2}\colon\begin{cases}
		\{x\}\mapsto a,\\
		\{y\}\mapsto b,\\
		\{u\}\mapsto c,\\
		\{v\}\mapsto d,\\
		\{w\}\mapsto o.
	\end{cases}
	\]
	To sum up, we have
	\[
	T=\{\bbT^{(1)},\bbT^{(2)},\bbT^{(3)},\bbT^{(4)},\bbT^{(5)}\}
	\]
	Let us fix  an embedding $\kappa\colon\Delta\injto\Gamma$. Then the set $P_\kappa$ of all $\kappa$-compatible pairs is given by
	\[
	P_\kappa = \{ (\mu_1,\mu_2)\mid \mu_1\text{ extends $\kappa$ along $\iota_1$ and $\mu_2$ extends $\mu_1\circ e$ along $\iota_2$}\}.
	\]
	Thus, we have
	\[\#(\Gamma,\bbT_1)\cdot\#(\Gamma,\bbT_2)= |P_\kappa|=\sum_{i=1}^5 \#(\Gamma,\bbT^{(i)},\kappa).\]
	If we suppose that $\Gamma$ is $\bbT^{(1)}$-, $\bbT^{(2)}$-, and $\bbT^{(5)}$-regular, then, taking into account that $\bbT^{(2)}\cong\bbT^{(3)}$, we obtain
	\[
	\#(\Gamma,\bbT^{(4)},\kappa) = \#(\Gamma,\bbT_1)\cdot \#(\Gamma,\bbT_2) - \#(\Gamma,\bbT^{(1)}) - 2\cdot\#(\Gamma,\bbT^{(2)}) - \#(\Gamma,\bbT^{(5)}).
	\]
	Finally, observing that $\#(\Gamma,\bbT,\kappa)=\#(\Gamma,\bbT^{(4)},\kappa)$. we arrive at
	\[
	\#(\Gamma,\bbT,\kappa) = \#(\Gamma,\bbT_1)\cdot\#(\Gamma,\bbT_2) - \#(\Gamma,\bbT^{(1)}) - 2\cdot\#(\Gamma,\bbT^{(2)}) - \#(\Gamma,\bbT^{(5)}).
	\]
	As this does not depend on $\kappa$, we conclude that $\Gamma$ is $\bbT$-regular.
\end{example}
\begin{remark}
	The formulation of the type counting Lemma is not as strong as it could be.  In particular, when analyzing the proof it becomes clear that the third condition can be weakened. It is not necessary that $\Gamma$ is $\bbT$-regular for \emph{all} graph types $\bbT$ strictly dominated by $\bbT_1\oplus_e\bbT_2$. Instead it is sufficient to claim that $\Gamma$ is $\bbT$-regular for all those graph types $\bbT$ for which there exists a morphism $(f,g)\colon\bbT_1\oplus_e\bbT_2\epito\bbT$ such that 
	\begin{enumerate}
		\item $f$ is an isomorphism,
		\item $g$ is surjective and not an isomorphism,
		\item $g\circ\lambda_1$ and $g\circ\lambda_2$ are embeddings,
	\end{enumerate}
	where $(\lambda_1,\lambda_2)$ is a limiting cocone for $(e,\iota_2)$.  
\end{remark}

\subsection{Criteria for \texorpdfstring{$(m,n)$}{(m,n)}-regularity}
The proofs of the following propositions  make use of a very basic induction principle for finite posets:
\begin{lemma}\label{indprinc}
    Let $(P,\le)$ be a finite partially ordered set and let $B\subseteq P$. If
    \begin{equation}\label{induct}
\forall p\in P\,:\,\big(\{q\in P\mid q<p\}\subseteq B\Rightarrow p\in B\big),
    \end{equation}
    then  we already have that $B$ is equal to $P$. 
\end{lemma}
\begin{proof}
    Suppose that \eqref{induct} holds for $B$, but that $B\neq P$. Let $x$ be a minimal element of $P\setminus B$ in $(P,\le)$ (this exists because $P$ is finite). Then for all $y< x$ we have $y\in B$. Thus, by \eqref{induct}, we also have $x\in B$, a contradiction.
\end{proof}

\begin{proposition}\label{bigbase}
    Let $\Gamma$ be an $(m,m)$-regular graph. Then, $\Gamma$ is $(m,n)$-regular if and only if it is $(\lseq m,n)$-regular.
\end{proposition}
\begin{proof}
    By definition, from $(m,n)$-regularity follows $(\lseq m,n)$-regularity.
    
    Suppose that $\Gamma$ is $(\lseq m,n)$-regular and $(m,m)$-regular. Let $P$ be a transversal of the isomorphism classes of graph types of order $(k,l)$ for $k\le m$ and for $l\le n$. Then, by Lemma~\ref{fingraphtypes}, $(P,\preceq)$ is a finite poset. Moreover, whenever $\bbT\in P$ and $\bbT'\preceq\bbT$, then $\bbT'$ is isomorphic to an element of $P$.
        
    Let $\bbT=(\Delta,\iota,\Theta)\in P$ be of order  $(k,l)$.  Suppose  that for all $\bbT'\prec\bbT$ the graph  $\Gamma$ is $\bbT'$-regular.
If $l\le m$, then $\Gamma$ is $\bbT$-regular, by assumption. So suppose that $m<l\le n$. Let $\hat{\Delta}$ be an induced subgraph of order $m$ of $\Theta$ that contains the image of $\iota$, and let $\hat{\iota}$ be the identical embedding of $\hat\Delta$ into $\Theta$. Then $\bbT_1:=(\Delta,\iota,\hat{\Delta})$ is a graph type of order $(k,m)$, and  $\bbT_2:=(\hat{\Delta},\hat{\iota},\Theta)$ is a graph type of order $(m,l)$. Moreover, $\bbT\cong\bbT_1\oplus_{\hat{\iota}}\bbT_2$. By the assumptions, we have that $\Gamma$ is $\bbT_1$- and $\bbT_2$-regular. Hence, by the type counting lemma, we conclude that $\Gamma$ is $\bbT$-regular. 

 By the arguments above and by Lemma~\ref{indprinc}, $\Gamma$ is $\bbT$-regular for all graph types $\bbT$ from $P$. In other words, $\Gamma$ is $(m,n)$-regular.  
\end{proof}

Note that a graph is $(2,2)$-regular if and only if it is regular. Thus, the previous proposition generalizes a  classic result by A.V.~Ivanov:
\begin{theorem}[A.V.~Ivanov {\cite[Proposition 2.1]{Iva94}}]
 Let $\Gamma$ be a regular graph. Then $\Gamma$ satisfies the the $t$-vertex condition if and only if it is $(\lseq 2,t)$-regular. 
\end{theorem}

\begin{definition}
    A graph $\Gamma=(V,E)$ is called \emph{$\kk$-isoregular} if for every subset $X\subseteq V$ with $|X|\le \kk$ the number of vertices $v\notin X$ that are adjacent to all elements of $X$ does not depend on $X$ but only on the isomorphism type of the subgraph of $\Gamma$  induced by $X$. 
\end{definition}

\begin{proposition}
    Let $\Gamma$ be a graph and let $\kk>0$ be a natural number. Then the following are equivalent:
    \begin{enumerate}
    \item $\Gamma$ is $\kk$-isoregular,
    \item $\Gamma$ is $(\lseq l,\lseq l+1)$-regular for every $1\le l\le \kk$,
    \item $\Gamma$ is $(\kk,\kk+1)$-regular.
    \end{enumerate}
\end{proposition}
\begin{proof}
    ``(1)$\Rightarrow$(2):'' Let $l\in\{1,\dots,\kk\}$, and let $P$ be a transversal of the isomorphism classes of graph types of order $(l,m)$ where $m\in\{l,l+1\}$. Without loss of generality we may assume for every graph type $\bbT=(\Delta,\iota,\Theta)\in P$ that $\iota$ is the identical embedding (\ie $\Delta$ is an induced subgraph of $\Theta$). By Lemma~\ref{fingraphtypes}, $(P,\preceq)$  is a finite poset. Moreover, for every $\bbT\in P$ of order $(l,m)$ and for every $\bbT'\prec\bbT$ we have  that the order of $\bbT'$ is $(l,n)$ for some $l\le n\le m$; hence there exists a unique $\bbT''\in P$ such that $\bbT'\cong\bbT''$.
    
    In the following we  show that $\Gamma$ is $\bbT$-regular, for all $\bbT\in P$. Let $\bbT=(\Delta,\iota,\Theta)$ be an element of $P$. Moreover, suppose that for all $\bbT'\prec\bbT$ from $P$ the graph $\Gamma$ is $\bbT'$-regular. If the order of $\bbT$ is $(l,l)$, then $\Gamma$ is $\bbT$-regular. So suppose that $\bbT$ has order $(l,l+1)$. Let $v$ be the unique vertex of $\Theta$ that is not in $V(\Delta)$. If $v$ has valency $l$, then $\Gamma$ is $\bbT$-regular, because $\Gamma$  is $\kk$-isoregular. So, suppose that the valency of $v$ is equal to $m<l$. Let $\hat{\Delta}$ be the subgraph of $\Delta$ induced by the neighbors of $v$, let $\hat\Theta$ be the subgraph of $\Theta$ induced by the vertices of $\hat\Delta$  together with $v$ itself, and let $\hat{\iota}\colon\hat\Delta\injto\hat\Theta$ be the identical embedding. 
	Then $\bbT_1:=(\Delta,1_\Delta,\Delta)$ and $\bbT_2:=(\hat\Delta,\hat\iota,\hat\Theta)$ are graph types. Moreover, $\bbT\cong \bbT_1\oplus_{e}\bbT_2$, where $e$ denotes the identical embedding of $\hat\Delta$ into $\Delta$. 
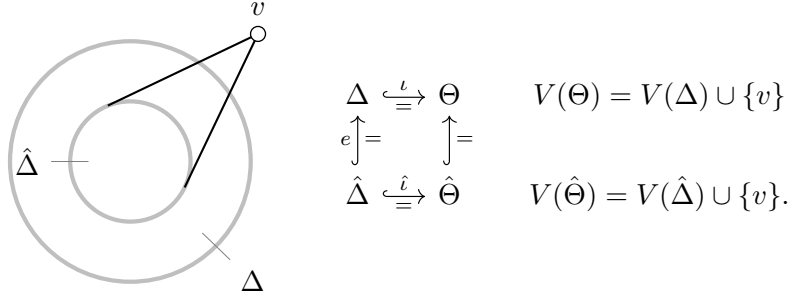
		\[
	\begin{tikzpicture}[scale=0.4,baseline={(current bounding box.center)}]
    \tkzDefPoints{0/0/Z,-2/0/X}
    \tkzDefPoint(45:6){A}
    \tkzDrawCircle[color=lightgray,ultra thick,R](Z,2cm)
    \tkzDrawCircle[color=lightgray,ultra thick,R](Z,4cm)
    \tkzTangent[from = A](Z,X)\tkzGetPoints{B}{C}
	\GraphInit[vstyle=Welsh]    
    \renewcommand*{\VertexSmallMinSize}{2pt}
    \SetVertexMath
	\Vertex[Node,Lpos=90,L={v}]{A}
	\Vertex[Node,empty]{B}
	\Edge(B)(A)\Edge(C)(A)
	\node [pin=180:$\hat\Delta$] at (-1,0){};
	\node [pin=-45:$\Delta$] at (2,-2){};
	\end{tikzpicture}\qquad
		\begin{tikzcd}
		\Delta \arrow[r,hook,"\iota", "="']& \Theta & V(\Theta)=V(\Delta)\cup\{v\}\\
		\hat\Delta\arrow[r,hook,"\hat\iota", "="']\arrow[u,hook,"e","="'] & \hat\Theta\arrow[u,hook,"="']	 & V(\hat\Theta)=V(\hat\Delta)\cup\{v\}.
	\end{tikzcd}
	\]
	Then $\bbT_1$ is of order $(l,l)$ thus, $\Gamma$ is $\bbT_1$-regular. Moreover, $\bbT_2$ is of order $(m,m+1)$ and the $\bbT_2$-regularity of $\Gamma$ follows from the $\kk$-isoregularity of $\Gamma$. Now, from the type counting lemma it follows that $\Gamma$ is $\bbT$-regular. Finally, from Lemma~\ref{indprinc} it follows that $\Gamma$ is regular for all types from $P$. In particular, $\Gamma$ is $(\lseq l,\lseq l+1)$-regular.

    ``(2)$\Rightarrow$(3):'' We show that $\Gamma$ is $(l,l+1)$-regular for all $l\in\{1,\dots,\kk\}$. We proceed by induction on $l$. For the induction base we note that  $\Gamma$ is $(1,2)$-regular if and only if it is $(\lseq 1,\lseq 1)$, and $(\lseq 1,\lseq 2)$-regular. The first regularity condition is trivially fulfilled and the $(\lseq 1,\lseq 2)$-regularity is given by assumption. Suppose, we know that $\Gamma$ is $(l,l+1)$-regular and $(\lseq l+1,\lseq l+2)$-regular, for some $1\le l\le \kk-1$. Then from the $(l,l+1)$-regularity follows immediately the $(l+1,l+1)$-regularity (indeed, a graph is $(l+1,l+1)$-regular iff it is $(l,l+1)$-regular and $(\lseq l+1,\lseq l+1)$-regular; however, trivially, every graph is $(\lseq l+1,\lseq l+1)$-regular). Moreover, we have that $\Gamma$ is $(\lseq l+1,l+2)$-regular, because $\Gamma$ is $(\lseq l+1,\lseq l+1)$-regular and $\Gamma$ is  $(\lseq l+1,\lseq l+2)$-regular. Hence, from Proposition~\ref{bigbase}, it follows that $\Gamma$ is $(l+1,l+2)$-regular. 
    
    ``(3)$\Rightarrow$(1):'' $\kk$-isoregularity of $\Gamma$ follows immediately from the $(\kk,\kk+1)$-regularity. 
\end{proof}

The following criterion by S.~Reichard  characterizes, when a $\kk$-isoregular
graph with the $(t-1)$-vertex condition satisfies the $t$-vertex condition:  
\begin{theorem}[{\cite[Theorem 3]{Reich00}}] Let $\Gamma$ be a $\kk$-isoregular graph that satisfies the $(t-1)$-vertex condition for $t>3$. Then, in order to verify the $t$-vertex condition, it suffices to test the $\mathbb{T}$-regularity for graph types $\mathbb{T}=(\Delta,\iota,\Theta)$ of order $(2,t)$ with the property that all vertices of $\Theta$ that are not in the image of $\iota$ have valency $\ge \kk+1$.
\end{theorem}
Our next goal is to generalize this result: 
\begin{proposition}\label{redlist} Let
$\Gamma$ be an  $(m,t)$-regular graph. Let $\mathcal{M}$ be a set of graph types and suppose that $\Gamma$ is $\bbT$-regular, for all $\bbT\in \cM$. Then, in order to verify the $(m,t+1)$-regularity of $\Gamma$
it suffices to test the $\mathbb{T}$-regularity for graph types of order $(m,t+1)$ that are $\widehat\bbT$-irreducible for all $\widehat\bbT\in\cM$. 
\end{proposition}
\begin{proof}
    Let $P$ be a transversal of the isomorphism classes of graph types of order $(m,t+1)$. Then, by Lemma~\ref{fingraphtypes},  $(P,\preceq)$ is a finite poset. Moreover, whenever $\bbT\in P$ and $\bbT'\preceq\bbT$ is a graph type of order $(m,t+1)$, then $\bbT'$ is isomorphic to an element of $P$.
    
      We will use the induction principle from Lemma~\ref{indprinc} on  $(P,\preceq)$: Let $\bbT=(\Delta,\iota,\Theta)\in P$ and suppose that $\Gamma$ is $\bbT'$-regular for all $\bbT'\in P$ with $\bbT'\prec\bbT$. Note that for every graph type $\bbT''\prec\bbT$ we either have that $\bbT''$ is isomorphic to an element of $P$ or it has order $(m,l)$ for some $l< t+1$. In both cases we conclude that $\Gamma$ is $\bbT''$-regular.
      
      If $\bbT$ is $\widehat\bbT$-irreducible for all $\widehat\bbT\in\cM$, then $\Gamma$ is $\bbT$-regular, by assumption. So suppose that there exists a $\widehat\bbT\in\cM$, such that $\bbT$ is $\widehat\bbT$-reducible. Then $\bbT\cong\bbT_1\oplus_e\widehat\bbT$ for some graph type $\bbT_1\ncong\bbT$. But then the order of $\bbT_1$ is $(m,l)$, for some $l<t+1$. Hence, by assumption $\Gamma$ is $\bbT_1$-regular and $\widehat\bbT$-regular. By the type counting lemma we obtain that $\Gamma$ is $\bbT$-regular. 
        
    Now, it remains to invoke Lemma~\ref{indprinc}, to obtain that $\Gamma$ is regular for all types from $P$. Consequently, $\Gamma$ is $(\lseq m,\lseq t+1)$-regular. By assumption, $\Gamma$ is $(m,t)$- and in particular $(m,m)$-regular. Hence, by Proposition~\ref{bigbase}, we have that  $\Gamma$ is $(m,t+1)$-regular. 
\end{proof}

\begin{proposition}\label{mncond}
    Let $\Gamma$ be a graph. Then $\Gamma$ is $(m,n+1)$-regular if and only if $\Gamma$ is   $(m,n)$-regular and it is $\bbT$-regular for every $(m,n+1)$-irreducible graph type $\bbT$ of order $(m,n+1)$.
\end{proposition}
\begin{proof}

``$\Rightarrow$:'' This is clear.

``$\Leftarrow$:''  Let $\cM$ be a transversal of the isomorphism classes of  graph types of order $(k,l)$, where $k\le m$ and where $l\le n$. By Assumption, $\Gamma$ is regular for all graph types from $\cM$. 
By Proposition~\ref{redlist}, in order to show that $\Gamma$ is $(m,n+1)$-regular it suffices to show that $\Gamma$ is $\bbT$-regular, for all graph types $\bbT$ of order $(m,n+1)$ that are $\widehat\bbT$-irreducible, for all $\widehat\bbT\in\cM$. 

It remains to show that a graph type $\bbT$ of order $(m,n+1)$ is $(m,n+1)$-reducible if and only if it is $\widehat\bbT$-reducible for some $\widehat\bbT\in\cM$. 

So suppose that $\bbT$ is $(m,n+1)$-reducible. Then $\bbT\cong\bbT_1\oplus_e\bbT_2$ such that $\bbT_2$ has order $(k,k')$ where $k\le m$ and $k'\le n+1$, and where neither $\bbT_1$ nor $\bbT_2$ is isomorphic to $\bbT$.  It follows that $\bbT$ is $\bbT_2$-reducible. Note that $\bbT_1$ has order $(m,k'')$, where $k''\le n$, for otherwise $\bbT_1\cong\bbT$. In particular, by assumption, $\Gamma$ is $\bbT_1$-regular. It remains to show that $\bbT_2$ is isomorphic to an element of $\cM$, or, in other words, that $k'\le n$.  Suppose on the contrary that $k'=n+1$.   Then $k''=k$ because $k'-k+k''=n+1$. Since $k\le m$ and $k''\ge m$, we conclude $k''=k=m$. Thus $\bbT_2$ has the same order as $\bbT$. It follows that $\bbT_2$ and $\bbT$ are isomorphic, a contradiction. Hence $k'\le n$.  

Suppose now that $\bbT$ is $\widehat\bbT$-reducible for some $\widehat\bbT\in\cM$. Then $\bbT\cong\bbT_1\oplus_e\widehat\bbT$, where $\bbT_1\ncong\bbT$. Suppose that the order of $\widehat\bbT$ is $(k,l)$. Then $k\le m$ and $l\le n$.  It follows that $\bbT\ncong\widehat\bbT$. From the definition of $(\lseq m,\lseq n+1)$-irreducibility, it follows that $\bbT$ is $(m,n+1)$-reducible. 
\end{proof}

\begin{corollary}\label{tvertcond}
	A graph $\Gamma$ is $(m,n+1)$-regular if and only if it is $(m,n)$-regular and it is $\bbT$-regular for all graph types $\bbT$ of order $(m,n+1)$ for which $\Cl(\bbT)$ is $(m+1)$-connected.
\end{corollary}
\begin{proof}
	This follows immediately from Proposition~\ref{mncond} in conjunction with Lemma~\ref{connectcond}.
\end{proof}

\begin{definition}
    Let $\Gamma$ be a graph and let $v\in V(\Gamma)$. Then with $\Gamma_1(v)$ we  denote the subgraph of $\Gamma$ induced by the neighbors of $v$. Moreover, with $\Gamma_2(v)$ we denote the subgraph of $\Gamma$ induced by the non-neighbors of $v$  (except $v$ itself). 
    $\Gamma_1(v)$ and $\Gamma_2(v)$ are called the \emph{first} and the \emph{second subconstituent} of $\Gamma$ with respect to $v$, respectively.
\end{definition}
The following proposition relates the regularities of a graph with the regularities of its subconstituents. This is used later on to identify a new class of graphs satisfying  the $6$-vertex condition:
\begin{proposition}\label{constituents}
    Let $\Gamma$ be an $(m,n)$-regular graph where $m\ge 1$, and let $v\in V(\Gamma)$. Then $\Gamma_1(v)$ and $\Gamma_2(v)$ are both $(m-1,n-1)$-regular. 
\end{proposition}
\begin{proof}
	About  $\Gamma_1(v)$: Let $\bbT=(\Delta,\iota,\Theta)$ be a graph type of order $(r,s)$ where $r\le m-1$ and $s\le n-1$. Let $\Delta':=\Delta+\{x\}$ and $\Theta':=\Theta+\{y\}$ be graphs obtained from $\Delta$ and $\Theta$ by adjoining a single new vertex that is connected to all old vertices, respectively. Let $\iota'\colon\Delta'\injto\Theta'$ be defined according to
	\[
	\iota'\colon  w\mapsto\begin{cases}
		\iota(w) & w\in V(\Delta)\\
		y & w=x.
	\end{cases}
	\]
	Then $\bbT':=(\Delta',\iota',\Theta')$ is a graph type of order $(r+1,s+1)$. As $r+1\le m$ and $s+1\le n$, we have that $\Gamma$ is $\bbT'$-regular. Let $\kappa\colon\Delta\injto\Gamma_1(v)$. Define $\kappa'\colon\Delta'\injto\Gamma$ according to
	\[
	\kappa'\colon w\mapsto\begin{cases}
		\kappa(w) & w\in V(\Delta),\\
		v & w=x.
	\end{cases}
	\]
	We claim that there is a bijection between set of extensions of $\kappa$ along $\iota$ in $\Gamma_1(v)$ and the set of extensions of $\kappa'$ along $\iota'$ in $\Gamma$:
	
	Let $\hat\kappa$ be any extension of $\kappa$ along $\iota$ in $\Gamma_1(v)$. We define  $\hat\kappa'\colon\Theta'\injto\Gamma$ according to 
	\[
	\hat\kappa'\colon w\mapsto\begin{cases}
		\hat\kappa(w) & w\in V(\Theta),\\
		v & w=y.
	\end{cases}
	\]
	Clearly, $\hat\kappa'$ is an extension of $\kappa'$ along $\iota'$ in $\Gamma$. 
	
	Let on the other hand $\hat\kappa'$ be any extension of $\kappa'$ along $\iota'$ in $\Gamma$. Then $\hat\kappa:=\hat\kappa'\restr_{V(\Theta)}$ is an extension of $\kappa$ along $\iota$ in $\Gamma_1(v)$. This establishes the desired bijection between extensions of $\kappa$ along $\iota$ in $\Gamma_1(v)$ and extensions of $\kappa'$ along $\iota'$ in $\Gamma$. In particular, we  have $\#(\Gamma_1(v), \bbT,\kappa)= \#(\Gamma, \bbT',\kappa')=\#(\Gamma,\bbT')$. Thus, $\Gamma_1(v)$ is $\bbT$-regular. As $\bbT$ was chosen arbitrarily, we conclude that $\Gamma_1(v)$ is $(m-1,n-1)$-regular.
	
	About $\Gamma_2(v)$: From Lemma~\ref{complreg} it follows that $\overline{\Gamma}$ is $(m,n)$-regular. Clearly, we have $\overline{\Gamma}_1(v)=\overline{\Gamma_2(v)}$. From the first part of the proof it follows that $\overline{\Gamma}_1(v)$ is $(m-1,n-1)$-regular. Again using Lemma~\ref{complreg} we conclude that $\Gamma_2(v)$ is $(m-1,n-1)$-regular. 
\end{proof}
\begin{remark}
	The previous proposition generalizes a well-known result from algebraic graph theory to the case of $(m,n)$-regular graphs. Namely, if a graph is $\kk+1$-isoregular, then all subconstituents are $\kk$-isoregular. This observation, together with spectral methods, stands at the center of  Gol'fand's lost proof that $5$-isoregular graphs are homogeneous (see \cite[Section 9.2]{Rei15} for a historical account and for further references). 
\end{remark}

\section{Checking the \texorpdfstring{$t$}{t}-vertex condition} 
Every graph satisfies the 1-vertex condition. A graph satisfies the 2-vertex condition if and only if it is regular. A bit less obvious but rather straight forward is the observation that a graph satisfies the 3-vertex condition if and only if it is strongly regular, \ie it is regular and the number of joint neighbors of every edge is equal to a constant $\lambda$ and the number of joint neighbors of every non-edge is equal to a constant $\mu$. A criterion for the $4$-vertex condition is given by:
\begin{theorem}[M.D.~Hestenes, D.G.~Higman \cite{HesHig71}]\label{HesHig}
Let $\Gamma$ be a strongly regular graph. Then, in order to verify the $4$-vertex condition it suffices to test the $\mathbb{T}$-regularity for the following two graph types of order $(2,4)$:
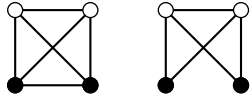
\[
\begin{tikzpicture}
\SetGraphUnit{1}
    \GraphInit[vstyle=Welsh]
    \renewcommand*{\VertexSmallMinSize}{2pt}
    \SetVertexMath
    \Vertex[Lpos=-135,NoLabel]{x}
    \EA[Lpos=-45,NoLabel](x){y}
    \NO[NoLabel](x){u}
    \NO[NoLabel](y){v}
    \Edges(v,x,u,v,y,u)
    \Edges(x,y)
    \AddVertexColor{black}{x,y}
\end{tikzpicture}\qquad
\begin{tikzpicture}
\SetGraphUnit{1}
    \GraphInit[vstyle=Welsh]
    \renewcommand*{\VertexSmallMinSize}{2pt}
    \SetVertexMath
    \Vertex[Lpos=-135,NoLabel]{x}
    \EA[Lpos=-45,NoLabel](x){y}
    \NO[NoLabel](x){u}
    \NO[NoLabel](y){v}
    \Edges(v,x,u,v,y,u)
    \AddVertexColor{black}{x,y}
\end{tikzpicture}
\]
\end{theorem}
In our terminology, this is a special case of Corollary~\ref{tvertcond} ($m=2$ and $n=3$, cf.~Example~\ref{exHesHig}). More generally, we have:
\begin{proposition}\label{cond-threecon}
    Let $\Gamma$ be a graph that satisfies the $t$-vertex condition for $t\ge 3$. Then, in order to verify the $(t+1)$-vertex condition it suffices to test the $\bbT$-regularity for all those graph types $\bbT$ of order $(2,t+1)$ for which $\Cl(\bbT)$ is $3$-connected.
\end{proposition}
\begin{proof}
	This is a special case of Corollary~\ref{tvertcond} ($m=2$, $n=t$).
\end{proof}

In \cite[Theorem 4.9]{Reich03} S.~Reichard proved  that a graph satisfying the $4$-vertex condition  satisfies the $5$-vertex condition if and only if it is regular for a list of $16$ graph types. The following proposition reduces the number of graph types to be tested to $10$:
\begin{proposition}\label{five-cond} 
	Given a graph $\Gamma$ that fulfills the $4$-vertex condition.  Then in order to test whether $\Gamma$ satisfies also the $5$-vertex condition it suffices to count the graph types in the table below.  
    
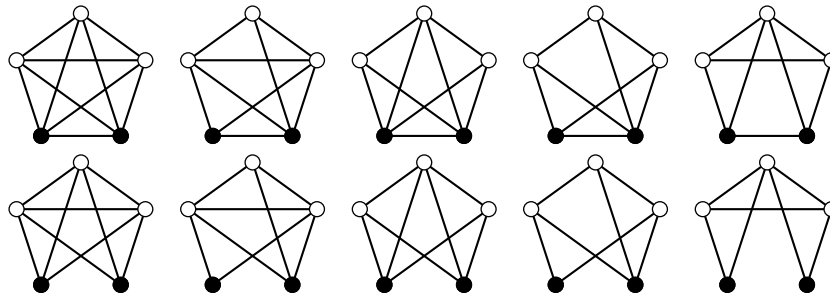
\[ 
    \begin{matrix}
    \begin{tikzpicture}[rotate=-126,scale=0.3]
        \SetGraphUnit{1}
        \GraphInit[vstyle=Welsh]
        \renewcommand*{\VertexSmallMinSize}{2pt}
        \SetVertexMath
        \Vertices[unit=3,NoLabel]{circle}{A,B,C,D,E}
        \Edges(B,C,D,E,A,C,E,B,D,A)
        \Edge(A)(B)
        \AddVertexColor{black}{A,B}
    \end{tikzpicture}&
    \begin{tikzpicture}[rotate=-126,scale=0.3]
        \SetGraphUnit{1}
        \GraphInit[vstyle=Welsh]
        \renewcommand*{\VertexSmallMinSize}{2pt}
        \SetVertexMath
        \Vertices[unit=3,NoLabel]{circle}{A,C,D,B,E}
        \Edges(B,C,D,E,A)
        \Edges(B,E,C)
        \Edges(A,D,B)
        \Edge(A)(C)
        \AddVertexColor{black}{A,C}
    \end{tikzpicture}&
    \begin{tikzpicture}[rotate=-126,scale=0.3]
        \SetGraphUnit{1}
        \GraphInit[vstyle=Welsh]
        \renewcommand*{\VertexSmallMinSize}{2pt}
        \SetVertexMath
        \Vertices[unit=3,NoLabel]{circle}{C,D,A,E,B}
        \Edge(B)(C)
        \Edges(D,E,A)
        \Edges(B,E,C)
        \Edges(A,D,B)
        \Edge(A)(C)
        \Edge(C)(D)
        \AddVertexColor{black}{C,D}
    \end{tikzpicture}&
    \begin{tikzpicture}[rotate=-126,scale=0.3]
        \SetGraphUnit{1}
        \GraphInit[vstyle=Welsh]
        \renewcommand*{\VertexSmallMinSize}{2pt}
        \SetVertexMath
        \Vertices[unit=3,NoLabel]{circle}{A,E,C,B,D}
        \Edges(B,C,E,B)
        \Edges(B,D,E)
        \Edge(A)(D)
        \Edge(A)(C)
        \Edge(A)(E)
        \AddVertexColor{black}{A,E}
    \end{tikzpicture}&
    \begin{tikzpicture}[rotate=-126,scale=0.3]
        \SetGraphUnit{1}
        \GraphInit[vstyle=Welsh]
        \renewcommand*{\VertexSmallMinSize}{2pt}
        \SetVertexMath
        \Vertices[unit=3,NoLabel]{circle}{A,C,B,E,D}
        \Edges(B,C,E,B)
        \Edges(B,D,E)
        \Edge(A)(D)
        \Edge(A)(C)
        \Edge(A)(E)
        \AddVertexColor{black}{A,C}
    \end{tikzpicture}\\
    \begin{tikzpicture}[rotate=-126,scale=0.3]
        \SetGraphUnit{1}
        \GraphInit[vstyle=Welsh]
        \renewcommand*{\VertexSmallMinSize}{2pt}
        \SetVertexMath
        \Vertices[unit=3,NoLabel]{circle}{A,B,C,D,E}
        \Edges(B,C,D,E,A,C,E,B,D,A)
        \AddVertexColor{black}{A,B}
    \end{tikzpicture}&
    \begin{tikzpicture}[rotate=-126,scale=0.3]
        \SetGraphUnit{1}
        \GraphInit[vstyle=Welsh]
        \renewcommand*{\VertexSmallMinSize}{2pt}
        \SetVertexMath
        \Vertices[unit=3,NoLabel]{circle}{A,C,D,B,E}
        \Edges(B,C,D,E,A)
        \Edges(B,E,C)
        \Edges(A,D,B)
        \AddVertexColor{black}{A,C}
    \end{tikzpicture}&
    \begin{tikzpicture}[rotate=-126,scale=0.3]
        \SetGraphUnit{1}
        \GraphInit[vstyle=Welsh]
        \renewcommand*{\VertexSmallMinSize}{2pt}
        \SetVertexMath
        \Vertices[unit=3,NoLabel]{circle}{C,D,A,E,B}
        \Edge(B)(C)
        \Edges(D,E,A)
        \Edges(B,E,C)
        \Edges(A,D,B)
        \Edge(A)(C)
        \AddVertexColor{black}{C,D}
    \end{tikzpicture}&
    \begin{tikzpicture}[rotate=-126,scale=0.3]
        \SetGraphUnit{1}
        \GraphInit[vstyle=Welsh]
        \renewcommand*{\VertexSmallMinSize}{2pt}
        \SetVertexMath
        \Vertices[unit=3,NoLabel]{circle}{A,E,C,B,D}
        \Edges(B,C,E,B)
        \Edges(B,D,E)
        \Edge(A)(D)
        \Edge(A)(C)
        \AddVertexColor{black}{A,E}
    \end{tikzpicture}&
    \begin{tikzpicture}[rotate=-126,scale=0.3]
        \SetGraphUnit{1}
        \GraphInit[vstyle=Welsh]
        \renewcommand*{\VertexSmallMinSize}{2pt}
        \SetVertexMath
        \Vertices[unit=3,NoLabel]{circle}{A,C,B,E,D}
        \Edges(B,C,E,B)
        \Edges(B,D,E)
        \Edge(A)(D)
        \Edge(A)(E)
        \AddVertexColor{black}{A,C}
    \end{tikzpicture}
    \end{matrix} 
\]
\end{proposition}
\begin{proof}
	According to Proposition~\ref{cond-threecon}, $\Gamma$ satisfies the $5$-vertex condition if and only if it is $\bbT$-regular for all $\bbT$ such that $\Cl(\bbT)$ is $3$-connected. So we start by enumerating all $3$-connected graphs of order $5$. While this task is routine for a human, it is trivial for a computer. Here and below we  used the \texttt{geng}-utility from the package \emph{nauty and traces} (\cf \cite{McKPip14}) in conjunction with \emph{GAP} (\cf \cite{GAP4}) and \emph{GRAPE} (\cf \cite{GRAPE}). This gives us the following three graphs:
    
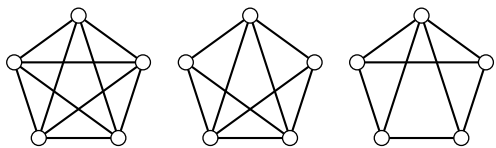
\[ 
    \begin{matrix}
    \begin{tikzpicture}[rotate=-126,scale=0.3]
        \SetGraphUnit{1}
        \GraphInit[vstyle=Welsh]
        \renewcommand*{\VertexSmallMinSize}{2pt}
        \SetVertexMath
        \Vertices[unit=3,NoLabel]{circle}{A,B,C,D,E}
        \Edges(B,C,D,E,A,C,E,B,D,A)
        \Edge(A)(B)
    \end{tikzpicture}&
    \begin{tikzpicture}[rotate=-126,scale=0.3]
        \SetGraphUnit{1}
        \GraphInit[vstyle=Welsh]
        \renewcommand*{\VertexSmallMinSize}{2pt}
        \SetVertexMath
        \Vertices[unit=3,NoLabel]{circle}{C,D,A,E,B}
        \Edge(B)(C)
        \Edges(D,E,A)
        \Edges(B,E,C)
        \Edges(A,D,B)
        \Edge(A)(C)
        \Edge(C)(D)
    \end{tikzpicture}&
    \begin{tikzpicture}[rotate=-126,scale=0.3]
        \SetGraphUnit{1}
        \GraphInit[vstyle=Welsh]
        \renewcommand*{\VertexSmallMinSize}{2pt}
        \SetVertexMath
        \Vertices[unit=3,NoLabel]{circle}{A,C,B,E,D}
        \Edges(B,C,E,B)
        \Edges(B,D,E)
        \Edge(A)(D)
        \Edge(A)(C)
        \Edge(A)(E)
    \end{tikzpicture}
    \end{matrix}
\]
  Next, for each graph $\Theta$ from this list we computed the orbits of $\Aut(\Theta)$ in its action on edges. Each orbit representative corresponds to a two-vertex subgraph $\Delta\cong K_2$, producing a graph type $(\Delta,\iota,\Theta)$ (as usually, $\iota$ is the identical embedding). This produces the upper row of graph types. The lower row is obtained by removing the distinguished edge in each case. Clearly, this produces a transversal of the isomorphism classes of graph types $\bbT$ of order $(2,5)$ for which $\Cl(\bbT)$ is $3$-connected.
\end{proof}

\section{Point graphs of partial quadrangles}
An incidence structure is a triple $(\mathscr{P},\mathscr{L},I)$, where $\mathscr{P}$ is a set of points (denoted by capital Latin letters $P,Q,\ldots$), $\mathscr{L}$ is a set of lines (denoted by small Latin letters $l,s,t,\ldots$), and $I\subseteq\mathscr{P}\times\mathscr{L}$ is an incidence relation. The elements of $I$ are called \emph{flags} and the elements of $(\mathscr{P}\times\mathscr{L})\setminus I$ are called \emph{antiflags}. A point $P$ is called incident with a line $l$ if $(P,l)$ is a flag. Slightly abusing the notation we write in this case $P\in l$. Two distinct flags $(P,p)$ and $(Q,q)$ are called \emph{collinear} if $p=q$, and \emph{concurrent} if $P=Q$. Two distinct points $P$ and $Q$ are called collinear if there exists a line $l$ such that $(P,l)$ and $(Q,l)$ are flags. In this case we say that $l$ goes through $P$ and $Q$. Dually, we say that two lines $p$ and $q$ are intersecting each other if there is a point $P$ such that $P\in p$ and $P\in q$.

For every incidence structure we may define its \emph{point graph}. This is a simple graph which has as vertices the points of the incidence structure such that between two points there is an edge if and only if the points are collinear.

In the following we restrict our attention to so-called partial linear spaces:
\begin{definition}
	A \emph{partial linear space of order $(s,t)$} (short $\pls(s,t)$) is an incidence structure $(\mathscr{P},\mathscr{L},I)$ with the following properties:
	\begin{enumerate}[label=PLS\arabic*.,ref=PLS\arabic*]	
		\item every line is incident with the same number $s+1$ of points, 
		\item every point is incident with the same number $t+1$ of lines, 
		\item\label{gq3} through any two distinct points goes at most one line, 
	\end{enumerate} 
\end{definition}
 If two lines $p$ and $q$ of a partial linear space intersect each other, then we denote the unique point of intersection by $p\cap q$.

We are interested in partial linear spaces because there are interesting classes of them whose point graphs are strongly regular graphs. Two such classes are defined below. The first class of interest is the one of generalized quadrangles as introduced by J.~Tits in \cite{Tit59}: 
\begin{definition}
	A \emph{generalized quadrangle} of order $(s,t)$ (short $\GQ(s,t)$) is a partial linear space of order $(s,t)$ with the the following additional property:
	\begin{enumerate}[label=GQ\arabic*.,ref=GQ\arabic*]	
\item\label{gq4} for every antiflag $(P,q)$ there is a unique point $Q$ such that $P$ and $Q$ are collinear and such that $Q\in q$.  
	\end{enumerate}
\end{definition}

It is well-known that the point graph of a generalized quadrangle of order $(s,t)$ is strongly regular with parameters $(v,k,\lambda,\mu)$ where 
\begin{align*} 
	v&=(s+1)(st+1) & k&=s(t+1) & \lambda&=s-1 &
	\mu&=t+1 
\end{align*} 
Axiom \ref{gq4} ensures that a generalized quadrangle does not contain triples of pairwise collinear points that are not all three on one line. From this follows that every set of points that induces a clique in the point graph is a subset of some line. In particular, the generalized quadrangle can be reconstructed from its point graph up to isomorphism by taking as points the vertices of the point graph, as lines the maximal cliques and as incidence relation the $\in$-relation. Moreover, the point graph of a generalized quadrangle cannot contain $K_4-e$ as an induced subgraph because this would imply the existence of two distinct maximal cliques that intersect in at least two points which cannot happen because of axiom \ref{gq3}.

In \cite{Cam75} P.~J.~Cameron examined point graphs of generalized quadrangles and made the above observations.  These observation lead him to study strongly regular graphs that do not contain $K_4-e$. It turns out that such graphs always arise as point graph of certain partial linear spaces. The class of partial linear spaces that have as a point graph an srg without $K_4-e$ as induced subgraph, are called \emph{partial quadrangles}. Below we give an axiomatization:
\begin{definition}
	A \emph{partial quadrangle}  with parameters $(s,t,\mu)$ (short $\PQ(s,t,\mu)$) is a partial linear space $(\mathscr{P},\mathscr{L},I)$ of order $(s,t)$ with the following properties:
	\begin{enumerate}[label=PQ\arabic*.,ref=PQ\arabic*]	
		\item if three points are pairwise collinear, then they are all three on one line, 
	\item for every pair $(P,Q)$ of non-collinear points there exist $\mu$ points $X$ that are collinear with both points $P$ and $Q$.  
	\end{enumerate}	
\end{definition}
Clearly, every generalized quadrangle is a partial quadrangle. However, there are examples of partial quadrangles that are not generalized quadrangles. For instance, the Petersen graph is the point graph of a partial quadrangle but not of a generalized quadrangle. The papers \cite{BamDeCDur11,Cos17} can  be used as a starting point to get an overview on the known constructions of partial quadrangles.

\begin{theorem}[P.~J.~Cameron {\cite[Theorem 2]{Cam75}}]
Let $\Gamma=(V,E)$ be a strongly regular graph with parameters $(v,k,\lambda,\mu)$. Then $\Gamma$ is isomorphic to the point graph of a partial quadrangle if and only if $\mu>0$ and it does not contain any induced subgraph isomorphic to $K_4-e$.
\end{theorem}
Let us recall that starting from a strongly regular graph $\Gamma$ with parameters $(v,k,\lambda,\mu)$ that has no induced subgraph isomorphic to $K_4-e$, we can construct a partial quadrangle by taking as points the vertices of $\Gamma$ and as lines the maximal cliques. The resulting partial quadrangle has parameters $(\lambda+1,\frac{k}{\lambda+1}-1,\mu)$. On the other hand, the parameters of the point graph of a $\PQ(s,t,\tilde\mu)$ are
\begin{align*}
	v&=\frac{s(t+1)(\mu+st)}{\mu}+1, & k&=s(t+1), & \lambda &= s-1, & \mu&=\tilde\mu.
\end{align*}

Now we are ready to formulate the first result of this section: 
\begin{theorem}\label{pqfivevert}
   Let $\Gamma$ be  the point graph of a partial quadrangle. Then $\Gamma$ satisfies the $5$-vertex condition.  
\end{theorem} 
\begin{proof}
	At first we note that by Theorem~\ref{HesHig}, in order to test the $4$-vertex condition for $\Gamma$ it is enough to test it for 
	\[
	\bbT_1\colon 
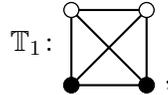
\begin{tikzpicture}[baseline={(current bounding box.center)}]
\SetGraphUnit{1}
    \GraphInit[vstyle=Welsh]
    \renewcommand*{\VertexSmallMinSize}{2pt}
    \SetVertexMath
    \Vertex[Lpos=-135,NoLabel]{x}
    \EA[Lpos=0,L={,}](x){y}
    \NO[NoLabel](x){u}
    \NO[NoLabel](y){v}
    \Edges(v,x,u,v,y,u)
    \Edges(x,y)
    \AddVertexColor{black}{x,y}
\end{tikzpicture}
\]
	as $\Gamma$ does not contain $K_4-e$ as an induced subgraph. Clearly, we have
	\[
	\#(\Gamma,\bbT_1) = (s-1)(s-2).
	\]
	Secondly we note that from all the graph types given in Proposition~\ref{five-cond} only the underlying graph of the first one  does not contain $K_4-e$ as an induced subgraph. Thus, in order to test the $5$-vertex condition, we have only to consider 
	\[
	    \bbT_2\colon 
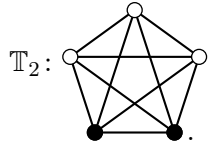
\begin{tikzpicture}[rotate=-126,scale=0.3,baseline={(current bounding box.center)}]
        \SetGraphUnit{1}
        \GraphInit[vstyle=Welsh]
        \renewcommand*{\VertexSmallMinSize}{2pt}
        \SetVertexMath
        \Vertices[unit=3,NoLabel]{circle}{A,B,C,D,E}
        \Edges(B,C,D,E,A,C,E,B,D,A)
        \Edge(A)(B)
        \AddVertexColor{black}{A,B}
    \end{tikzpicture}\raisebox{-5.5ex}{\hspace*{-0.7em}.}
	\]
	However, we easily compute
	\[
	 \#(\Gamma,\bbT_2) = (s-1)(s-2)(s-3).\qedhere
	\]
\end{proof}
Let us also have a look at a criterion  for the $6$-vertex condition for partial quadrangles: 
\begin{proposition}\label{pq6vc}
	Let $\Gamma$ be the point graph of a partial quadrangle. Then in order to test the $6$-vertex condition for $\Gamma$ it suffices to check it for the following $8$ graph types: 
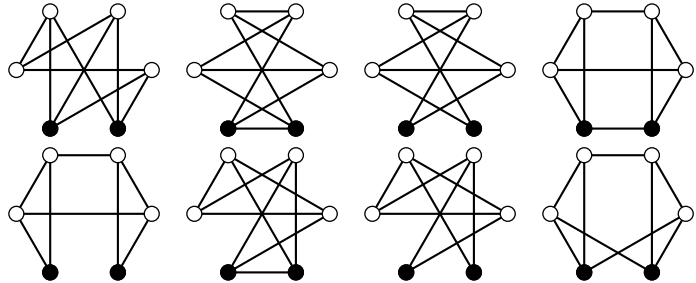
\[
\begin{matrix} 
    \begin{tikzpicture}[rotate=0,scale=0.3]
        \SetGraphUnit{1}
        \GraphInit[vstyle=Welsh]
        \renewcommand*{\VertexSmallMinSize}{2pt}
        \SetVertexMath
        \Vertices[unit=3,NoLabel]{circle}{A,B,C,D,E,F}
        \Edges(A,F,B,E,C,D)
        \Edge(C)(F)
        \Edge(A)(E)
        \Edge(B)(D)
        \Edge(D)(A)
        \AddVertexColor{black}{E,F}
    \end{tikzpicture}&
    \begin{tikzpicture}[rotate=-60,scale=0.3]
        \SetGraphUnit{1}
        \GraphInit[vstyle=Welsh]
        \renewcommand*{\VertexSmallMinSize}{2pt}
        \SetVertexMath
        \Vertices[unit=3,NoLabel]{circle}{A,B,C,F,E,D}
        \Edges(A,F,B,E,C,D)
        \Edge(C)(F)
        \Edge(A)(E)
        \Edge(B)(D)
        \Edge(D)(A)
        \AddVertexColor{black}{D,A}
    \end{tikzpicture}&
    \begin{tikzpicture}[rotate=-60,scale=0.3]
        \SetGraphUnit{1}
        \GraphInit[vstyle=Welsh]
        \renewcommand*{\VertexSmallMinSize}{2pt}
        \SetVertexMath
        \Vertices[unit=3,NoLabel]{circle}{A,B,C,F,E,D}
        \Edges(A,F,B,E,C,D)
        \Edge(C)(F)
        \Edge(A)(E)
        \Edge(B)(D)
        \AddVertexColor{black}{D,A}
    \end{tikzpicture}&    
    \begin{tikzpicture}[rotate=-60,scale=0.3]
        \SetGraphUnit{1}
        \GraphInit[vstyle=Welsh]
        \renewcommand*{\VertexSmallMinSize}{2pt}
        \SetVertexMath
        \Vertices[unit=3,NoLabel]{circle}{A,E,D,C,B,F}
        \Edges(F,B,E)
        \Edge(C)(D)
        \Edge(C)(F)
        \Edge(A)(E)
        \Edge(D)(E)
        \Edge(B)(C)
        \Edge(D)(A)
        \Edge(A)(F)
        \AddVertexColor{black}{F,A}
    \end{tikzpicture}\\
    \begin{tikzpicture}[rotate=-60,scale=0.3]
        \SetGraphUnit{1}
        \GraphInit[vstyle=Welsh]
        \renewcommand*{\VertexSmallMinSize}{2pt}
        \SetVertexMath
        \Vertices[unit=3,NoLabel]{circle}{A,E,D,C,B,F}
        \Edges(F,B,E)
        \Edge(C)(D)
        \Edge(C)(F)
        \Edge(A)(E)
        \Edge(D)(E)
        \Edge(B)(C)
        \Edge(D)(A)
        \AddVertexColor{black}{F,A}
    \end{tikzpicture}&
    \begin{tikzpicture}[rotate=-120,scale=0.3]
        \SetGraphUnit{1}
        \GraphInit[vstyle=Welsh]
        \renewcommand*{\VertexSmallMinSize}{2pt}
        \SetVertexMath
        \Vertices[unit=3,NoLabel]{circle}{A,D,F,E,C,B}
        \Edges(F,B,E)
        \Edge(C)(D)
        \Edge(C)(F)
        \Edge(A)(E)
        \Edge(D)(E)
        \Edge(B)(C)
        \Edge(D)(A)
        \Edge(A)(F)
        \AddVertexColor{black}{A,D}
    \end{tikzpicture}&
    \begin{tikzpicture}[rotate=-120,scale=0.3]
        \SetGraphUnit{1}
        \GraphInit[vstyle=Welsh]
        \renewcommand*{\VertexSmallMinSize}{2pt}
        \SetVertexMath
        \Vertices[unit=3,NoLabel]{circle}{A,D,F,E,C,B}
        \Edges(F,B,E)
        \Edge(C)(D)
        \Edge(C)(F)
        \Edge(A)(E)
        \Edge(D)(E)
        \Edge(B)(C)
        \Edge(A)(F)
        \AddVertexColor{black}{A,D}
    \end{tikzpicture}&
    \begin{tikzpicture}[rotate=0,scale=0.3]
        \SetGraphUnit{1}
        \GraphInit[vstyle=Welsh]
        \renewcommand*{\VertexSmallMinSize}{2pt}
        \SetVertexMath
        \Vertices[unit=3,NoLabel]{circle}{B,C,D,A,E,F}
        \Edges(F,B,E)
        \Edge(C)(D)
        \Edge(C)(F)
        \Edge(A)(E)
        \Edge(D)(E)
        \Edge(B)(C)
        \Edge(D)(A)
        \Edge(A)(F)
        \AddVertexColor{black}{E,F}
    \end{tikzpicture}
    \end{matrix}
\]
\end{proposition}
\begin{proof}
	The above given $8$ graph types form a transversal of the isomorphism classes of  all those graph types $\bbT=(\Delta,\iota,\Theta)$ of order $(2,6)$ for which $\Cl(\bbT)$ is $3$-connected and for which $\Theta$ does not contain an induced subgraph isomorphic to $K_4-e$. Now the claim follows from Theorem~\ref{pqfivevert} together with Proposition~\ref{cond-threecon}.  
\end{proof}

In the special case of generalized quadrangles S.~Reichard showed in \cite{Reich03}  among these $8$ graph types there are $5$ types $\mathbb{T}$ such that the point graph of every generalized quadrangle is $\mathbb{T}$-regular. Together with this observation we obtain:
\begin{proposition}
	The point graph of a generalized quadrangle satisfies the $6$-vertex condition if and only if it is regular for the following graph types:
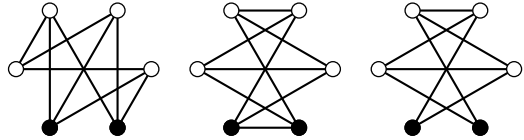
\[
    \begin{tikzpicture}[rotate=0,scale=0.3]
        \SetGraphUnit{1}
        \GraphInit[vstyle=Welsh]
        \renewcommand*{\VertexSmallMinSize}{2pt}
        \SetVertexMath
        \Vertices[unit=3,NoLabel]{circle}{A,B,C,D,E,F}
        \Edges(A,F,B,E,C,D)
        \Edge(C)(F)
        \Edge(A)(E)
        \Edge(B)(D)
        \Edge(D)(A)
        \AddVertexColor{black}{E,F}
    \end{tikzpicture}\quad
    \begin{tikzpicture}[rotate=-60,scale=0.3]
        \SetGraphUnit{1}
        \GraphInit[vstyle=Welsh]
        \renewcommand*{\VertexSmallMinSize}{2pt}
        \SetVertexMath
        \Vertices[unit=3,NoLabel]{circle}{A,B,C,F,E,D}
        \Edges(A,F,B,E,C,D)
        \Edge(C)(F)
        \Edge(A)(E)
        \Edge(B)(D)
        \Edge(D)(A)
        \AddVertexColor{black}{D,A}
    \end{tikzpicture}\quad
    \begin{tikzpicture}[rotate=-60,scale=0.3]
        \SetGraphUnit{1}
        \GraphInit[vstyle=Welsh]
        \renewcommand*{\VertexSmallMinSize}{2pt}
        \SetVertexMath
        \Vertices[unit=3,NoLabel]{circle}{A,B,C,F,E,D}
        \Edges(A,F,B,E,C,D)
        \Edge(C)(F)
        \Edge(A)(E)
        \Edge(B)(D)
        \AddVertexColor{black}{D,A}
    \end{tikzpicture}
\]
\end{proposition}
\begin{proof}
	By Proposition~\ref{pq6vc}, in order to proof the claim, we need to show that $\Gamma$ is regular for the following graph-types (to get a better understanding, we depict the types not as graphs but as geometrical configurations):
	
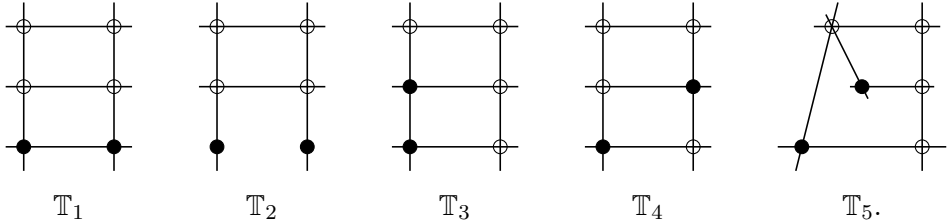
\[
	\begin{matrix}
	    \begin{tikzpicture}[scale=0.8]
	        \tkzDefPoint(0,0){a}
	        \tkzDefPoint(0,1){b}
	        \tkzDefPoint(0,2){c}
	        \tkzDefPoint(1.5,0){t}
	        \tkzDefPoint(1.5,1){u}
	        \tkzDefPoint(1.5,2){v}
\tkzDrawPoints[shape=circle,fill=black,size=13](a,t)
\tkzDrawPoints[shape=circle,fill=white,size=13](b,c,u,v)
	        \tkzDrawLine(a,c)
	        \tkzDrawLine(t,v)
	        \tkzDrawLine(c,v)
	        \tkzDrawLine(b,u)
	        \tkzDrawLine(a,t)
	    \end{tikzpicture} & 
	    \begin{tikzpicture}[scale=0.8]
	        \tkzDefPoint(0,0){a}
	        \tkzDefPoint(0,1){b}
	        \tkzDefPoint(0,2){c}
	        \tkzDefPoint(1.5,0){t}
	        \tkzDefPoint(1.5,1){u}
	        \tkzDefPoint(1.5,2){v}
\tkzDrawPoints[shape=circle,fill=black,size=13](a,t)
\tkzDrawPoints[shape=circle,fill=white,size=13](b,c,u,v)
	        \tkzDrawLine(a,c)
	        \tkzDrawLine(t,v)
	        \tkzDrawLine(c,v)
	        \tkzDrawLine(b,u)
	    \end{tikzpicture} &
	    \begin{tikzpicture}[scale=0.8]
	        \tkzDefPoint(0,0){a}
	        \tkzDefPoint(0,1){b}
	        \tkzDefPoint(0,2){c}
	        \tkzDefPoint(1.5,0){t}
	        \tkzDefPoint(1.5,1){u}
	        \tkzDefPoint(1.5,2){v}
\tkzDrawPoints[shape=circle,fill=black,size=13](a,b)
\tkzDrawPoints[shape=circle,fill=white,size=13](t,c,u,v)
	        \tkzDrawLine(a,c)
	        \tkzDrawLine(t,v)
	        \tkzDrawLine(c,v)
	        \tkzDrawLine(b,u)
	        \tkzDrawLine(a,t)
	    \end{tikzpicture}  &
	    \begin{tikzpicture}[scale=0.8]
	        \tkzDefPoint(0,0){a}
	        \tkzDefPoint(0,1){b}
	        \tkzDefPoint(0,2){c}
	        \tkzDefPoint(1.5,0){t}
	        \tkzDefPoint(1.5,1){u}
	        \tkzDefPoint(1.5,2){v}
\tkzDrawPoints[shape=circle,fill=black,size=13](a,u)
\tkzDrawPoints[shape=circle,fill=white,size=13](b,c,t,v)
	        \tkzDrawLine(a,c)
	        \tkzDrawLine(t,v)
	        \tkzDrawLine(c,v)
	        \tkzDrawLine(b,u)
	        \tkzDrawLine(a,t)
	    \end{tikzpicture}&
	    \begin{tikzpicture}[scale=0.8]
	        \tkzDefPoint(-0.5,0){a}
	        \tkzDefPoint(0.5,1){b}
	        \tkzDefPoint(0,2){c}
	        \tkzDefPoint(1.5,0){t}
	        \tkzDefPoint(1.5,1){u}
	        \tkzDefPoint(1.5,2){v}
\tkzDrawPoints[shape=circle,fill=black,size=13](a,b)
\tkzDrawPoints[shape=circle,fill=white,size=13](t,c,u,v)
	        \tkzDrawLine(a,c)
	        \tkzDrawLine(b,c)
	        \tkzDrawLine(t,v)
	        \tkzDrawLine(c,v)
	        \tkzDrawLine(b,u)
	        \tkzDrawLine(a,t)
	    \end{tikzpicture}\\
	    \bbT_1 & \bbT_2 & \bbT_3 & \bbT_4 & \bbT_5. 
	    \end{matrix}
\]
Suppose that the generalized quadrangle $\Pi=(\mathscr{P},\mathscr{L})$ has order $(s,t)$ and the parameters of its point graph $\Gamma$ as a strongly regular graph are $(v,k,\lambda,\mu)$. 
Then we can count:
\begin{align*}
	\#(\Gamma,\bbT_1) &= t^2s(s-1),\\
	\#(\Gamma,\bbT_2) &= (t+1)t(s-1)(s-2),\\
	\#(\Gamma,\bbT_3) &= (k-s)t(s-1),  \\
	\#(\Gamma,\bbT_4) &=(t+1)t(s-1).
\end{align*}
Counting $\#(\Gamma,\bbT_5)$ is slightly more involved. Observe that  
\[
	\#(\Gamma,\bbT_5) =\mu(|\mathscr{L}|-3t-1) - \#(\Gamma,\widetilde{\bbT}_5)=(t+1)t(st+s-2)-\#(\Gamma,\widetilde{\bbT}_5),
\]
where $\widetilde\bbT_5$ is given by: 
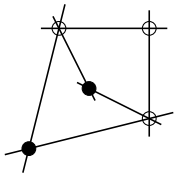
\[
	    \begin{tikzpicture}[scale=0.8]
	        \tkzDefPoint(-0.5,0){a}
	        \tkzDefPoint(0.5,1){b}
	        \tkzDefPoint(0,2){c}
	        \tkzDefPoint(1.5,0.5){t}
	        \tkzDefPoint(1.5,2){v}
\tkzDrawPoints[shape=circle,fill=black,size=13](a,b)
\tkzDrawPoints[shape=circle,fill=white,size=13](t,c,v)
	        \tkzDrawLine(a,c)
	        \tkzDrawLine(b,c)
	        \tkzDrawLine(t,v)
	        \tkzDrawLine(c,v)
	        \tkzDrawLine(b,t)
	        \tkzDrawLine(a,t)
	    \end{tikzpicture}
\]
From Theorem~\ref{pqfivevert} it follows that $\Gamma$ is $\widetilde\bbT_5$-regular. Thus it is also $\bbT_5$-regular.
\end{proof}

%\section{\texorpdfstring{$(3,7)$}{(3,7)}-regular graphs}
	Recall, that in a partial linear space, three pairwise non-collinear points are called a \emph{triad}. Moreover, a \emph{center} of a triad is a point collinear to all three points of the triad.
\begin{theorem}[P.~J.~Cameron {\cite[Theorem 2]{Cam75}}]
	Given a partial quadrangle of order $(s,t,\mu)$. Then
    \begin{equation*} 
    	\Big( s(t-1)+(\mu-1)(\mu-2) \Big)\Big( \frac{(t+1)ts^2}{\mu}-1-(t+1)s+\mu\Big)\ge \mu(t-1)^2s^2.
    \end{equation*}
    Moreover, equality holds if and only if every triad in the partial quadrangle has the same number $c$ of centers. In this case we have
    \[
    	c=1+\frac{(\mu-1)(\mu-2)}{s(t-1)}
    \]
\end{theorem}
For the special case of generalized quadrangles this simplifies to the following well-known result by Higman:
\begin{theorem}[D.~G.~Higman 1971 {\cf \cite[Corollary to Theorem 1]{Cam75}}]\label{GQqqsq} 
    Given a generalized quadrangle of order $(s,t)$. Then $s^2\ge t$. Moreover, equality holds if and only if every triad in the generalized quadrangle has the same number $(s+1)$ of centers.
\end{theorem}

\begin{corollary}[{\cite[Corollary 3]{Rei15}}]\label{gqthreeir}
    Let $\Gamma$ be the point graph of a generalized quadrangle of order $(q,q^2)$. Then $\Gamma$ is $3$-isoregular.  
\end{corollary}

\begin{proposition}\label{threeisoregpq}
    Let $\Pi$ be a partial quadrangle of order $(s,t,\mu)$, such that every triad in $\Pi$ has the same number $c$ of centers, and let $\Gamma$ be its point graph. Then $\Gamma$ is $3$-isoregular if and only if either $\Pi$ is a generalized quadrangle and $t=s^2$, or $\Gamma$ is triangle-free (\ie $s=1$). 
\end{proposition}
\begin{proof}
 ``$\Rightarrow$:''    Suppose that $\Gamma$ is $3$-isoregular. Consider the following graph type $\bbT=(\Delta,\iota,\Theta)$ of order $(3,4)$:
    
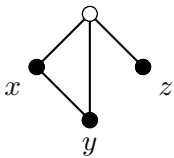
\[
    \begin{tikzpicture}[rotate=-45]
        \SetGraphUnit{1}
        \GraphInit[vstyle=Welsh]
        \renewcommand*{\VertexSmallMinSize}{2pt}
        \SetVertexMath
        \Vertex[Lpos=-135]{x}
        \EA[Lpos=-90](x){y}
        \NOEA[Lpos=-45](x){z}
        \NO[NoLabel](x){u}
        \Edges(u,x,y,u)
        \Edge(u)(z)
        \AddVertexColor{black}{x,y,z}
    \end{tikzpicture}
    \]
    Then any embedding $\kappa$ of $\Delta$ into $\Gamma$ determines a line $l$ of $\Pi$ (spanned by $\kappa(x)$ and $\kappa(y)$) and a vertex $p=\kappa(z)$ not on this line such that neither $\kappa(x)$ nor $\kappa(y)$ is collinear with $p$. In any partial quadrangle there exists at most one vertex $q$ on $l$ that is collinear with $p$ (otherwise $\Pi$ would contain a triangle of lines). So we have  $\#(\Gamma, \bbT)\in\{0,1\}$. If $\#(\Gamma,\bbT)=0$, then $\Gamma$ is triangle-free and if   $\#(\Gamma,\bbT)=1$, then $\Pi$ is a generalized quadrangle. By Theorem~\ref{GQqqsq}, we obtain that $t=s^2$. 

	``$\Leftarrow$:'' If $\Pi$ is a generalized quadrangle of order $(s,s^2)$, then $\Gamma$ is $3$-isoregular, by Corollary~\ref{gqthreeir}. So suppose that  $\Gamma$ is triangle-free. Let $u,v,w$ be three mutually distinct vertices of $\Gamma$. If the subgraph of $\Gamma$ induced by $u$, $v$, and $w$ contains an edge, then neither of the edges has a common neighbor (otherwise $\Gamma$ would contain triangles). So $u$, $v$, $w$ form a triad in $\Pi$.  Hence, they have $c$ common neighbors in $\Gamma$. Consequently, $\Gamma$ is $3$-isoregular.
\end{proof}

$3$-isoregular triangle free graphs appear to be extremely rare. The following observation was made by R.~Noda:
\begin{proposition}[{\cf \cite[Page 70]{Cam75}}]
	Let $\Gamma$ be a non-degenerate triangle-free $3$-isoregular graph in which any three pairwise non-adjacent points are joint to exactly $n$ vertices. Then $\Gamma$ is the point graph of a $\PQ((n^2+2n-1)(n+1), 1, n(n+1))$.
\end{proposition}
\begin{remark}
	The first two members of this series are the Clebsch graph ($n=1$) and the Higman-Sims graph ($n=2$). For $n=3$ the putative graph would have parameters $(v,k,\lambda,\mu)=(324,57,0,12)$. It was shown by A.~L.~Gavrilyuk and  A.~A.~Makhnev in \cite{GavMak05} that such a graph does not exist. For a very interesting account   about the history of the discovery of the Higman-Sims-graph we refer to \cite{KliWol17}.
\end{remark}

\begin{theorem}\label{threseven}
	Let $\Gamma$ be the point graph of a partial quadrangle and suppose that $\Gamma$ is $3$-isoregular. Then $\Gamma$ is $(3,7)$-regular. 
\end{theorem}
\begin{proof}
	As $\Gamma$ is $3$-isoregular, it is $(3,4)$-regular. By Corollary \ref{tvertcond}, in order to prove $(3,5)$-regularity of $\Gamma$ it suffices to prove the $\bbT$-regularity  for all graph types $\bbT$ of order $(3,5)$ for which $\Cl(\bbT)$ is $4$-connected. Since $\Gamma$ does not have $K_4-e$ as an induced subgraph, we can shorten this list by all $\bbT$ whose underlying graph contains $K_4-e$.  A computer search reveals that only the graph type type $\bbT_a$ depicted below  fulfills all these requirements:
    
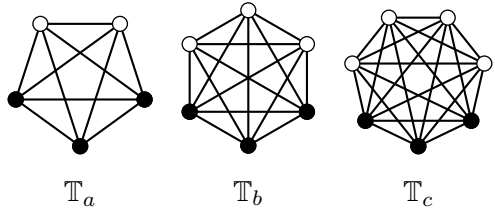
\[
    \begin{tikzpicture}[rotate=-162,scale=0.3]
        \SetGraphUnit{1}
        \GraphInit[vstyle=Welsh]
        \renewcommand*{\VertexSmallMinSize}{2pt}
        \SetVertexMath
        \Vertices[unit=3,NoLabel]{circle}{A,B,C,D,E}
        \Edges(A,B,C,D,E,A,C,E,B,D,A)
        \node[below=2ex] at (B) {\ensuremath{\bbT_{a}}};
        \AddVertexColor{black}{A,B,C}
    \end{tikzpicture}\quad
    \begin{tikzpicture}[rotate=-150,scale=0.3]
        \SetGraphUnit{1}
        \GraphInit[vstyle=Welsh]
        \renewcommand*{\VertexSmallMinSize}{2pt}
        \SetVertexMath
        \Vertices[unit=3,NoLabel]{circle}{A,B,C,D,E,F}
        \Edges(A,B,C,D,E,F,A,C,E,A)
        \Edges(B,D,F,B)
        \Edge(A)(D)
        \Edge(B)(E)
        \Edge(C)(F)
        \node[below=2ex] at (B) {\ensuremath{\bbT_{b}}};
        \AddVertexColor{black}{A,B,C}
    \end{tikzpicture}\quad
    \begin{tikzpicture}[rotate=-141.429,scale=0.3]
        \SetGraphUnit{1}
        \GraphInit[vstyle=Welsh]
        \renewcommand*{\VertexSmallMinSize}{2pt}
        \SetVertexMath
        \Vertices[unit=3,NoLabel]{circle}{A,B,C,D,E,F,G}
        \Edges(A,B,C,D,E,F,G,A,C,E,G,B,D,F,A,D,G,C,F,B,E,A)
        \node[below=2ex] at (B) {\ensuremath{\bbT_{c}}};
        \AddVertexColor{black}{A,B,C}
    \end{tikzpicture}
    \]
    However, it is easy to see that 
    \[
    	\#(\Gamma,\bbT_a)=
    	\begin{cases}
        	(s-2)(s-3) & s\ge 4\\
            0 & \text{else.}
        \end{cases}
    \]
    Thus, $\Gamma$ is $(3,5)$-regular. 
    
    With the same reasoning as before and again using a computer, we obtain that $\Gamma$ is $(3,6)$-regular if and only if it is $\bbT_b$-regular. However, it is easy to see that 
    \[
    	\#(\Gamma,\bbT_b)=
    	\begin{cases}
            (s-2)(s-3)(s-4) & s\ge 5\\
            0 & \text{else.}
        \end{cases}
    \]
    Thus, $\Gamma$ is $(3,6)$-regular. 
    
    Finally, once more using the same reasoning as above and using a computer, we obtain that  $\Gamma$ is $(3,7)$-regular if and only if it is $\bbT_c$-regular. However, it is easy to see that 
    \[
    	\#(\Gamma,\bbT_c)=
    	\begin{cases}
            (s-2)(s-3)(s-4)(s-5) & s\ge 6\\
            0 & \text{else.}
        \end{cases}\]
    Thus, $\Gamma$ is $(3,7)$-regular. 
\end{proof}
The previous theorem generalizes a result by Reichard (\cite[Theorem 2]{Rei15}) that states that the point graphs of generalized quadrangles of order $(q,q^2)$ satisfy the $7$-vertex condition. 
\begin{corollary}
	Let $\Gamma$ be the point graph of a partial quadrangle and suppose that $\Gamma$ is $3$-isoregular. Then, for every $v\in V(\Gamma)$, the second subconstituent $\Gamma_2(v)$ satisfies the $6$-vertex condition. 
\end{corollary}
\begin{proof}
	This follows from Proposition~\ref{constituents}.
\end{proof}
The previous corollary applies in particular to the point graphs of generalized quadrangles of order $(q,q^2)$. It is clear, that if $\Gamma$ is  the point graph of a generalized quadrangle of  of order $(q,q^2)$, then for every $v\in V(\Gamma)$, the graph $\Gamma_2(v)$ is the point graph of a partial quadrangle of order $(q-1,q^2,q^2-q)$. On the other hand it was shown by A.~A.~Ivanov and S.~V.~Shpectorov in \cite[Theorem A]{IvaShp91} that every partial quadrangle of order $(q-1,q^2,q^2+q)$ extends to a generalized quadrangle of order $(q,q^2)$. Consequently, we obtain.
\begin{corollary}\label{pqsixvert}
    Let $\Gamma$ be the point graph of a partial quadrangle of order $(q-1,q^2,q^2-q)$. Then $\Gamma$ satisfies the $6$-vertex condition.
\end{corollary}

\begin{example}
    There exists an infinite family of generalized quadrangles of order $(q,q^2)$ whose point graphs are non rank $3$ graphs (\cf \cite{Kan80,Kan86,Pay92}). By Theorem~\ref{threseven}, the point graph of any such generalized quadrangle is $(3,7)$-regular. The second subconstituents of these graphs give rise to a hitherto unknown family of non-rank 3 graphs satisfying the $6$-vertex condition. 
    
    The smallest actual example is the point graph $\Gamma$ of a non-classical generalized quadrangle of order $(5,25)$. Its parameters are given by $(v,k,\lambda,\mu)=(756,130,4,26)$. Its automorphism group is intransitive of rank $11$.
    
    $\Gamma$ has two non-isomorphic second subconstituents $\Gamma'$ and $\Gamma''$. Both satisfy the $6$-vertex condition and both are in turn point graphs of partial quadrangles of order  $(4,25,20)$. The automorphism group of $\Gamma'$ is intransitive of rank $52$ and the automorphism group of $\Gamma''$ is transitive of rank $5$. 
\end{example}

\begin{proposition}
	Let $\Gamma$ be the point graph of a partial quadrangle, and suppose that $\Gamma$ is $3$-isoregular. Then $\Gamma$ satisfies the $8$-vertex condition if and only if it is regular for the following graph types: 
	\[\scalebox{1}
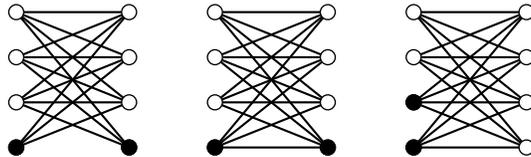
{
	    \begin{tikzpicture}[scale=0.6]
	        \SetGraphUnit{1}
	        \GraphInit[vstyle=Welsh]
	        \renewcommand*{\VertexSmallMinSize}{2pt}
	        \SetVertexMath
	        \Vertices[dir=\SO,NoLabel]{line}{a,b,c,d}
	        \Vertices[x=2.5,y=0,dir=\SO,NoLabel]{line}{t,u,v,w}
	        \Edges(a,t,b,u,c,v,d,u,a,v,b,w,c,t,d)
	        \Edge(a)(w) 
	        \AddVertexColor{black}{d,w}
	    \end{tikzpicture} \qquad 
	    \begin{tikzpicture}[scale=0.6]
	        \SetGraphUnit{1}
	        \GraphInit[vstyle=Welsh]
	        \renewcommand*{\VertexSmallMinSize}{2pt}
	        \SetVertexMath
	        \Vertices[dir=\SO,NoLabel]{line}{a,b,c,d}
	        \Vertices[x=2.5,y=0,dir=\SO,NoLabel]{line}{t,u,v,w}
	        \Edges(a,t,b,u,c,v,d,w,b,v,a,u,d,t,c,w,a) 
	        \AddVertexColor{black}{d,w}
	    \end{tikzpicture} \qquad
	    \begin{tikzpicture}[scale=0.6]
	        \SetGraphUnit{1}
	        \GraphInit[vstyle=Welsh]
	        \renewcommand*{\VertexSmallMinSize}{2pt}
	        \SetVertexMath
	        \Vertices[dir=\SO,NoLabel]{line}{a,b,c,d}
	        \Vertices[x=2.5,y=0,dir=\SO,NoLabel]{line}{t,u,v,w}
	        \Edges(a,t,b,u,c,v,d,w,b,v,a,u,d,t,c,w,a) 
	        \AddVertexColor{black}{d,c}
	    \end{tikzpicture}}
	\]
\end{proposition}
\begin{proof}
	By Theorem~\ref{threseven} we already know that $\Gamma$ is $(3,7)$-regular. Let $\mathcal{M}$ be a transversal of all isomorphism classes of graph types of order $(m,n)$, where $m\le 3$ and $n\le 7$. To show that $\Gamma$ satisfies the $8$-vertex condition means to show that it is $(2,8)$-regular. By Proposition~\ref{redlist} it suffices to show that $\Gamma$ is $\bbT$-regular for all graph types of order $(2,8)$ that are $\widehat\bbT$-irreducible, for all $\widehat\bbT\in\mathcal{M}$. We can reduce the list of graph types further by observing that they should not contain an induced subgraph isomorphic to $K_4-e$. When asking the computer for a complete list of such graph types it returns the above depicted graph types and the four graph types given below (for better visibility they are depicted as geometric configurations rather than graphs): 
	
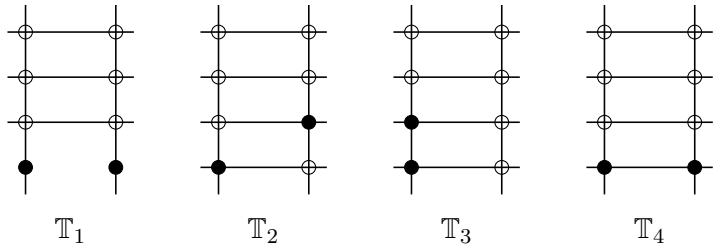
\[
	\begin{matrix}
	    \begin{tikzpicture}[scale=0.6]
	        \tkzDefPoint(0,0){a}
	        \tkzDefPoint(0,1){b}
	        \tkzDefPoint(0,2){c}
	        \tkzDefPoint(0,3){d}
	        \tkzDefPoint(2,0){t}
	        \tkzDefPoint(2,1){u}
	        \tkzDefPoint(2,2){v}
	        \tkzDefPoint(2,3){w}
\tkzDrawPoints[shape=circle,fill=black,size=13](a,t)
\tkzDrawPoints[shape=circle,fill=white,size=13](b,c,d,u,v,w)
	        \tkzDrawLine(a,d)
	        \tkzDrawLine(t,w)
	        \tkzDrawLine(d,w)
	        \tkzDrawLine(c,v)
	        \tkzDrawLine(b,u)
	    \end{tikzpicture} & 
	    \begin{tikzpicture}[scale=0.6]
	        \tkzDefPoint(0,0){a}
	        \tkzDefPoint(0,1){b}
	        \tkzDefPoint(0,2){c}
	        \tkzDefPoint(0,3){d}
	        \tkzDefPoint(2,0){t}
	        \tkzDefPoint(2,1){u}
	        \tkzDefPoint(2,2){v}
	        \tkzDefPoint(2,3){w}
\tkzDrawPoints[shape=circle,fill=black,size=13](a,u)
\tkzDrawPoints[shape=circle,fill=white,size=13](b,c,d,t,v,w)
	        \tkzDrawLine(a,d)
	        \tkzDrawLine(t,w)
	        \tkzDrawLine(d,w)
	        \tkzDrawLine(c,v)
	        \tkzDrawLine(b,u)
	        \tkzDrawLine(a,t)
	    \end{tikzpicture}&
	    \begin{tikzpicture}[scale=0.6]
	        \tkzDefPoint(0,0){a}
	        \tkzDefPoint(0,1){b}
	        \tkzDefPoint(0,2){c}
	        \tkzDefPoint(0,3){d}
	        \tkzDefPoint(2,0){t}
	        \tkzDefPoint(2,1){u}
	        \tkzDefPoint(2,2){v}
	        \tkzDefPoint(2,3){w}
\tkzDrawPoints[shape=circle,fill=black,size=13](a,b)
\tkzDrawPoints[shape=circle,fill=white,size=13](c,d,t,u,v,w)
	        \tkzDrawLine(a,d)
	        \tkzDrawLine(t,w)
	        \tkzDrawLine(d,w)
	        \tkzDrawLine(c,v)
	        \tkzDrawLine(b,u)
	        \tkzDrawLine(a,t)
	    \end{tikzpicture}&
	    \begin{tikzpicture}[scale=0.6]
	        \tkzDefPoint(0,0){a}
	        \tkzDefPoint(0,1){b}
	        \tkzDefPoint(0,2){c}
	        \tkzDefPoint(0,3){d}
	        \tkzDefPoint(2,0){t}
	        \tkzDefPoint(2,1){u}
	        \tkzDefPoint(2,2){v}
	        \tkzDefPoint(2,3){w}
\tkzDrawPoints[shape=circle,fill=black,size=13](a,t)
\tkzDrawPoints[shape=circle,fill=white,size=13](b,c,d,u,v,w)
	        \tkzDrawLine(a,d)
	        \tkzDrawLine(t,w)
	        \tkzDrawLine(d,w)
	        \tkzDrawLine(c,v)
	        \tkzDrawLine(b,u)
	        \tkzDrawLine(a,t)
	    \end{tikzpicture}\\
	    \bbT_1 & \bbT_2 & \bbT_3 & \bbT_4
	\end{matrix}
\]
	By Proposition~\ref{threeisoregpq}, $\Gamma$ is either the point graph of a generalized quadrangle of order $(q,q^2)$ or it is triangle free. Neither of the graph types $\bbT_1,\dots,\bbT_4$ is triangle free. Thus, if $\Gamma$ is triangle free then we are done. Suppose therefore that $\Gamma$ is the point graph of a generalized quadrangle $\Pi=(\mathscr{P},\mathscr{L})$ of order $(q,q^2)$, and suppose that the parameters of $\Gamma$, considered as a strongly regular graph, are $(v,k,\lambda,\mu)$.  Then we compute:
	\begin{align*}
		\#(\Gamma,\bbT_1) &= (t+1)t(s-1)(s-2)(s-3),\\
		\#(\Gamma,\bbT_2) &= (t+1)t(s-2)(s-2),\\
		\#(\Gamma,\bbT_3) &= (k-2)t(s-1)(s-2),\\
		\#(\Gamma,\bbT_4) &= t^2s(s-1)(s-2).\qedhere
	\end{align*}
\end{proof}

Let us at the end have a look on partial quadrangles $\Pi$ in which every triad has the same number $c$ of centers, but where the point graph $\Gamma$ is not necessarily $3$-isoregular. 
\begin{lemma}\label{pqthreefour}
    Let $\Pi$ be a partial quadrangle in which every triad has $c$ centers, and let $\Gamma$ be the point graph of $\Pi$. Then $\Gamma$ is regular for all graph types of order $(3,4)$, except possibly the following:
	
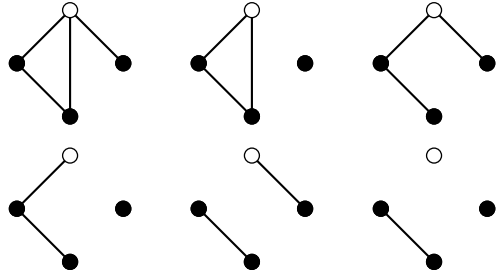
\begin{gather*}
    \begin{tikzpicture}[rotate=-45]
        \SetGraphUnit{1}
        \GraphInit[vstyle=Welsh]
        \renewcommand*{\VertexSmallMinSize}{2pt}
        \SetVertexMath
	    \SetVertexNoLabel	
        \Vertex[Lpos=-135]{x}
        \EA[Lpos=-90](x){y}
        \NOEA[Lpos=-45](x){z}
        \NO[NoLabel](x){u}
        \Edge(x)(y)
        \Edge(u)(x)
        \Edge(u)(y)
        \Edge(u)(z)
        \AddVertexColor{black}{x,y,z}
    \end{tikzpicture}\qquad
    \begin{tikzpicture}[rotate=-45]
        \SetGraphUnit{1}
        \GraphInit[vstyle=Welsh]
        \renewcommand*{\VertexSmallMinSize}{2pt}
        \SetVertexMath
	    \SetVertexNoLabel	
        \Vertex[Lpos=-135]{x}
        \EA[Lpos=-90](x){y}
        \NOEA[Lpos=-45](x){z}
        \NO[NoLabel](x){u}
        \Edge(x)(y)
        \Edge(u)(x)
        \Edge(u)(y)
        \AddVertexColor{black}{x,y,z}
    \end{tikzpicture}\qquad
    \begin{tikzpicture}[rotate=-45]
        \SetGraphUnit{1}
        \GraphInit[vstyle=Welsh]
        \renewcommand*{\VertexSmallMinSize}{2pt}
        \SetVertexMath
	    \SetVertexNoLabel	
        \Vertex[Lpos=-135]{x}
        \EA[Lpos=-90](x){y}
        \NOEA[Lpos=-45](x){z}
        \NO[NoLabel](x){u}
        \Edge(u)(x)
        \Edge(u)(z)
        \Edge(x)(y)
        \AddVertexColor{black}{x,y,z}
    \end{tikzpicture}\\
    \begin{tikzpicture}[rotate=-45]
        \SetGraphUnit{1}
        \GraphInit[vstyle=Welsh]
        \renewcommand*{\VertexSmallMinSize}{2pt}
        \SetVertexMath
        \Vertex[Lpos=-135,NoLabel]{x}
        \EA[Lpos=-90,NoLabel](x){y}
        \NOEA[Lpos=-45,NoLabel](x){z}
        \NO[NoLabel](x){u}
        \Edge(x)(y)
        \Edge(x)(u)
        \AddVertexColor{black}{x,y,z}
    \end{tikzpicture}\qquad
    \begin{tikzpicture}[rotate=-45]
        \SetGraphUnit{1}
        \GraphInit[vstyle=Welsh]
        \renewcommand*{\VertexSmallMinSize}{2pt}
        \SetVertexMath
        \Vertex[Lpos=-135,NoLabel]{x}
        \EA[Lpos=-90,NoLabel](x){y}
        \NOEA[Lpos=-45,NoLabel](x){z}
        \NO[NoLabel](x){u}
        \Edge(x)(y)
        \Edge(z)(u)
        \AddVertexColor{black}{x,y,z}
    \end{tikzpicture}\qquad
    \begin{tikzpicture}[rotate=-45]
        \SetGraphUnit{1}
        \GraphInit[vstyle=Welsh]
        \renewcommand*{\VertexSmallMinSize}{2pt}
        \SetVertexMath
	    \SetVertexNoLabel	
        \Vertex[Lpos=-135]{x}
        \EA[Lpos=-90](x){y}
        \NOEA[Lpos=-45](x){z}
        \NO[NoLabel](x){u}
        \Edge(x)(y)
        \AddVertexColor{black}{x,y,z}
    \end{tikzpicture}
    \end{gather*}
\end{lemma}
\begin{proof}
	The above list of graph types was obtained by computer analysis of the domination order of graph types of order $(3,3)$ and $(3,4)$, using the known regularities of $\Gamma$ and the type counting lemma. In this case, since the amount of work is not too high, we give a computer free proof:
	Let us first of all list all graph types of order $(3,4)$ not mentioned above:
	\[\setlength{\arraycolsep}{0.8em}
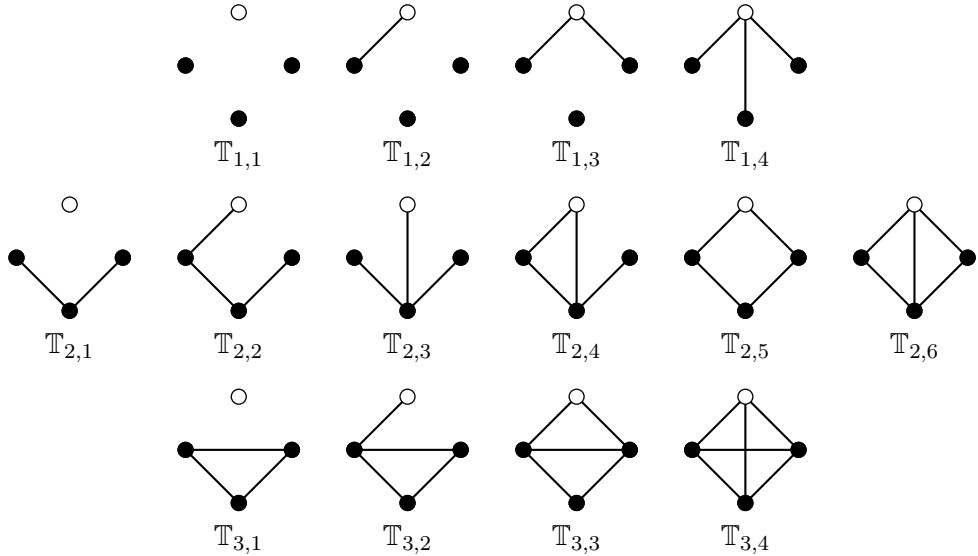
	\begin{matrix} & 
    \begin{tikzpicture}[rotate=-45]
        \SetGraphUnit{1}
        \GraphInit[vstyle=Welsh]
        \renewcommand*{\VertexSmallMinSize}{2pt}
        \SetVertexMath
	    \SetVertexNoLabel	
        \Vertex[Lpos=-135]{x}
        \EA[Lpos=-90](x){y}
        \NOEA[Lpos=-45](x){z}
        \NO[NoLabel](x){u}
        \AddVertexColor{black}{x,y,z}
    \end{tikzpicture} &
    \begin{tikzpicture}[rotate=-45]
        \SetGraphUnit{1}
        \GraphInit[vstyle=Welsh]
        \renewcommand*{\VertexSmallMinSize}{2pt}
        \SetVertexMath
	    \SetVertexNoLabel	
        \Vertex[Lpos=-135]{x}
        \EA[Lpos=-90](x){y}
        \NOEA[Lpos=-45](x){z}
        \NO[NoLabel](x){u}
        \Edge(x)(u)
        \AddVertexColor{black}{x,y,z}
    \end{tikzpicture} & 
    \begin{tikzpicture}[rotate=-45]
        \SetGraphUnit{1}
        \GraphInit[vstyle=Welsh]
        \renewcommand*{\VertexSmallMinSize}{2pt}
        \SetVertexMath
	    \SetVertexNoLabel	
        \Vertex[Lpos=-135]{x}
        \EA[Lpos=-90](x){y}
        \NOEA[Lpos=-45](x){z}
        \NO[NoLabel](x){u}
        \Edge(x)(u)
        \Edge(z)(u)
        \AddVertexColor{black}{x,y,z}
    \end{tikzpicture} & 
    \begin{tikzpicture}[rotate=-45]
        \SetGraphUnit{1}
        \GraphInit[vstyle=Welsh]
        \renewcommand*{\VertexSmallMinSize}{2pt}
        \SetVertexMath
	    \SetVertexNoLabel	
        \Vertex[Lpos=-135]{x}
        \EA[Lpos=-90](x){y}
        \NOEA[Lpos=-45](x){z}
        \NO[NoLabel](x){u}
        \Edge(x)(u)
        \Edge(z)(u)
        \Edge(y)(u)
        \AddVertexColor{black}{x,y,z}
    \end{tikzpicture}\\
    & \bbT_{1,1} & \bbT_{1,2} & \bbT_{1,3} & \bbT_{1,4}\\[2ex] 
    \begin{tikzpicture}[rotate=-45]
        \SetGraphUnit{1}
        \GraphInit[vstyle=Welsh]
        \renewcommand*{\VertexSmallMinSize}{2pt}
        \SetVertexMath
	    \SetVertexNoLabel	
        \Vertex[Lpos=-135]{x}
        \EA[Lpos=-90](x){y}
        \NOEA[Lpos=-45](x){z}
        \NO[NoLabel](x){u}
        \Edge(x)(y)
        \Edge(y)(z)
        \AddVertexColor{black}{x,y,z}
    \end{tikzpicture} &
    \begin{tikzpicture}[rotate=-45]
        \SetGraphUnit{1}
        \GraphInit[vstyle=Welsh]
        \renewcommand*{\VertexSmallMinSize}{2pt}
        \SetVertexMath
	    \SetVertexNoLabel	
        \Vertex[Lpos=-135]{x}
        \EA[Lpos=-90](x){y}
        \NOEA[Lpos=-45](x){z}
        \NO[NoLabel](x){u}
        \Edge(x)(y)
        \Edge(y)(z)
        \Edge(x)(u)
        \AddVertexColor{black}{x,y,z}
    \end{tikzpicture} &
    \begin{tikzpicture}[rotate=-45]
        \SetGraphUnit{1}
        \GraphInit[vstyle=Welsh]
        \renewcommand*{\VertexSmallMinSize}{2pt}
        \SetVertexMath
	    \SetVertexNoLabel	
        \Vertex[Lpos=-135]{x}
        \EA[Lpos=-90](x){y}
        \NOEA[Lpos=-45](x){z}
        \NO[NoLabel](x){u}
        \Edge(x)(y)
        \Edge(y)(z)
        \Edge(y)(u)
        \AddVertexColor{black}{x,y,z}
    \end{tikzpicture} &
    \begin{tikzpicture}[rotate=-45]
        \SetGraphUnit{1}
        \GraphInit[vstyle=Welsh]
        \renewcommand*{\VertexSmallMinSize}{2pt}
        \SetVertexMath
	    \SetVertexNoLabel	
        \Vertex[Lpos=-135]{x}
        \EA[Lpos=-90](x){y}
        \NOEA[Lpos=-45](x){z}
        \NO[NoLabel](x){u}
        \Edge(x)(y)
        \Edge(y)(z)
        \Edge(x)(u)
        \Edge(y)(u)
        \AddVertexColor{black}{x,y,z}
    \end{tikzpicture} &
    \begin{tikzpicture}[rotate=-45]
        \SetGraphUnit{1}
        \GraphInit[vstyle=Welsh]
        \renewcommand*{\VertexSmallMinSize}{2pt}
        \SetVertexMath
	    \SetVertexNoLabel	
        \Vertex[Lpos=-135]{x}
        \EA[Lpos=-90](x){y}
        \NOEA[Lpos=-45](x){z}
        \NO[NoLabel](x){u}
        \Edge(x)(y)
        \Edge(y)(z)
        \Edge(x)(u)
        \Edge(z)(u)
        \AddVertexColor{black}{x,y,z}
    \end{tikzpicture} &
    \begin{tikzpicture}[rotate=-45]
        \SetGraphUnit{1}
        \GraphInit[vstyle=Welsh]
        \renewcommand*{\VertexSmallMinSize}{2pt}
        \SetVertexMath
	    \SetVertexNoLabel	
        \Vertex[Lpos=-135]{x}
        \EA[Lpos=-90](x){y}
        \NOEA[Lpos=-45](x){z}
        \NO[NoLabel](x){u}
        \Edge(x)(y)
        \Edge(y)(z)
        \Edge(x)(u)
        \Edge(y)(u)
        \Edge(z)(u)
        \AddVertexColor{black}{x,y,z}
    \end{tikzpicture} \\
    \bbT_{2,1} & \bbT_{2,2} & \bbT_{2,3} & \bbT_{2,4} & \bbT_{2,5} & \bbT_{2,6}\\[2ex]
     & 
    \begin{tikzpicture}[rotate=-45]
        \SetGraphUnit{1}
        \GraphInit[vstyle=Welsh]
        \renewcommand*{\VertexSmallMinSize}{2pt}
        \SetVertexMath
	    \SetVertexNoLabel	
        \Vertex[Lpos=-135]{x}
        \EA[Lpos=-90](x){y}
        \NOEA[Lpos=-45](x){z}
        \NO[NoLabel](x){u}
        \Edge(x)(y)
        \Edge(y)(z)
        \Edge(x)(z)
        \AddVertexColor{black}{x,y,z}
    \end{tikzpicture} &
    \begin{tikzpicture}[rotate=-45]
        \SetGraphUnit{1}
        \GraphInit[vstyle=Welsh]
        \renewcommand*{\VertexSmallMinSize}{2pt}
        \SetVertexMath
	    \SetVertexNoLabel	
        \Vertex[Lpos=-135]{x}
        \EA[Lpos=-90](x){y}
        \NOEA[Lpos=-45](x){z}
        \NO[NoLabel](x){u}
        \Edge(x)(y)
        \Edge(y)(z)
        \Edge(x)(z)
        \Edge(x)(u)
        \AddVertexColor{black}{x,y,z}
    \end{tikzpicture} &
    \begin{tikzpicture}[rotate=-45]
        \SetGraphUnit{1}
        \GraphInit[vstyle=Welsh]
        \renewcommand*{\VertexSmallMinSize}{2pt}
        \SetVertexMath
	    \SetVertexNoLabel	
        \Vertex[Lpos=-135]{x}
        \EA[Lpos=-90](x){y}
        \NOEA[Lpos=-45](x){z}
        \NO[NoLabel](x){u}
        \Edge(x)(y)
        \Edge(y)(z)
        \Edge(x)(z)
        \Edge(x)(u)
        \Edge(z)(u)
        \AddVertexColor{black}{x,y,z}
    \end{tikzpicture} &
    \begin{tikzpicture}[rotate=-45]
        \SetGraphUnit{1}
        \GraphInit[vstyle=Welsh]
        \renewcommand*{\VertexSmallMinSize}{2pt}
        \SetVertexMath
	    \SetVertexNoLabel	
        \Vertex[Lpos=-135]{x}
        \EA[Lpos=-90](x){y}
        \NOEA[Lpos=-45](x){z}
        \NO[NoLabel](x){u}
        \Edge(x)(y)
        \Edge(y)(z)
        \Edge(x)(z)
        \Edge(x)(u)
        \Edge(y)(u)
        \Edge(z)(u)
        \AddVertexColor{black}{x,y,z}
    \end{tikzpicture} \\
    & \bbT_{3,1} & \bbT_{3,2} & \bbT_{3,3} & \bbT_{3,4}
	\end{matrix}
	\]
	Suppose that the parameters of $\Gamma$ as a strongly regular graph are $(v,k,\lambda,\mu)$. Then we count:
	\begin{align*}
		\#(\Gamma,\bbT_{1,4}) &= c, \\ 
		\#(\Gamma,\bbT_{1,3}) &= \lambda-\#(\Gamma,\bbT_{1,4}),\\
		\#(\Gamma,\bbT_{1,2}) &= k-2\#(\Gamma,\bbT_{1,3})-\#(\Gamma,\bbT_{1,4}),\\ 
		\#(\Gamma,\bbT_{1,1}) &= v-3\#(\Gamma,\bbT_{1,2})-3\#(\Gamma,\bbT_{1,3})-\#(\Gamma,\bbT_{1,4}),\\
		\#(\Gamma,\bbT_{2,6}) &= 0,\\
		\#(\Gamma,\bbT_{2,5}) &= \mu-1,\\
		\#(\Gamma,\bbT_{2,4}) &= \lambda,\\
		\#(\Gamma,\bbT_{2,3}) &= (k-2)-2\#(\Gamma,\bbT_{2,4}),\\
		\#(\Gamma,\bbT_{2,2}) &= (k-1) - \#(\Gamma,\bbT_{2,4})-\#(\Gamma,\bbT_{2,5}),\\
		\#(\Gamma,\bbT_{2,1}) &=v-2\#(\Gamma,\bbT_{2,2})-\#(\Gamma,\bbT_{2,3})-2\#(\Gamma,\bbT_{2,4}) -\#(\Gamma,\bbT_{2,5}).\qedhere
	\end{align*}
\end{proof}

\begin{proposition}
	Let $\Pi$ be a partial quadrangle in which every triad has $c$ centers, and let $\Gamma$ be the point graph of $\Pi$. Then $\Gamma$ satisfies the $6$-vertex condition if and only if it is regular for the following graph type:
	\[\setlength{\arraycolsep}{0.8em}
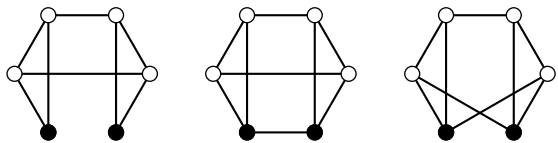
	\begin{matrix}
    \begin{tikzpicture}[rotate=0,scale=0.3]
        \SetGraphUnit{1}
        \GraphInit[vstyle=Welsh]
        \renewcommand*{\VertexSmallMinSize}{2pt}
        \SetVertexMath
	    \SetVertexNoLabel	
        \Vertices[unit=3,NoLabel]{circle}{A,B,C,D,X,Y}
        \Edges(X,D,C,X)
        \Edges(Y,A,B,Y)
        \Edge(C)(B)
        \Edge(D)(A)
        \AddVertexColor{black}{X,Y}
    \end{tikzpicture} &
    \begin{tikzpicture}[rotate=-60,scale=0.3]
        \SetGraphUnit{1}
        \GraphInit[vstyle=Welsh]
        \renewcommand*{\VertexSmallMinSize}{2pt}
        \SetVertexMath
        \Vertices[unit=3,NoLabel]{circle}{A,E,D,C,B,F}
        \Edges(F,B,E)
        \Edge(C)(D)
        \Edge(C)(F)
        \Edge(A)(E)
        \Edge(D)(E)
        \Edge(B)(C)
        \Edge(D)(A)
        \Edge(A)(F)
        \AddVertexColor{black}{F,A}
    \end{tikzpicture} &
    \begin{tikzpicture}[rotate=0,scale=0.3]
        \SetGraphUnit{1}
        \GraphInit[vstyle=Welsh]
        \renewcommand*{\VertexSmallMinSize}{2pt}
        \SetVertexMath
        \Vertices[unit=3,NoLabel]{circle}{B,C,D,A,E,F}
        \Edges(F,B,E)
        \Edge(C)(D)
        \Edge(C)(F)
        \Edge(A)(E)
        \Edge(D)(E)
        \Edge(B)(C)
        \Edge(D)(A)
        \Edge(A)(F)
        \AddVertexColor{black}{E,F}
    \end{tikzpicture}
	\end{matrix}
	\]
\end{proposition}
\begin{proof}
	The three given graph types are exactly those from Proposition~\ref{pq6vc} that are irreducible for any of the graph types of order $(3,4)$ depicted in the proof of Lemma~\ref{pqthreefour}. Now the claim follows from Proposition~\ref{redlist}.
\end{proof}

\section{Concluding remarks}
The $(m,n)$-regularity introduced in Section~\ref{regularitycond} is a very strong condition. It is in fact interesting only for $m\le 4$, because any $5$-isoregular graph is $5$-homogeneous and, in fact, homogeneous (\cite{GolKli78,Cam80,Gar76}). At first sight this appears to limit the use of the general theory of regularity conditions that is developed in this paper.  However, in principle the definitions and results from Section~\ref{regularitycond} are applicable to other categories of combinatorial objects. Finite metric spaces (possibly with integer or with rational distances), directed graphs, or semilinear spaces come to mind.

For the category of finite graphs, the most interesting are $(m,n)$-regular graphs where $m\in\{2,3,4\}$. Here the goal is to find  $(m,n)$-regular graphs that are not $m$-homogeneous and, if feasible, to classify such graphs completely, up to isomorphism. 

As was noted above, every graph $\Gamma$ is $(0,n)$-regular, because for every graph type $\bbT=(\emptyset,\iota,\Theta)$ the number $\#(\Gamma,\bbT)$ is equal to the number of embeddings of $\Theta$ into $\Gamma$. Nevertheless, counting subgraphs of a graph has been used as a global invariant for distinguishing non-isomorphic graphs. For instance, in \cite{KliKriWol14} subgraphs isomorphic to $K_4$ are counted in order to distinguish point-symmetric strongly regular graphs in three infinite families. 

We would also like to mention K.~Kov\'a\c{c}ikov\'a's Dissertation Thesis \cite{Kov15} about counting subgraphs in strongly regular graphs, where she, among other things, counts the induced subgraphs of order $\le 9$ in the putative Moore graph of valency $47$. Her methods involve the solution of huge linear systems of equations. It will be interesting to compare her approach with the one given in this paper. 

The $(2,t)$-regular graphs correspond exactly to the graphs that satisfy the $t$-vertex condition. 
There is a longstanding conjecture by M.~Klin \cite{FarKliMuz94}, that there exists a natural number $t_0$ such that for each $t\ge t_0$ all $(2,t)$-regular graphs are $2$-homogeneous (\ie they are rank $3$ graphs).  The largest $t$ for which the existence of a non rank 3,  $(2,t)$-regular graph is settled is $t=7$, due to Reichard \cite[Theorem 2]{Rei15}. Thus in Klin's conjecture we have $t_0\ge 8$.   

We should mention that the motivation to study graphs with the $t$-vertex condition comes not only from Klin's conjecture. The driving motivation to introduce the $t$-vertex condition was to distinguish the rank-$3$-graphs from other strongly regular graphs with the same parameters. In the times before the announcement of the classification of finite simple groups there was the hope to uncover in this way new sporadic finite simple groups. A typical example for the use of the $t$-vertex condition as a distinguishing invariant is \cite{Pas92}.

Up till now, $(3,t)$-regular graphs were known only for $t=4$ (apart from the $3$-homogeneous graphs). 
In this paper, the first cases of non $3$-homogeneous $(3,7)$-regular graphs are observed. Among the examples there are graphs whose automorphism group is intransitive. In view of Klin's conjecture and in view of the observation that $(3,t)$-regular graphs appear to be much  rarer than $(2,t)$-regular graphs, it seems sensible to ask whether there exists a $t_0$ such that all $(3,t)$-regular graphs with $t\ge t_0$ are $3$-homogeneous. This paper shows that if such $t_0$ exists, then $t_0\ge 8$. Note that every $(3,t)$-regular graph is $(2,t)$-regular. Thus, if Klin's conjecture turns out to be true, then this question can be answered using the classification of rank 3 graphs.

There is known only one $(4,5)$-regular graph that is not $4$-homogeneous, the McLaughlin graph on 275 vertices. A computer experiment showed that this graph is not $(4,6)$-regular. Is every $(4,6)$-regular graph $4$-homogeneous? 

\subsection*{Acknowledgements} %\longthanks 
 The present paper owes much to the numerous discussions  with  Misha Klin and Sven Reichard that we had over the years on the topic of regularity conditions of strongly regular graphs. My thanks go to Andy Woldar whose comments helped to improve a preliminary version of the paper. %Last but not least the many helpful remarks by the anonymous referees are gratefully acknowledged.    

\providecommand{\bysame}{\leavevmode\hbox to3em{\hrulefill}\thinspace}
\providecommand{\MR}{\relax\ifhmode\unskip\space\fi MR }
\providecommand{\MRhref}[2]{%
  \href{http://www.ams.org/mathscinet-getitem?mr=#1}{#2}
}
\providecommand{\href}[2]{#2}

%\bibliographystyle{amsplain-ac}
%\bibliography{tvc}

\end{document}